\documentclass[twoside,11pt]{article}

\usepackage{jmlr2e}

\usepackage[utf8]{inputenc}
\usepackage[english]{babel}

\usepackage[textwidth=8em,textsize=small]{todonotes}
\usepackage{amsmath,bm}
\usepackage{mathtools}
\usepackage{hyperref}

\hypersetup{%
  unicode=true,%
  colorlinks=true,%
  urlcolor=cyan,%
  linkcolor=red,%
  citecolor=blue,%
  filecolor=magenta%
}

\usepackage{float}
\usepackage{comment}
\usepackage{siunitx}

\usepackage{booktabs}
\usepackage{multirow} % Required for multirows
\newcommand{\1}{\mathbf 1}
\newcommand{\0}{\mathbf 0}
 
\newcommand{\spn}{\mathrm{span}}
\newcommand{\diag}{\mathrm{diag}}
\newcommand{\E}{\mathrm{E}}

\newcommand{\rank}{\mathrm{rank}}

\newcommand{\tr}{\mathrm{\,tr}}
\newcommand{\vecc}{\mathrm{vec}}
\newcommand{\vech}{\mathrm{\, vech}}
\newcommand{\unvecc}{\mathrm{unvec}}
\newcommand{\spc}{{\mathcal S}}

\newcommand{\X}{{\mathbf X}}
\newcommand{\x}{{\mathbf x}}
\newcommand{\Xbb}{\mathbb{X}}

\newcommand{\Z}{{\mathbf Z}}
\newcommand{\R}{\mathbf{R}}
\newcommand{\I}{\mathbf I}
\newcommand{\Uu}{\mathbf U}
\newcommand{\Gg}{\mathbf G}
\newcommand{\D}{\mathbf D}
\newcommand{\Cc}{\mathbf C}
\newcommand{\J}{\mathbf J}
\newcommand{\A}{\mathbf A}
\newcommand{\Ahat}{\widehat{\A}}
\newcommand{\F}{\mathbf F}
\newcommand{\f}{\mathbf f}
\newcommand{\T}{\mathbf T}
\newcommand{\V}{\mathbf V}
\newcommand{\W}{\mathbf W}
\newcommand{\w}{\mathbf w}
\newcommand{\K}{\mathbf K}
\newcommand{\B}{\mathbf B}
\newcommand{\Pb}{\mathbf P}
\newcommand{\Q}{\mathbf Q}
\newcommand{\M}{\mathbf M}

\newcommand{\bb}{\mathbf b}
\newcommand{\bbs}{\mathbf{\scriptstyle b}}
\newcommand{\abf}{\mathbf a}
\newcommand{\cb}{\mathbf c}
\newcommand{\hbf}{\mathbf h}

\newcommand{\Hbf}{\mathbf H}
\newcommand{\Hbfs}{\scriptstyle{\Hbf}}
\newcommand{\Lbf}{\mathbf L}

\newcommand{\Sigmabf}{\bm{\Sigma}}

\newcommand{\Deltabf}{\bm{\Delta}}
\newcommand{\Gammabf}{\bm{\Gamma}}
\newcommand{\Gammabfs}{\bm{{\scriptstyle \Gamma}}}
\newcommand{\Omegabf}{\bm{\Omega}}
\newcommand{\omegabf}{\bm{\omega}}
\newcommand{\Upsilonbf}{\bm{\Upsilon}}

\newcommand{\mubf}{\bm{\mu}}
\newcommand{\muxbar}{\mubf_{\textsc{\tiny $\X$}}}
\newcommand{\muhbar}{\mubf_{\textsc{\tiny $\Hbf$}}}

\newcommand{\betabf}{\bm{\beta}}
\newcommand{\alphabf}{\bm{\alpha}}
\newcommand{\alphabfhat}{\widehat{\alphabf}}
\newcommand{\alphabfs}{\bm{\scriptstyle \alpha}}
\newcommand{\taubf}{\bm{\tau}}
\newcommand{\etabf}{\bm{\eta}}
\newcommand{\thetabf}{\bm{\theta}}
\newcommand{\xibf}{\bm{\xi}}
\newcommand{\kappabf}{\bm{\kappa}}
\newcommand{\iotabf}{\bm{\iota}}

\newcommand{\varthetabf}{\bm{\vartheta}}

\newcommand{\real}{{\mathbb R}}
\newcommand{\paramall}{\Deltabf, \mubf,  \muhbar, \A, \betabf, \taubf_0, \taubf}

\DeclareMathOperator*{\argmin}{arg\,min}

\newcommand{\aVar}{\mathrm{avar}}
\newcommand{\norm}[1]{\left\lVert#1\right\rVert}

\ShortHeadings{Sufficient reductions for mixed predictors}{Bura et al.}
\firstpageno{1}

\begin{document}

\title{Sufficient reductions in regression with mixed predictors}

\author{\name Efstathia Bura (Corresponding author) \email efstathia.bura@tuwien.ac.at \\
       \addr  Institute of Statistics and Mathematical Methods in Economics\\
       Faculty of Mathematics and Geoinformation\\
       TU Wien\\
       Vienna, 1040, Austria
       \AND
       \name Liliana Forzani \email liliana.forzani@gmail.com \\
       \addr Facultad de Ingenier\'ia Qu\'imica\\
       Universidad Nacional del Litoral\\
       Researcher of CONICET\\
       Santa Fe, Argentina
       \AND
       \name Rodrigo Garc\'ia Arancibia \email r.garcia.arancibia@gmail.com \\
       \addr Instituto de Econom\'ia Aplicada Litoral-FCE-UNL\\
       Universidad Nacional del Litoral\\
       Researcher of CONICET\\
       Santa Fe, Argentina
       \AND
       \name Pamela Llop \email 
       lloppamela@gmail.com \\
       \addr Facultad de Ingenier\'ia Qu\'imica\\
       Universidad Nacional del Litoral\\
       Researcher of CONICET\\
       Santa Fe, Argentina
       \AND
       \name Diego Tomassi \email 
       diegotomassi@gmail.com \\
       \addr Facultad de Ingenier\'ia Qu\'imica\\
       Universidad Nacional del Litoral\\
       Researcher of CONICET\\
       Santa Fe, Argentina}

\editor{Genevera Allen, Kenji Fukumizu, Maya Gupta}

\maketitle

\begin{abstract}
Most data sets comprise of measurements on continuous and categorical variables. In regression and classification Statistics literature,  modeling high-dimensional mixed predictors has received limited attention. In this paper we study the general regression problem of inferring on a variable of interest based on high dimensional mixed continuous and binary predictors. The aim is to find a lower dimensional function of the mixed predictor vector that contains all the modeling information in the mixed predictors for the response, which can be either continuous or categorical. The approach we propose identifies sufficient reductions by reversing the regression and modeling the mixed predictors conditional on the response.  
We derive the maximum likelihood estimator of the sufficient reductions, asymptotic tests for dimension, and 
a regularized estimator, which simultaneously achieves variable (feature) selection and dimension reduction (feature extraction).
We study the performance of the proposed method and compare it with other approaches through simulations and  real data examples.
\end{abstract}


\section{Introduction}\label{sec:intro}

Most data sets comprise of measurements on a mixture of categorical and continuous features. Examples abound in
the biomedical and health sciences, neuro-imaging, genomics, finance, social media, and internet advertising. 
The first statistical approach to modeling the dependence structure of mixed data we found in the literature is the \textit{location model} of \cite{OlkinTate1961}. The location model uses correlation as a measure of dependence and bypasses the mixed nature of the data by grouping the continuous variables using the categorical ones and requiring they be normally distributed with different means but same variance within the groups. 

More recently, Markov Networks, or undirected graphical models, that encode pairwise conditional dependence relationships among random variables have been used to model multivariate mixed data.  With  few exceptions \citep{Yangetal2014,Yangetal2014b, Yangetal2015, Chenetal2015},  mixed continuous and categorical data are modeled with  the Gaussian Graphical Model (GGM) in a manner similar to the location model. Binary variables are used to define the different categories and GGM requires the continuous variables be conditionally normal and pairwise conditionally independent within categories. References for GMMs for  low-dimensional mixed data include  \cite{LauritzenWermuth1989}, \cite{Lauritzen1996}, \cite{YuanLin2007}, \cite{WainwrightJordan2008}, and in the high-dimensional setting, \cite{Chengetal14, Chengetal2017} and \cite{LeeHastie2015}. In particular,  \cite{Chengetal2017} proposed a simplified version of the conditional Gaussian distribution that reduces the number of parameters significantly while maintaining flexibility. 

Both GGMs and the location model are unsupervised approaches for mixed data that do not include an output of interest. In the case of a categorical output, approaches for the treatment of mixed, in particular, binary and continuous input variables, include methods based on nonparametric density estimation \citep{AitchisonAitken1976}, the use of logistic discrimination \citep{DayKerridge1967}, in
 which the probability of group membership is assumed to be a logistic function of the
 observed variates \citep{Anderson1972,Anderson1975}, and a likelihood ratio classification rule \citep{Krzanowski75} based on the  location model of \citet{OlkinTate1961}. \citet{Krzanowski1993} surveys and summarizes the associated developments. More recently, the location model has been used in multiple imputation [see, e.g., \cite{JavarasVanDyk2003}, \citet[Ch. 4, Sec. 4.4]{Buuren2018}].

In this paper we study the general regression and classification problem with high-dimensional mixed predictors. Specifically, we consider the conditional distribution of
\begin{equation}\label{modelmix}
Y\mid \X,\Hbf,
\end{equation} 
where the response $Y$ is either continuous or categorical, $\X=(X_1,X_2, \dots,X_p)^T$ is a vector of $p$ continuous, and  $\Hbf=(H_1, H_2, \dots, H_q)^T$ is a vector of $q$ binary predictor variables. 
Our aim is to find a lower dimensional function of the mixed predictor vector $\Z=(\X^T,\Hbf^T)^T$ that encapsulates \textit{all} information the mixed predictors contain for the response $Y$. Specifically, our target is the identification of a function, other than the identity, %We (a) assume there exists a function 
$\R: \real^p \times \real^q \to \real^d$ such that $F(Y \mid \Z)=F(Y \mid \R(\Z))$, % $F(Y\mid \X,\Hbf)=F(Y \mid \R(\X,\Hbf))$,
where $F(\cdot|\cdot)$ denotes the conditional cumulative distribution function of the response given the predictors. Such a function $\R$ is called a \textit{sufficient reduction} of the regression of $Y$ on $\Z$.

This seemingly ambitious goal turns out to be  surprisingly simple using the inventive tool of \textit{inverse regression}. 
When $Y$ and $\Z$ are both random, inverse regression is based on the equivalence of the following two statements [see \cite{Cook2007}, \cite{BuraDuarteForzani2016}, \cite{BuraForzani2015}], 
\begin{itemize}
    \item[(i)] $Y\mid \Z \,{\buildrel d \over =}\, Y \mid \R(\Z)$
    \item[(ii)] $\Z \mid(Y,\R(\Z)) \,{\buildrel d \over =}\, \Z \mid\R(\Z)$
\end{itemize}
 where ${\buildrel d \over =}$ signifies equal in distribution. Statement (i) is an alternative definition of a sufficient reduction for the forward regression in \eqref{modelmix} and (ii) is the usual definition of a sufficient \textit{statistic} for a \textit{parameter} $Y$ indexing the distribution of the mixed $\Z$. % [see, for example, \cite{CasellaBerger2002}].
 The equivalence of (i) and (ii) obtains that if one considers $Y$ as a \textit{parameter}, the sufficient ``statistic'' for $Y$ \textit{is the sufficient reduction} for the regression of $Y$ on $\Z$. In consequence, in order to find a sufficient reduction for \textit{the forward regression of $Y$ on $\Z$}  in \eqref{modelmix}, we can equivalently solve the \textit{inverse problem} of finding  a sufficient statistic for  the regression of $\Z$ on $Y$.

Our approach exploits the  factorization 
\begin{equation}\label{cdf.fact}
F(\X,\Hbf \mid Y)= F(\X \mid Y,\Hbf)F(\Hbf \mid Y),
\end{equation}
by allowing us to treat the continuous and binary  predictors separately, while at the same time we account for their interdependence in their relationship with $Y$ in Section \ref{model}. An advantageous aspect of  \eqref{cdf.fact} is that it requires fewer parameters in order to characterize the distributional structure of the data.  

In Section \ref{SufRed} we  model $\Hbf \mid Y$ as multivariate Bernoulli, and $\X \mid (Y,\Hbf)$ as multivariate normal,  in analogy to the Gaussian graphical model and the location model in unsupervised  multivariate analysis of mixed data. We show that the resulting distribution \eqref{cdf.fact} belongs to the exponential family, and derive  sufficient reductions for the regression $Y \mid (\X,\Hbf)$ from the two separate regressions,  $\X \mid (Y,\Hbf)$ and $ \Hbf \mid Y$ in Section \ref{SufRed}. 
We compute the maximum likelihood estimator of  the sufficient reduction in Section \ref{estimation}, its asymptotic distribution in Section \ref{asympt.distn}, and an asymptotic test for the dimension of the sufficient reduction in Section \ref{testdimension}.   
 We complete our treatment with a method for \textit{simultaneous} sufficient dimension reduction and variable selection in  Section \ref{varsel}.

Section \ref{simulation} contains an extensive simulation study that demonstrates the competitive performance of our approach. Furthermore, we show the superior performance of our methods as compared with generalized linear models and a version of principal component regression that allows for mixed predictors in the analysis of three data sets in Section \ref{realdata}. 

Even though our focus in this paper is the regression of the usually univariate $Y$ on the mixed $\Z$ vector, our development results in a new multivariate regression method for a \textit{mixed continuous and binary response}, on which we comment further as we conclude in Section \ref{discussion}.

\section{The Model}\label{model}

We start by specifying the notation we use throughout.
The $\vecc$ operator converts its  matrix argument into a column vector. More precisely, if $\Gg$ is an $m\times n$ matrix then $\vecc(\Gg)$ is an $m n\times 1$ vector obtained by stacking the columns of $\Gg$.
The $\unvecc$ operator is such that $\unvecc(\vecc(\Gg)) = \Gg$. We let $k_q = q(q-1)/2$ and $m_p = p(p+1)/2$.
The $\vech$ operator converts the lower half of a matrix including the main diagonal to a vector. That is, if $\Gg$ is a square $q\times q$ matrix then $\vech(\Gg)$ is a $m_q\times 1$ vector obtained by stacking the columns of the lower triangular  part of $\Gg$ including the diagonal.
There is a unique $\D_q \in {\mathbb R}^{q^2 \times q (q+1)/2}$  and  $\Cc_q  \in {\mathbb R}^{q (q+1)/2\times q^2}$ such that $\vecc(\Gg) = \D_q \vech(\Gg)$ and $\vech(\Gg) = \Cc_q \vecc(\Gg)$ for any $\Gg$ symmetric  $q \times q$ matrix.

{The matrix $\Lbf_q \in {\mathbb R}^{q \times q(q+1)/2}$ has entries 1 and 0, so that $\Lbf_q\Cc_q$ is equal to $\Cc_q$ but replacing the values 1/2 by zeros.   
The matrix $\J_q \in {\mathbb R}^{k_q \times q(q+1)/2}$ has entries 1 and 0, so that $\J_q\Cc_q$ is equal to $\Cc_q$ but replacing the ones by zeros. A projection onto to the columns of $b$ is denoted ${\bf P}_b $ and the projection onto the orthogonal complement of $b$ will be denotes as ${\bf Q}_b$.}

To regress $(\X, \Hbf)$ on $Y$, we model  $\X \mid (Y,\Hbf)$ and $ \Hbf \mid Y$ separately and use the factorization in \eqref{cdf.fact}.

\subsection{The distribution of $\X \mid (\Hbf,Y)$}\label{continuous}

We let the $p$-dimensional vector of continuous random variables $\X|(\Hbf, Y)$ be multivariate normal with
\begin{equation}\label{modelbase}
\X\mid \Hbf,Y \sim  N\left(\muxbar + \A (\f_Y-\bar{\f}_Y) + \betabf (\Hbf-\muhbar), \Deltabf\right),
\end{equation}
where $\muxbar=\E_{\X}(\X)$, $\muhbar=\E_{\Hbfs}(\Hbf)$, $\f_Y: \real\rightarrow \real^{r}$ is a known function of $Y$,  $\bar{\f}_Y =  \E_Y(\f_Y)$, $\A:p \times r$, and $\betabf: p \times q$, are unconstrained parameter matrices,  and $\Deltabf$ is a $p\times p$ positive definite covariance matrix.
For example, if the response is continuous, $\f_Y$ can be a vector  of polynomials of order $r$, or, in order to avoid multicollinearity, of a set of $r$ orthonormal basis functions. If the response is categorical with values in one of $h$ categories $C_k$,   $k=1,\dots,h$, we set $r=h-1$ and let the $k$-th element of $\f_Y$ to be $I(Y \in C_k)$, where $I$ is the indicator function.  To simplify notation, henceforth $\f_Y$ will signify the centered $\f_Y-\bar{\f}_Y$.

The probability density function of $\X\mid (\Hbf,Y)$ in model (\ref{modelbase}) is
\begin{align}\label{ising-vech3}
f(\X\mid \Hbf,Y=y) &= \frac{1}{\sqrt{2\pi} \sqrt{|\Deltabf|}} \exp \bigg\{ 
-\frac{1}{2} \Big((\X - \muxbar) - \A \f_y - \betabf (\Hbf-\muhbar)\Big)^T   \notag \\ & \hspace{4cm} \Deltabf^{-1} \Big((\X - \muxbar) - \A \f_y - \betabf (\Hbf-\muhbar) \Big)\bigg\}.
\end{align}

\subsection{The distribution of $\Hbf \mid Y$}\label{Binary}

The joint distribution of a random vector whose elements are binary random variables is modelled with the multivariate Bernoulli distribution [see \cite{Whittaker2009}; \cite{Dai2012}; \cite{Daietal2013}]. Its  probability mass function involves terms representing third and higher order moments of the random variables. The Ising model \citep{Ising1925} is frequently used instead  to alleviate the complexity of modeling as it includes up to second order interactions among the binary variables.  
For the multivariate binary regression $\Hbf|Y$ we use the Ising with covariates model introduced in \cite{Chengetal14}, where covariates are incorporated directly. 

Let  $\mathcal H =$ all possible combinations of $\Hbf \in \{0,1\}^q$,  $\Hbf_{-j}=(H_1,\ldots,H_{j-1},H_{j+1},\ldots,H_q)$, $\Hbf_{-i,-j}=(H_1,\ldots,$ $H_{i-1},H_{i+1},\ldots$,$H_{j-1},H_{j+1},\ldots,H_q),$ $i,j=1,\ldots,q$. The joint probability mass function of the $q$-dimensional vector of binary variables $\Hbf$ conditional on $Y$  is [see  \cite{Chengetal14}] 
\begin{align}\label{ising-vech1}
P(\Hbf\mid Y=y)&= \frac{1}{G(\Gammabf_y)}\exp\left\{\vech^T(\Hbf \Hbf^T)\vech(\Gammabf_y)\right\},
\end{align}
where 
$
G(\Gammabf_y) = \sum_{\Hbf \in \mathcal H}\exp\left(\vech^T(\Hbf \Hbf^T)\vech(\Gammabf_y)\right),
$ 
and 
 $\Gammabf_y=(\gamma_{ij}^{y})$ is a $q\times q$ symmetric matrix with elements 
\begin{align*}
& \gamma^y_{jj}= \log \left( \frac{\Pr(H_j=1\mid\Hbf_{-j}= 0,y)}{1-\Pr(H_j=1|\Hbf_{-j}= 0,y)} \right),  \\
& \\
& \gamma^y_{ij}=\log \frac{\Pr(H_i=1,H_j=1\mid 0,y)\Pr(H_i=0,H_j=0\mid \Hbf_{-i,-j}= 0,y)}{\Pr(H_i=1,H_j=0\mid \Hbf_{-i,-j}= 0,y)\Pr(H_i=0,H_j=1\mid\Hbf_{-i,-j}= 0,y)},  
\end{align*}
for $i\not=j $.

A linear model with independent variables $\f_Y \in \real^{r}$ is a natural choice for  each $\gamma^{y}_{ij}$,
\begin{equation}\label{gammadey}
\gamma^y_{ij}=\tau^{*}_{ij,0} + \taubf_{ij}^T \f_Y, \hspace{1cm} i,j=1,\ldots,q, 
\end{equation}
where $\taubf_{ij}^T=(\tau_{ij,1}, \ldots, \tau_{ij,r})$ is a vector of parameters independent of $Y$, and $\tau^{*}_{ij,0}$ is the intercept for each $(i,j)$. Here again, $\f_Y$   is also centered, and can be different from that in  \eqref{ising-vech3}, even though, as will be seen later, choosing the same $\f_Y$ simplifies the formula  for the joint distribution in \eqref{jointlik} as well as the derivation of a sufficient reduction for the regression of $Y$ on $\X,\Hbf$.

Next we define the $q\times q$ matrices, $\taubf_0$ and $\taubf_{k}$, $k=1,\ldots, r$, as $[\taubf^{*}_0]_{ij}=\tau^{*}_{ij,0}$ and $[\taubf_{k}]_{ij}=\tau_{ij,k}$   with $i,j=1,\ldots,q$ and $k=1,\ldots, r$. We let  $\taubf_0 = \vech ({\taubf}^{*}_0)$, a $q(q+1)/2$ vector, and $\taubf = \left( \vech ({\taubf}_1),\dots, \vech ({\taubf}_r) \right)$, a $q(q+1)/2 \times r$ matrix, so that the $q(q+1)/2$ vector $\vech(\Gammabf_y)$ is
\[
\vech(\Gammabf_y) = \taubf_0 + \taubf\f_y.
\]
Under \eqref{gammadey} the probability mass function of $\Hbf|Y$ in (\ref{ising-vech1}) is
\begin{equation}\label{ising-vech2}
P(\Hbf\mid Y=y) = \frac{1}{G(\Gammabf_y)}\exp\left\{\vech^T(\Hbf \Hbf^T)(\taubf_0 + \taubf\f_y)\right\},
\end{equation}
with
$
G(\Gammabf_y) = \sum_{\Hbf \in \mathcal H}\exp\left(\vech^T(\Hbf \Hbf^T)(\taubf_0 + \taubf\f_y)\right).$

Under \eqref{ising-vech2} and \eqref{ising-vech3}, the joint distribution of the inverse regression ($\X,\Hbf \mid Y$) has probability density function
\begin{align}
f(\X,\Hbf\mid Y=y) &= f(\X \mid \Hbf,Y=y)f(\Hbf \mid Y=y) \notag\\  &=\frac{1}{\sqrt{2\pi} \sqrt{|\Deltabf|}} \exp  \bigg\{-\frac{1}{2} \Big((\X - \muxbar) - \A \f_y - \betabf (\Hbf-\muhbar)\Big)^T   \notag  \\ & \hspace{4.5cm} \Deltabf^{-1} \Big((\X - \muxbar) - \A \f_y - \betabf (\Hbf-\muhbar) \Big)\bigg\} \notag \notag \\ & \hspace{1cm} \times
\frac{1}{G(\Gammabf_y)}\exp\bigg\{\vech^T(\Hbf \Hbf^T)\left(\taubf_0 + \taubf\f_y\right)\bigg\}. \label{jointlik}
\end{align}
%\begin{remark}
Our regression model for the mixed vector $\Z$ is similar to the regression model of \cite{FitzmauriceLaird1997} with the difference that we do not allow $\mubf_{\Hbfs}$ to vary with $Y$ in \eqref{ising-vech3}. This results in different maximum likelihood estimates for the parameters in \eqref{jointlik} in Section \ref{MLEs}. 
%\end{remark}

\section{Sufficient Reductions}\label{SufRed}

We focus on the regression problem \eqref{modelmix}, where we aim to identify a reduction $\R(\Z)$ such that $Y\mid \Z \,{\buildrel d \over =}\, Y \mid \R(\Z)$. Since the latter is equivalent to $\Z \mid(Y,\R(\Z)) \,{\buildrel d \over =}\, \Z \mid\R(\Z)$, as discussed in the introduction, we will derive the sufficient reduction $\R(\Z)$ using \eqref{cdf.fact}.

Of central importance to our development is showing that the density of $(\X,\Hbf) \mid Y$ in \eqref{jointlik} belongs to  the exponential family of distributions.  In %the notation defined above and 
Appendix \ref{A}, we express (\ref{jointlik}) as \begin{equation}\label{genExpFam}
f(\X,\Hbf\mid Y=y)= h(\X,\Hbf) \exp\left(\T^T(\X,\Hbf)\etabf_y    - \psi(\etabf_y )\right),
\end{equation}
which belongs to the natural exponential family of distributions [see, e.g., \cite{Morris2006}]. In \eqref{genExpFam}, 
$h(\X,\Hbf) = (2\pi)^{-1/2}$, the sufficient statistic is 
\begin{equation}\label{suffstat}
\T(\X,\Hbf)= \left( \begin{array}{c} 
\X \\ 
\Hbf\\
- \frac{1}{2} \D_p^T \D_p \vech (\X \X^T)\\
\vecc( \X \Hbf^T)\\
\J_q {\color{black} \vech } (\Hbf \Hbf^T)
\end{array}\right),
\end{equation}
the natural parameters are 
\begin{align}\label{naturalpar}
\etabf_y &=\left(\begin{array}{c}
\etabf_{y1}\\
\etabf_{y2}\\
\etabf_{y3}\\
\etabf_{y4}\\
\etabf_{y5}
\end{array}\right) = \left(\begin{array}{cccccccc}
 \I_p & \f_y^T \otimes \I_p  & 0  & 0  & 0  & 0  & 0  & 0  \\ 
 0 & 0  & \I_q &\f_y^T \otimes \I_q & 0  & 0  & 0  & 0  \\ 
 0 & 0  & 0  & 0  & \I_{m_p} & 0  & 0  & 0  \\ 
 0 & 0  & 0  & 0  & 0  & \I_{pq} & 0  & 0  \\ 
 0 & 0  & 0  & 0  & 0  & 0  & \I_{k_q} & \f_y^T \otimes \I_{k_q} 
 \end{array} \right) 
 \left(\begin{array}{c}
\varthetabf_1\\
\varthetabf_2\\
\varthetabf_3\\
\varthetabf_4\\
\varthetabf_5\\
\end{array}\right) \notag\\
&\doteq \F_y \varthetabf,
\end{align}
with $\varthetabf^T = (\varthetabf_1^T, \varthetabf_2^T, \varthetabf_3^T, \varthetabf_4^T, \varthetabf_5^T)^T$ where
\begin{align}
\varthetabf_1&=   
\begin{pmatrix} \varthetabf_{1,0}  \\ \varthetabf_{1,1}
\end{pmatrix} = \begin{pmatrix}
\Deltabf^{-1} \muxbar - \Deltabf^{-1} \betabf \mubf_{\Hbf}  \\
\vecc (\Deltabf^{-1}\A )
\end{pmatrix} : (p +pr) \times 1,\notag  \\
\varthetabf_2 &=  \begin{pmatrix} \varthetabf_{2,0}  \\ \varthetabf_{2,1}\end{pmatrix} =  \begin{pmatrix}
-\betabf^T   \Deltabf^{-1}\muxbar + \betabf^T \Deltabf^{-1} \betabf \mubf_{\Hbf}+ \Lbf_q \taubf_0 -\frac{1}{2} \Lbf_q \D_q^T \vecc (\betabf^T \Deltabf^{-1} \betabf)  \\
\vecc(\Lbf_q \taubf - \betabf^T \Deltabf^{-1} \A)
\end{pmatrix}: (q +qr) \times 1, \notag \\
{\varthetabf}_3&=  \varthetabf_{3,0} = \vech (\Deltabf^{-1}): k_p \times 1, \label{varetas} \\ {\varthetabf}_4&=  \varthetabf_{4,0} = \vecc (\Deltabf^{-1}\betabf): pq \times 1, \notag \\
\varthetabf_5 &=   \begin{pmatrix} \varthetabf_{5,0} \\ \varthetabf_{5,1}\end{pmatrix}= \begin{pmatrix}
 - \frac{1}{2} \J_q \D^T_q \vecc   (\betabf^T \Deltabf^{-1} \betabf ) + \J_q \taubf_0  \\
 \vecc (\J_{q} \taubf)
 \end{pmatrix},\notag 
\end{align} 
and
\begin{eqnarray} \label{psi}
\psi (\etabf_y) &=& - \frac{1}{2} \log |  \unvecc (\D_p {\etabf}_{y3})| + \log(G(\Gammabf_y)) + \frac{1}{2} {\etabf}_{y1}^T  (\unvecc (\D_p {\etabf}_{y3}))^{-1} {\etabf}_{y1}  \\
&\doteq & \psi_1 (\etabf_y) + \psi_2(\etabf_y) + \psi_3 (\etabf_y),  \notag 
\end{eqnarray}
with
\begin{eqnarray}\label{S}
G(\Gammabf_y)&=& \sum_H \exp \left[ \left( \J_q C_q \vecc (\Hbf \Hbf^T)\right)^T \left({\etabf}_{y5 }+   \J_q \frac{1}{2}{\D}^T_q {\vecc}(  \bar{\etabf}_{y4}^T (\unvecc (\D_p {\etabf}_{y3}))^{-1} \bar{\etabf}_{y4}) )\right.\right)\\
&&\left. + \Hbf^T \left({\etabf}_{y2}+ \bar{\etabf}_{y4}^T (\unvecc (\D_p {\etabf}_3))^{-1}{\etabf}_{y1} + \frac{1}{2}\Lbf_q {\D}_q^T \vecc( \bar{\etabf}_{y4}^T (\unvecc (\D_p {\etabf}_{y3}))^{-1}\bar{\etabf}_{y4}))\right)\right],  \notag
\end{eqnarray}
where $\bar{\etabf}_{y4}=\unvecc ({\etabf}_{y4})$. 

For any matrix $\V$, let $\spc_{\V}$ denote the span of the columns of $\V$; that is,  $\spc_{\V}= \spn(\V)$. Theorem \ref{propo1} obtains the sufficient reduction for the  regression of $Y$ on $(\X, \Hbf)$ using a result from \cite{BuraDuarteForzani2016}.

\begin{theorem}\label{propo1}
Suppose that $(\X,\Hbf)\mid Y$ has  density given by (\ref{genExpFam}). The minimal sufficient reduction for the regression $Y\mid (\X,\Hbf)$ is
\begin{align}
\R(\X,\Hbf)&=\alphabf_{\mathbf{a}}^T\left(\T(\X,\Hbf)-\E(\T(\X,\Hbf))\right), \label{reduction1}
\end{align}
where  $\T(\X,\Hbf)$ is given by (\ref{suffstat}) and $\alphabf_{\mathbf{a}}$ is a basis for $\spc_{\alphabf_{\mathbf{a}}}=\spn\{\etabf_Y - \E({\etabf_Y}), Y\in \mathcal{Y}\}$,  with $\etabf_Y$ 
given in \eqref{naturalpar}. 
\end{theorem}

We provide the proof of Theorem \ref{propo1} in Appendix \ref{proofpropo1}, where we see that the reduction in \eqref{reduction1} is characterized by the coefficients of the basis  for $\spn\{\etabf_Y - \E({\etabf_Y}), Y\in \mathcal{Y}\} = \spn (\abf)$  with
\[ 
\mathbf{a}= \left( \begin{array}{c} \Deltabf^{-1} \A\\
\Lbf_q \taubf - \betabf^T \Deltabf^{-1} \A\\
0\\
0\\
\J_q \taubf \end{array}\right)=\left( \begin{array}{c} \unvecc (\varthetabf_{1,1})\\
\unvecc (\varthetabf_{2,1})\\
0\\
0\\
\unvecc (\varthetabf_{5,1}) \end{array}\right).
\]
Since $\etabf_{3}\doteq \etabf_{y3}$ and $\textcolor{red}{\etabf_4}\doteq \etabf_{y4}$ do not depend on $y$, Corollary \ref{optimalSDR} follows. 

\begin{corollary}\label{optimalSDR} %{\bf Optimal SDR}
Suppose the density of $(\X,\Hbf)|Y$ is given by \eqref{genExpFam}. A minimal sufficient dimension reduction for the regression of $Y$ on $(\X, \Hbf)$ is given by 
\begin{align}
\R(\X,\Hbf)=\alphabf_{\bb}^T\left({\mathbf t}\left(\X,\Hbf\right)-\E\left({\mathbf t}(\X,\Hbf)\right)\right), \label{optSDR}
\end{align}
where 
\begin{eqnarray}\label{pame}
    {\mathbf t}(\X,\Hbf)&=&\left(\X^T, \Hbf^T, \left[\J_q \vech (\Hbf \Hbf^T)\right]^T\right)^T, \label{tXH}
\end{eqnarray}
and $\alphabf_{\bb}$ is a basis for $\mathcal{S}_{\alphabf_{\tiny \bb}}=\spn\{\bb\}$ 
with
\begin{eqnarray}\label{bb}
\bb &= & \left( \begin{array}{c} \Deltabf^{-1} \A\\
\Lbf_q \taubf - \betabf^T \Deltabf^{-1} \A\\
\J_q \taubf \end{array}\right) =
  \left( \begin{array}{c} \unvecc (\varthetabf_{1,1}) \\
\unvecc (\varthetabf_{2,1})\\
\unvecc (\varthetabf_{5,1}) \end{array}\right). \label{b}    
\end{eqnarray}
\end{corollary}

As the reduction in \eqref{optSDR} is not only sufficient but also \textit{minimal}, we call it {\bf optimal SDR} in the sequel.

\begin{corollary}\label{CoroPFC} When the predictor vector contains only continuous variables; that is, $q=0$ and $\Z=\X$, the sufficient dimension reduction is
\begin{align}
\R(\X)=\alphabf_1^T{\color{black}\left(\X-\E(\X)\right)}, \label{PFC}
\end{align}
where $  \spc_{\alphabf_1}=\spn (\alphabf_1) = \spn (\Deltabf^{-1} \A)$, and  $\A: p \times r$ in \eqref{modelbase}.
\end{corollary}

The reduction \eqref{PFC} coincides with {\it {Principal Fitted Components}} (PFC) in \cite{CookForzani2008}.

\begin{corollary}\label{binonly}  When the predictor vector contains only binary variables; that is, $p=0$ and $\Z=\Hbf$,   the sufficient dimension reduction is
\begin{align}\label{binSR}
\R(\X)&=\alphabf_2^T ({\mathbf s}( \Hbf) - \E ({\mathbf s}( \Hbf))),
\end{align}
where
\begin{equation}\label{nonec}
{\mathbf s}( \Hbf) =  \left(\Hbf^T, \left[\J_q  \vech  (\Hbf \Hbf^T)\right]^T\right)^T,
\end{equation}
and 
\begin{align}
  \spc_{\alphabfs_2}&=\spn(\alphabf_2) = \spn  \left( \begin{array}{c} 
\Lbf_q \taubf \\
\J_q \taubf \end{array}\right). \label{a_2}  
\end{align}
\end{corollary}

When the predictors are mixed, we derive a \textit{sufficient but not minimal} reduction in Corollary \eqref{suboptSDR}, which we call {\bf sub-optimal  SDR}. 

\begin{corollary} \label{suboptSDR} 
Suppose that $(\X,\Hbf)|Y$ has  density  (\ref{genExpFam}). A sufficient  reduction for the regression of $Y $ on $(\X, \Hbf)$ is given by
\begin{align}
\R(\X,\Hbf)&=\alphabf_{\mathbf{c}}^T \left( \w(\X,\Hbf)-\E( \w (\X,\Hbf))\right), \label{suboptSR}
\end{align}
with 
\begin{align}
\textcolor{black}{\w(\X,\Hbf) = (\X^T, \Hbf^T,     [\vech   (\Hbf \Hbf^T)]^T)^T}, \label{wXH}    
\end{align}

\begin{align}\label{Spanc} 
 \mathcal{S}_{\alphabf_{\cb}}&= \spn(\alphabf_{\cb})=%\spn\left(\cb \right) =
 \spn \left(\begin{array}{cc} \mathbf{c}_1 & \mathbf{0}\\
\mathbf{0} & \mathbf{c}_2
\end{array}  \right) =  \spn \left(\begin{array}{cc} \Deltabf^{-1}\A & 0\\
-\betabf^T \Deltabf^{-1} \A & 0  \\
0 &   \taubf
\end{array}  \right) .  
\end{align}
\end{corollary}

If  $\rank(\cb_1) =d_1 \leq  \min \{r,p\}$ and $\rank(\cb_2) =d_2\leq \min\{r, q(q+1)/2\}$, then 
\begin{align}
{\cb}_1=\left(\begin{array}{c} {\Deltabf}^{-1}{\A}  \\
-\betabf^T \Deltabf^{-1} \A   
\end{array}  \right) = \left(\begin{array}{c} {\alphabf \xibf}  \\
-\betabf^T  \alphabf \xibf   
\end{array}  \right), \;  \cb_2  =  
   \kappabf \iotabf,  \label{c1defn}
\end{align} 
where $\A=\alphabf\xibf$, $\alphabf \in \real^{p\times d_1}$, $\xibf \in \real^{d_1\times r}$, $\kappabf \in \real^{q(q+1)/2 \times d_2}$ and $\iotabf \in \real^{d_2 \times r}$ are full rank matrices. Therefore,
\begin{align}
\spn(\cb_1)&=\spn\{(\alphabf^T, -\alphabf^T\betabf)^T\}, \label{Spanc1} \\ \spn(\cb_2)&=\spn ( \kappabf ). \label{Spanc2}
\end{align}
  
In Table~\ref{Tab:SuffReds}, we summarize the results of this Section and tabulate the sufficient reductions for mixed normal and binary predictors. 

\begin{table}
\caption{\label{Tab:SuffReds}Sufficient Reductions  in Regressions with Mixed Predictors.}
%\ra{1.3}
%\resizebox{\columnwidth}{!}{%
\centering
\fbox{%
{\scriptsize
\begin{tabular}{llr} \\ 
%\toprule
& \multicolumn{2}{c}{Sufficient Reductions} \\
%\cmidrule(r){2-3}
%\hline
Predictor Distribution    & Optimal SDR & Sub-optimal \\
\midrule
% \cmidrule(r){2-3}
 %\hline
 \multirow{3}{*}{$(\X,\Hbf)\mid Y$ with density \eqref{genExpFam}}    & $\alphabf_{\bb}^T ({\mathbf t}(\X,\Hbf)-\E({\mathbf t}(\X,\Hbf)))$
 & \;\;\;$\alphabf_{\mathbf{c}}^T ({\mathbf w}(\X,\Hbf)-\E({\mathbf w}(\X,\Hbf)))$\\
 & ${\mathbf t}(\X,\Hbf)$ in \eqref{tXH}, &  \;\;\; ${\mathbf w}(\X,\Hbf)$ in \eqref{wXH},  \\
 & $\mathcal{S}_{\alphabf_{\tiny \bb}}=\spn\{\bb \}$, $\bb$ in \eqref{b} & \;\;\;$\mathcal{S}_{\alphabf_{\cb}}$ in \eqref{Spanc}, \eqref{Spanc1}, and \eqref{Spanc2} \\
 %\cmidrule(r){2-3}
 \hline
\multirow{2}{*}{$\X\mid Y \sim  N\left(\muxbar + \A (\f_Y-\bar{\f}_Y), \Deltabf\right)$}  & $\alphabf_1^T (\X- \E (\X))$     &    \\
& $\spc_{\alphabf_1} = \spn (\Deltabf^{-1} \A)$ & \\ %\cmidrule(r){2-3}
\hline
\multirow{2}{*}{ $\Hbf\mid Y$ with mass function \eqref{ising-vech2}}     & $\alphabf_2^T  ({\mathbf s}(\Hbf)- \E({\mathbf s}(\Hbf))) $&\\
&  ${\mathbf s}(\Hbf)$ in \eqref{nonec} &\\
&  $\spc_{\alphabf_2}$ in \eqref{a_2} & \\
%\bottomrule
\end{tabular}
}
}
\end{table}

\section{Reduction Estimators and their Asymptotic Distribution}\label{estimation}

In this section we derive maximum likelihood estimators for our optimal and sub-optimal sufficient reductions, the asymptotic normality of the projection matrix of the optimal SDR, with which we also obtain asymptotic tests of dimension of both optimal and sub-optimal reductions.  

\subsection{Parameter Estimation via Maximum Likelihood }\label{MLEs}

We assume a random sample $(y_i,\x_i,\hbf_i)$, $i=1,\ldots,n$, is drawn from the joint distribution of $(Y,\X,\Hbf)$ and that the conditional distribution models (\ref{ising-vech1}) and (\ref{modelbase}) hold.
Finding the maximum likelihood estimators of the reductions derived in Section~\ref{SufRed} requires {first} the estimation of the parameters $\Deltabf, \mubf, \mubf_{\Hbfs}, \A, \betabf, \taubf_0, \taubf$, in the joint density \eqref{jointlik} with log-likelihood  %\eqref{loglik} which we express as
\begin{equation}\label{max1}
 \sum_{i=1}^n \log f_{\X,\Hbf} (\x_i,\hbf_i|y_i; \paramall).
\end{equation}
We maximize (\ref{max1}) %using \eqref{jointlik} 
in two steps. First, we maximize %\begin{equation}\notag
$\sum_{i=1}^n \log f_{\X} (\x_i|y_i,\hbf_i; \Omegabf)$
%\end{equation} 
to estimate the parameters $\Omegabf = \{\Deltabf, \mubf,\mubf_{\Hbfs}, \A, \betabf \}$. 
Since $\X\mid(\Hbf,Y)$ follows a normal distribution, the maximum likelihood estimator (MLE) of $\Omegabf$ is obtained from fitting a multivariate normal linear model of $\X$ on the centered $(\Hbf,Y)$ via MLE. The MLE  of $\A$ and $\betabf$ are $(\widehat{\A}, \widehat{\betabf})=\mathbb{X}^T\mathbb{L}\left(\mathbb{L}^T\mathbb{L}\right)^{-1}$, where $\Xbb$ denotes the $n\times p$ matrix with rows $(\x_i -\bar{\x})^T$, and $\mathbb{L}$  the $n \times (r+q)$ matrix with rows $\left((\f_{y_i} -\bar{\f}_y)^T, (\hbf_i-\bar{\hbf})^T\right)$,  $\bar{\x}=\sum_{i=1}^n \x_{i}/n$, $\bar{\f}_y=\sum_{i=1}^n f_{y_i}/n$ and $\bar{\hbf}=\sum_{i=1}^n \hbf_{i}/n$. The MLE of the covariance matrix is $\widehat{\Deltabf}= \left(\Xbb^T - (\widehat{\A}, \widehat{\betabf})\mathbb{L}^T\right) \left(\Xbb^T - (\widehat{\A}, \widehat{\betabf})\mathbb{L}^T\right)^T/n$.

Next, we estimate $\Upsilonbf  = (\taubf_0,\taubf)$  maximizing the conditional log-likelihood function
\begin{equation}\notag
 \sum_{i=1}^n  \log f_{\Hbf} (\hbf_i \mid y_i;\Upsilonbf).
 \end{equation}
Using  parametrization (\ref{gammadey}), the joint probability mass function \eqref{ising-vech1} can be written as 
\begin{align*}
 &P(\Hbf|Y=y)=\exp \bigg( \sum_{j=1}^q \tau^{*}_{jj0}H_j + \sum_{j=1}^q \taubf^T_{jj}\f_y H_j \\
&\hspace{3cm} + \sum_{1\leq j < j' \leq q} \tau^{*}_{jj'0}H_jH_{j'} + \sum_{1\leq j < j' \leq q} \taubf^T_{jj'}\f_y H_jH_{j'}\bigg)\frac{1}{G(\Gammabf_y)}.
\end{align*}
Following \cite{Chengetal14}, we consider a single binary variable $H_j$ and condition over the rest $\Hbf_{-j}= (H_1,\dots, H_{j-1},H_{j+1},\dots,H_q)$ to obtain 
\begin{equation}\label{condlogit}
\log \frac{P(H_j=1\mid \Hbf_{-j},Y)}{P(H_j=0 \mid \Hbf_{-j},Y)}=\tau^{*}_{jj0}+ \taubf^T_{jj}\f_y + \sum_{j\neq j'}\tau^{*}_{jj'0}H_{j'}+\sum_{j< j'}\taubf^T_{jj'}\f_y H_{j'}.
\end{equation}
Thus, the conditional log-odds for a specific binary variable $H_j$ is linear in the parameters. Moreover, the conditional maximum likelihood estimators for these parameters can be obtained by fitting a logistic regression of $H_j$ on \textcolor{black}{$(\f_y,\Hbf_{-j}, \f_y \Hbf_{-j})$}, so that we obtain estimators for  $\taubf_0$ and $\taubf$ by fitting $q$ univariate logistic regressions. In particular, for the sample points $(\hbf_i^T,y_i) = (h_{i1},\ldots, h_{iq},y_i)$, for each binary variable $j$ ($j=1,\ldots,q$), the conditional log-likelihood function is
\begin{equation}\label{verologit}
\ell_j(\taubf_0,\taubf; \hbf_i,y_i)=\frac{1}{n}\sum_{i=1}^{n} \log P(h_{ij}\mid \hbf_{i,-j},y_i)=\frac{1}{n}\sum_{i=1}^{n} \left(h_{ij}\epsilon_{ij} - \log(1+\exp(\epsilon_{ij}))\right),
\end{equation}
where $\hbf_{i,-j}=(h_{11},\ldots, h_{i,j-1}, h_{i, j+1}, \ldots, h_{iq})$ and 
\[
\epsilon_{ij}=\log \frac{P(h_{ij}=1 \mid \hbf_{i,-j},y)}{P(h_{ij}=0 \mid \hbf_{i,-j},y_i)}=\tau^{*}_{jj0}+ \taubf^T_{jj}\f_{y_i} + \sum_{j\neq j'}\tau^{*}_{jj'0}h_{ij'}+\sum_{j \neq j'}\taubf^T_{jj'}\f_{y_i} h_{ij'}.
\]
To estimate $\Upsilonbf$ we use the joint estimation algorithm proposed by \cite{Chengetal14} that maximizes $\sum_{j} \ell_j(\taubf_0,\taubf; \hbf_i,y_i)$.

\subsection{Maximum Likelihood Estimation of the Reductions}

To estimate the {\bf optimal SDR} %(\ref{optSDR}) 
$\alphabf_{\bb}$ in Corollary \ref{optimalSDR} and the {\bf sub-optimal SDR}  $\alphabf_{\cb}$ in Corollary \ref{suboptSDR}, we need first to estimate $\bb$ in (\ref{bb}) and $ \mathbf{c}_1$ and $\mathbf{c}_2$ in (\ref{Spanc}). We use the ML estimators  $(\widehat{\Deltabf}, \widehat{\mubf},  \widehat{\mubf}_{\Hbfs}, \widehat{\A}, \widehat{\betabf}, \widehat{\taubf}_0, \widehat{\taubf})$ of the corresponding parameters in \eqref{max1} from Section \ref{MLEs}.

%%%%%%%%%%%%%%%%%%%%%%%%%%%%%%%%%%%%%%%%%%%%%%%%%%%%%%%%%%%%

\subsubsection{Optimal SDR}\label{redmin}

To estimate the minimal sufficient reduction in \eqref{optSDR}, or equivalently, derive a basis estimate of  $\spc_{\alphabfs_{\bbs}}$, we need to first estimate $\bb $ in \eqref{bb}. If $d=\dim(\mathcal{S}_{ \alphabf_{\tiny \bb}})$, with $d\leq \min\{r,p+q(q+1)/2\}$, the rank of $\bb$ is also $d$ with singular value decomposition 
\begin{equation}\label{nec}
{\bb}=\Uu^{T}\left(\begin{array}{cc} {\mathbf{K}} & \mathbf{0}\\
\mathbf{0} & \mathbf{0} \end{array}\right)\R,
\end{equation}
where $k_1 \geq \ldots \geq  k_d > 0$ are the singular values of $\bb$, $\K=\diag(k_1, \ldots,k_d)$,  $\Uu^{T}= (\Uu_{1},
\Uu_{0})$  is a $m \times m$ orthogonal matrix with $m=p+q(q+1)/2$, $\Uu_1: m\times d$, $\Uu_0:
m\times (m-d)$, and $\R^T=(\R_{1},
\R_{0}) $ is an  $r \times r$ orthogonal matrix with $\R_1: r\times d$, $\R_0:
r\times (r-d)$. The submatrices satisfy $\Uu_{1}\Uu_{1}^T+ \Uu_{0}\Uu_{0}^T= \I_{m}$,
$\Uu_{1}^T\Uu_{1}= \I_d$, $\Uu_{0}^T\Uu_{0}=\I_{m-d}$,
$\Uu_{0}^T\Uu_{1}=\mathbf{0}$, $\R_{1}\R_{1}^T+
\R_{0}\R_{0}^T= \I_{r}$, $\R_{1}^T\R_{1}= \I_d$,
$\R_{0}^T\R_{0}=\I_{r-d}$, $\R_{0}^T\R_{1}=\mathbf{0}$.
Then,
\begin{equation}\label{bbb}
\bb = \Uu_1  \mathbf{K} \R^{T}_1,
\end{equation}
and, as a consequence, $\alphabf_{\bb}$ in Corollary \ref{optimalSDR} can be set to $\Uu_1$. Plugging in  the ML estimators $(\widehat{\Deltabf}, \widehat{\A}, \widehat{\betabf}, \widehat{\taubf})$ we  obtain that the ML estimator of $\bb$ is 
\begin{align}\label{mleb}
\widehat{\bb} &= \begin{pmatrix} \widehat{\Deltabf}^{-1} \widehat{\A}\\
\Lbf_q \widehat{\taubf} - \widehat{\betabf}^T \widehat{\Deltabf}^{-1} \widehat{\A} \\
\J_q \widehat{\taubf} \end{pmatrix} =
  \begin{pmatrix} \unvecc (\widehat{\varthetabf}_{1,1}) \\
\unvecc (\widehat{\varthetabf}_{2,1})\\
\unvecc (\widehat{\varthetabf}_{5,1}) \end{pmatrix}.     
\end{align}
The singular value decomposition of the MLE of $\bb$ is
\begin{equation}\label{bbhat}
\widehat{\bb}=\widehat{\Uu}^T\left(\begin{array}{cc} \widehat{\mathbf{K}}_1 & \mathbf{0}\\
\mathbf{0} & \widehat{\mathbf{K}}_0 \end{array}\right)\widehat{\mathbf{R}},
\end{equation}
where  $\widehat{\mathbf{K}}_1 = \text{diag}(\widehat{k}_1,\ldots, \widehat{k}_{d})$,
$\widehat{\mathbf{K}}_0=\diag(\widehat{k}_{d+1},\ldots, \widehat{k}_{\min(m,r)})$, %of dimension $(m-d) \times (r-d)$, 
$\widehat{k}_i$ are the singular values of $\widehat{\bb}$ in decreasing order, $\widehat{\Uu}$ is an $m\times m$ orthogonal matrix whose columns  are  the left singular vectors of $\widehat{\bb}$, and  $\widehat{\mathbf{R}}$ is an $r\times r$ orthogonal matrix, whose columns are the right-singular vectors of $\widehat{\bb}$. Let $\widehat{\Uu}_1$ be the first $d$ columns of $\widehat{\Uu}$,  $\widehat{\mathbf{R}}_1$ the first $d$ columns of $\widehat{\mathbf{R}}^T$, and  $\widehat{\B}=\widehat{\K}_1\widehat{\R}_1^T$. An estimator of $\bb$ subject to $d=\dim(\mathcal{S}_{ \alphabf_{\tiny \bb}})$ is
\begin{equation}\label{bestimator}
\widehat{\bb}^{(d)}=
\widehat{\Uu}_1\widehat{\K}_1\widehat{\R}_1^T =\widehat{\Uu}_1\widehat{\B}.
\end{equation}
and an estimator of the reduction $\alphabf_{\bb}$ in Corollary \ref{optimalSDR} is
\begin{equation}\label{bound}
\widehat{\alphabf}_{\bb}= \widehat{\mathbf{U}}_1.
\end{equation}

\subsubsection{Sub-optimal SDR: $\widehat{\mathcal{S}}_{\mathbf{c}}$}\label{sub}

To obtain an estimator for the space  $\mathcal{S}_{\alphabf_{\cb}}$ in \eqref{Spanc} that gives the sub-optimal sufficient reduction  \eqref{suboptSR},  we set $\mathbf{c} = (\cb_1,\cb_2)$, where $\cb_1$  and $\cb_2$ are given in \eqref{c1defn}, with $\hbox{rank}(\cb_1) =d_1$ and $\hbox{rank}(\cb_2) =d_2$. Plugging in the MLE  $(\widehat{\Deltabf}, \widehat{\mubf},  \widehat{\mubf}_{\Hbfs}, \widehat{\A}, \widehat{\betabf}, \widehat{\taubf}_0, \widehat{\taubf})$ of the corresponding parameters in \eqref{max1} from Section \ref{MLEs}, we obtain estimators of $\cb_1$ and $\cb_2$,
\[ \widehat{\cb}_1=\left(\begin{array}{c} \widehat{\Deltabf}^{-1}\Ahat  \\
-\widehat{\betabf}^T \widehat{\Deltabf}^{-1} \Ahat   
\end{array}  \right), \quad \widehat{\cb}_2 =  \widehat{\tau}.
\]
We then consider their respective 
SVD decompositions as in Section \ref{redmin}. 
Let $\widehat{\Uu}_{c_1}$ denote the first $d_1$ left eigenvectors of 
$\widehat{\cb}_1$ and $\widehat{\Uu}_{c_2}$ the first $d_2$ left eigenvectors of 
$\widehat{\cb}_2$. Then, an estimator for the {\bf sub-optimal} sufficient reduction in \eqref{suboptSR} is defined as
\[ \widehat{\alphabf}_{\cb}= \left(\begin{array}{cc}
\widehat{\Uu}_{\cb_1} & 0\\
0 & \widehat{\Uu}_{\cb_2}\end{array}\right).\]

\subsection{Asymptotic distribution of the optimal sufficient reduction estimator}\label{asympt.distn}

In this section we derive the asymptotic distribution of the projection onto the column space of the estimated optimal sufficient reduction $\widehat{\alphabf}_{\bb}$ in (\ref{bound}),  $\Pb_{\alphabfhat_{\mathbf b}}=\alphabfhat_{\mathbf b} (\alphabfhat_{\mathbf b}^T \alphabfhat_{\mathbf b})^{-1} \alphabfhat_{\mathbf b}^T$.  We use this result in the derivation of the asymptotic tests for dimension in Section \ref{testdimension} and for inference about the sufficient dimension reduction. 

\begin{proposition}\label{AsymptDistn} 
Suppose that $(\X,\Hbf)|Y$ has probability mass function  (\ref{genExpFam}) with the natural parameters  $\etabf_Y$ satisfying (\ref{naturalpar}) and that $\bb$ has rank $d$. Then, 
\[
\sqrt{n} \; \vecc\left(\Pb_{\alphabfhat_{\mathbf b}} -\Pb_{\alphabf_{\mathbf b}} \right) \xrightarrow{\mathcal{D}}  \mathcal{N} \left(\mathbf{0}, \V_{\alphabfhat_{\mathbf b}}\right),
\]
with 
\begin{equation}\label{VarAsintoticad}
\V_{\alphabfhat_{\bb}}=(\I_{m^2} \otimes {\mathbf{K}}_{mm}) \left( \bb^- \otimes  
\Q_{\bb}\right)^T   \V_{rcl} \left( \bb^- \otimes  
\Q_{\bb}\right) (\I_{m^2} \otimes {\mathbf{K}}_{mm}),
\end{equation}
where  $\bb^-$ is the Moore-Penrose generalized inverse of $\bb$, 
\begin{equation}\label{vrcl}
\V_{rcl}= \W \M \V \M^T \W^T,
\end{equation}
with 
\begin{equation}\label{Vinv}
\V^{-1} = \E \left( {\mathbf F}_y^T \mathbf{J}{\mathbf F}_y \right), 
\end{equation}
$\F_y$ is defined in (\ref{naturalpar}), $\mathbf{J}$ is the matrix of partial derivatives  given by 
\begin{equation}\label{deriva}
\mathbf{J} = \frac{\partial^2 \psi (\etabf_y)} {\partial \etabf_y \partial \etabf_y^T},
\end{equation} 
\begin{align}
\M&= 
\begin{pmatrix}
\mathbf 0_{pr\times p}&  \mathbf I_{pr} & \mathbf 0_{pr \times q} & \mathbf 0_{pr \times qr}  & \mathbf 0_{pr \times m_p}&\mathbf 0_{pr \times qp} &  \mathbf 0_{pr \times k_q}  & \mathbf 0_{pr \times rk_q}  \\
\mathbf 0_{qr\times p}&  \mathbf 0_{qr \times pr} & \mathbf 0_{qr \times q} & \mathbf I_{qr}  & \mathbf 0_{qr \times m_p} &\mathbf 0_{qr \times qp}  &  \mathbf 0_{qr \times k_q}  & \mathbf 0_{qr \times rk_q}   \\
\mathbf 0_{rk_q\times p} &  \mathbf 0_{rk_q \times pr} & \mathbf 0_{rk_q\times q} & \mathbf 0_{rk_q \times qr}  & \mathbf 0_{rk_q \times m_p} &  \mathbf 0_{rk_q \times qp} & \mathbf 0_{rk_q \times k_q}  & \mathbf I_{rk_q} \end{pmatrix} ,\label{laM}   
\end{align}
and 
\begin{equation}\label{W}
\W=\left(\begin{array}{c} 
\mathbf I_r \otimes \left(\begin{array}{l}\mathbf{I}_p \\ \mathbf{0}_{q\times p}\\ \mathbf{0}_{k_q \times p } \end{array}\right),
\mathbf I_r \otimes \left(\begin{array}{l} \mathbf{0}_{p \times q}\\ \mathbf{I}_{q}\\ \mathbf{0}_{ k_q \times q }\end{array}\right),
\mathbf I_r \otimes  \left(\begin{array}{l}\mathbf{0}_{p \times k_q} \\ \mathbf{0}_{q \times k_q}\\ \mathbf{I}_{k_q} \end{array} \right)
\end{array}\right).
\end{equation}
\end{proposition}

\subsection{Tests for dimension}\label{testdimension}

We propose  two asymptotic tests for the dimension of the sufficient reduction in {\it optimal SDR}.  
We adapt these tests for the case of \textit{sub-optimal SDR},  to estimate the dimension of the continuous predictors separately from the binary predictors. 

The dimension of the sufficient reduction is the rank of $\bb$ in \eqref{bb}. 
We estimate the rank $d$ of $\bb$ by  sequentially
testing the hypotheses 
\begin{equation}\label{test}
H_0: \rank (\bb) = j \quad \text{vs.} \quad H_1: \rank (\bb) > j,
\end{equation}
for $j=0, 1,\ldots, \min(r,m)$, where $m=p+q(q+1)/2$.
For a fixed level $\alpha$, the estimated rank is the smallest value of $j$ for which the null is not rejected. 

\citet{BuraYang2011} proposed asymptotic tests for the rank of random matrices in sequential hypothesis testing. To construct the corresponding tests for dimension, 
we consider the singular value decomposition  of $\bb$ in \eqref{bbb} and $\widehat{\bb}$ in \eqref{bestimator} with $d=j$.

The first statistic we use to test (\ref{test}) is
$\Lambda_1(j)=n\sum_{i=j+1}^{\min(m,r)}\hat{k}_i^2,$
where $\widehat{k}_i$'s are the  singular values of $\widehat{\bb}$ in descending order.  
Proposition \ref{AsymptDistn} obtains the asymptotic normality of  $\widehat{\bb}$ with  covariance ${\V}_{rlc}$ in \eqref{vrcl}.  
When $\hbox{rank} (\bb) = j$,
\begin{equation}\label{wchisq-stat}
\Lambda_1(j)\overset{\mathcal{D}}\longrightarrow \sum_{i=1}^{s}\omega_i X^2_i,
\end{equation}
where $s=\min(\rank (\V_{rcl}), (r-j)(m-j))$, $X_i^2$ are independent chi-squared random variables with 1 degree of freedom, and the weights are the descending eigenvalues of  $\Q=(\R_0^T \otimes \Uu_0^T) \V_{rcl} (\R_0 \otimes \Uu_0)$  [see \citet{BuraYang2011}]. In practice, the weights $\omega_i$, $i=1,\ldots,s$  are replaced by $\widehat\omega_1 \ge \widehat\omega_2 \ge \ldots \ge \widehat\omega_s$, the descending eigenvalues of 
\begin{align}\label{Qhat}
\widehat{\Q}&=(\widehat{\R}_0^T \otimes \widehat{\Uu}_0^T) \widehat{\V}_{rcl} (\widehat{\R}_0 \otimes \widehat{\Uu}_0),
\end{align}
where $\widehat{\V}_{rcl}$ is a consistent estimate of $\V_{rcl}$.
This test rejects $H_0$ if $\Lambda_1(j) > q_{\alpha}$, where $q_{\alpha}$ is the $(1-\alpha)$
percentile of the distribution of $\sum_{i=1}^{s}\widehat{\omega}_i X^2_i$.
We estimate $q_{\alpha}$ from  the  empirical distribution function of $\Lambda_1$, by generating $10000$ realizations of $\sum_{i=1}^{s}\widehat{\omega}_i X_{i}^2$ and computing the empirical quantile $\widehat{q}_{\alpha}$.

The second is a Wald test with test statistic,
%\begin{equation*}
$\Lambda_2(j)=n \vecc(\widehat{\mathbf{K}}_0)^T\widehat{\Q}^{\dag}\vecc(\widehat{\mathbf{K}}_0)$,
%\end{equation*}
where $\widehat{\mathbf{K}}_0$ is defined in (\ref{bbhat}) and $\widehat{\Q}^{\dag}$ is the Moore-Penrose inverse of $\widehat{\Q}$  in \eqref{Qhat}.

Following \cite{BuraYang2011},  since $\widehat{\bb}$ is asymptotically normal, if $j=\rank (\bb)$, 
then 
$\Lambda_2(j)\overset{\mathcal{D}}\longrightarrow \chi^2(s)$,
where the degrees of freedom are $s=\min(\rank (\V_{rcl}), (r-j)(m-j))$.
The rejection region is  $\Lambda_2(j)
> \chi^2_{\alpha}(s)$, where $\chi^2_{\alpha}(s)$ is the
$(1-\alpha)$ percentile of the $\chi^2(s)$ distribution.

\section{Variable selection}\label{varsel}

Identifying variables that are not associated with the outcome is important for both interpretation and for improving the predictive power of a classifier or  a regression model. We propose a method to simultaneously obtain the sufficient reduction and carry out variable selection by removing redundant variables from the reduction. 
This is obtained jointly with the estimate of the reduction by introducing structured regularization on a matrix factorization problem.

In particular, we exploit the factorization of the full rank maximum likelihood estimate $\widehat{\bb}$ into a relevant full-rank factor $\Cc \in {\mathbb R}^{p+q(q+1)/2 \times d}$, which determines the reduction, and a matrix $\B$ that is immaterial.

The building block of the procedure proposed here is to note that the reduced rank estimator $\widehat{\bb}^{(d)}= \widehat{\Uu}_1\widehat{\B} $ in \eqref{bestimator} 
is also the solution to the least squares minimization  problem  
\begin{equation}\label{cuadratic}
 \min_{\Cc \in {\mathbb R}^{p+q(q+1)/2 \times d}, \Cc^T\Cc=\I, \B \in {\mathbb R}^{d \times r}} (\vecc (\widehat{\bb}) - \vecc (\Cc\B))^T (\vecc (\widehat{\bb}) - \vecc (\Cc\B)),
\end{equation}
where $\widehat{\bb}$ is the maximum likelihood estimator of $\bb$. 
The solution can be expressed as $\widehat{\Cc}= \widehat{\Uu}_1\V$, for some orthogonal matrix $\V\in {\mathbb R}^{d\times d}$, so that $\spn(\widehat{\Cc})=\spn(\widehat{\Uu}_1)$. 

All sufficient  reductions in Section~\ref{redmin} are of the form 
$\R(\X,\Hbf) 
= \Uu_{1}^T  ({\mathbf t} (\X,\Hbf)-\E({\mathbf t}(\X,\Hbf))$. If $t_j$ is the
$j$th component of ${\mathbf t}(\X,\Hbf)$, and  $t_j$ is not associated with $Y$, the $j$th row 
of $\Uu_1$ is  zero.  
Therefore,  identifying predictors that are conditionally independent of $Y$ corresponds to identifying the rows of $\Cc$ that contain only  $0$. This can be achieved using mixed-norm regularizers that are known to induce structured sparsity in the estimates \citep{bach2012}. 

The proposed procedure is as follows. For a fixed $d$, once we obtain  $\widehat{\bb}^{(d)}= \widehat{\Uu}_1 \widehat{\B}$ in \eqref{bestimator}, we solve 
\begin{equation} \label{eq:penquadratic}
 \argmin_{\Cc \in {\mathbb R}^{(p+q(q+1)/2)\times d}, \Cc^T\Cc=\I }\left(\vecc (\widehat{\bb}) - \vecc (\Cc \widehat{\B})\right)^T \left(\vecc (\widehat{\bb}) - \vecc (\Cc \widehat{\B})\right) + \lambda\Omega(\Cc),
\end{equation}
where $\Omega(\Cc)$ is a mixed-norm regularizer which penalizes the rows of $\Cc$ in a similar manner to group-lasso. The specific form of $\Omega(\Cc)$ depends on the type of predictor variables involved in the problem, as follows.

\begin{itemize}
\item[(a)] When all predictors are continuous (normal), we  use the penalty $\Omega(\Cc) = \sum_{j=1}^p || {\Cc}_j ||_2$, with $\Cc_j$ the $j$th row of $\Cc$. In this case the sufficient reduction contains no interaction terms and each row of $\Cc$ affects a single element of $\X$. Hence, by shrinking the $j$th row of $\Cc$ to $0$, the computed reduction becomes insensitive to the measured value of $X_j$. 
When all predictors are continuous, under the assumed model the optimization problem is indeed fairly similar to group lasso \citep{yuan2006} as can be seen after rewriting (\ref{eq:penquadratic}) as
\begin{equation*}
	\argmin_{\Cc \in {\mathbb R}^{(p+q(q+1)/2)\times d}, \Cc^T\Cc=\I } \lVert \vecc (\widehat{\bb}) - (\widehat{\B}^T\otimes\I)\vecc(\Cc) \rVert_2^2 + \lambda\sum_{j=1}^p \lVert\Cc_j\rVert_2.
\end{equation*}  

\item[(b)] When all predictors are binary, the sufficient reduction includes interaction effects $H_iH_j$. Thus, to discard the effect of a given binary variable, say $H_j$, we need to set all the entries in $\Cc$ related to $H_j$ to zero. For a reduction of dimension $d$, there are $d$  such entries related to the main effects and $d(q-1)$ related to the interaction terms. 
The grouping of the entries of $\Cc$ does not form a partition, since the entries affecting the interaction terms appear twice. For instance, assume for simplicity that $d=1$. Parameter $\theta_{13}$ operates on variables $H_1$ and $H_3$ and then it enters the regularizer in groups $\{\eta_1,\theta_{12},\theta_{13},\dots,\theta_{1q}\}$ and $\{\eta_3,\theta_{13},\theta_{23},\dots,\theta_{q3}\}$. Both groups of parameters overlap at $\theta_{13}$. Thus, the regularizer inducing the desired sparsity structure is a mixed-norm regularizer with overlapping groups, $\Omega(\Cc)=\sum_{g \in \mathcal{G}}\lVert \Cc_{g}\rVert_2$. Here, $g  \subset \{1,\dots,dq(q+1)/2\}$ indicates the subset of entries that affect the binary variable $H_i$ and $\mathcal{G}$ is the collection of such groups. Moreover, each binary variable is associated with two groups, one derived from the main effects and one from the interaction terms, since they typically have rather different scales. The obtained regularized problem can be solved using algorithms for overlapping group lasso, as proposed, for example, in \cite{Liuetal2010}.

\item[(c)] When the predictors are mixed normal and binary, we  combine the regularizers described in (a) and (b) in a single penalty $\Omega(\Cc) = \gamma \sum_{j=1}^{p} \lVert {\Cc}_j \rVert_2  + (1-\gamma) \sum_{g \in \mathcal{G}} \lVert \Cc_{G_i}\rVert_2$. The value of $\gamma$ serves as a tuning  weight for the amount of regularization in the continuous and binary parts, respectively. In sub-optimal SDR,  we carry out variable selection  separately for the continuous and binary variables as described in (a) and (b).

\end{itemize}

Selection of hyperparameters $(\lambda,\gamma)$ is done using 10-fold cross validation, with prediction error as the optimization criterion.
The procedure starts by estimating a maximum value $\lambda_m$ so that the whole estimate vanishes for any $\lambda>\lambda_m$. We then set a grid of $n_\lambda$ candidate values for $\lambda$, uniformly spaced on a logarithmic scale between $0$ and $\lambda_m$.  We typically use $n_\lambda=100$. For $\gamma$ we test 11 values uniformly spaced in $[0,1]$.
In each fold, an initial full-rank estimate of the reduction is computed using the training set and then factorized using  truncated SVD to give $\widehat{\mathbf{B}}$ and an initial estimate for $\Cc$. Problem (\ref{eq:penquadratic}) is  solved for each pair of candidate values $(\lambda_k,\gamma_k)$. The obtained reduction is applied to both the training and the test sample. With the reduced training set we fit a prediction model and then we evaluate the prediction error on the reduced test sample.
The average prediction error over the ten cross-validation folds is then computed for each candidate pair $(\lambda_k,\gamma_k)$. We pick the combination that attains the smallest mean prediction error.

\section{Simulation Studies}\label{simulation}

We assess the performance of the proposed methods in estimating the sufficient reduction and its dimension, out-of-sample prediction, and variable selection   in simulations. 

In all our simulations the response is generated from the uniform distribution on the integers $\{1,\ldots, r+1\}$, with $r=5$, and set $\f_{y}=I(y=j)-n_j/n$, where $I$ is the indicator function,  $n$ denotes the total sample size and  $n_j$ the number of observations in category $j$ for $j=1,\ldots, r$.  All reported results are based on sample sizes $n=100, 200, 300, 500, 750$, and 100 repetitions.

\subsection{Estimation, prediction and dimension tests}

We assess the accuracy of estimating  $\spn(\alphabf)$ with $\spn(\widehat{\alphabf})$ using  $||\Pb_{\alphabfs}- \Pb_{\widehat{\alphabfs}}||_2$ [see \cite{YeLim2016}].
The prediction error is computed as $||\Pb_{\alphabfs^T (\X_N, \Hbf_N)}- \Pb_{\widehat{\alphabfs}^T (\X_N, \Hbf_N)}||_2$, where $(\X_N,\Hbf_N)$ is a new sample of size $N=2000$ that is independent of the training sample.  
We estimate the sufficient  reduction using the true $d$. 
 
\subsubsection{Continuous predictors}\label{cont.sim}

We generate $p$-variate continuous predictors as $\X \mid Y=y \sim\mathcal{N}( \mubf_y, \Deltabf) $ with $\mubf_y = \A \f_y$ for  $\A = \Deltabf \alphabf \xibf$,  where $\alphabf \in \real^{p \times d}$ of  $\rank(\alphabf)=d$  and $\xibf \in \real^{d \times r}$.  We let $p=20$ and $\mathbf{0}_l$, $\1_l$ denote the $l$-vectors of zeros and ones, respectively. 
\begin{itemize}
\item[(a)] For $d=1$, we set $\xibf = \1_{r}^T$ % (1,\stackrel{r-1}{\stackrel{\smile}{\dots}}, 1)^T$,   
$\alphabf =(\mathbf{0}_{p/2}^T,\1_{p/2}^T)^T$,
$\Deltabf=  5(\I_p+\rho \alphabf\alphabf^T) $ with  $\rho=0.55$.
   
\item[(b)] For $d=2$, we set
\[\xibf =\begin{pmatrix} 
1 & 1 & 1 & 1 & 1\\
0 & 0  & 0  & 1 & 1
\end{pmatrix},
\]
and $ \alphabf =(\alphabf_1, \alphabf_2)$
be an orthonormal basis of $\spn((\mathbf{0}_{p/2}^T, \1_{p/2}^T)^T,(\mathbf{0}_{p/2}^T, \1_{p/4}^T, -\1_{p/4}^T)^T)$,
    $\Deltabf=  5(\I_p+\rho_1\alphabf_1\alphabf_1^T + \rho_2\alphabf_2\alphabf_2^T )$ for $\rho_1=0.55$ and $\rho_2=0.25$.
    
\end{itemize}
  
 \subsubsection{Binary predictors}\label{bin.sim}
 
 We  generate $q=10$ binary predictors  assuming that $\Hbf\mid Y$ follows an Ising model with parameters $\{\taubf_0,\taubf\}$, where $\taubf = [\vech ({\taubf}_1),\dots, \vech ({\taubf}_{r})]$, $\taubf_j$ are $q\times q$ matrices and set $\taubf_0=\mathbf{0}$.  %  and $d=1,2$ with following values
 
 \begin{itemize}
\item[(a)] For $d=1$, and $j=1, \dots, r$, $\taubf_j = 3\times{\mathbf K_1}/\sqrt{\sum_{ij}([\mathbf K_1]_{ij})}$ with \[{\mathbf K_1} = \left(
\begin{array}{cccccc|c}
1 & 30 & 5 & 0  & 0&0&\0_{1 \times 4}\\
30 & 1 & 10 & 0& 0&0& \vdots\\
5 & 10& 1 & 30 & 0& 0&\vdots \\
0 & 0  & 30 & 1 & 30 & 0& \vdots\\
0 & 0&0& 30 &1 & 30 & \vdots\\
0&0&0&0&30 &1 & \vdots\\
0&0&0&0&0 &30 & \0_{1 \times 4}\\
\cline{1-6} 
\0_{3 \times 1} & \cdots & \cdots &\cdots &\cdots &\multicolumn{1}{c}{\0_{3 \times 1}} & \0_{3 \times 4}\\
 \end{array} \right).\] 

\item[(b)] For $d=2$, $\taubf_j=3\times\mathbf{K}_1/\sqrt{\sum_{ij}([\mathbf K_1]_{ij})}$, 
for $j=1,3,4,5$, and 
\[\taubf_2=\frac{12}{\sqrt{6}} \times \left(
\begin{array}{ll}
{\mathbf I}_{6} & {\mathbf 0}_{6 \times 4}\\
{\mathbf 0}_{4 \times 6} & {\mathbf 0}_{4 \times 4}
\end{array}\right)
. \]

\end{itemize}

\subsubsection{Mixed predictors}

 \begin{itemize}
\item[(a)] For $d=1$,  we use the same parameters as for continuous $\X\mid (\Hbf,Y)$ and binary variables $\Hbf \mid Y$ in Sections \ref{cont.sim}, \ref{bin.sim}, respectively. Moreover, we set $\mubf_H =0$ and 
%\rodri{\[ 
$\betabf = \left(  \1_{p\times 6}/10,\0_{p \times 4}\right) \in \real^{p \times q}$. %}
to induce sparsity in the binary predictors. 
%\rodri{(Note: we change $\beta$ to have sparcity in binaries predictors)}

\item[(b)] For $d=2$, we  generate $\Hbf \mid Y$ as in Section \ref{bin.sim} with dimension 1 and $\X \mid (\Hbf,Y) $ as in (a) with dimension 2. 
%The final dimension is $d=2$. 
\end{itemize}

\begin{figure}
    \centering
    \includegraphics[width=7cm, height=6.9cm]{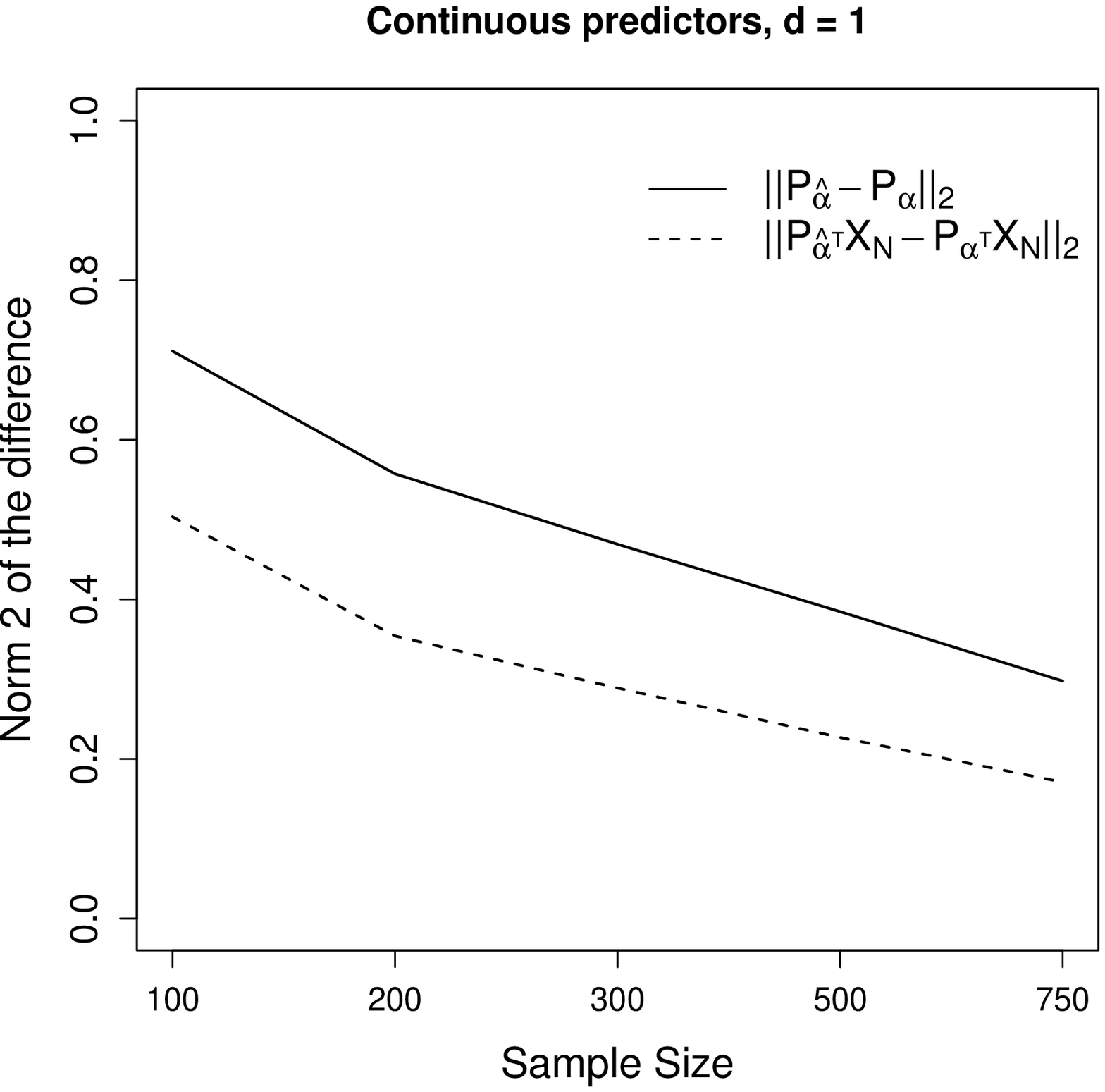}
    \includegraphics[width=7cm, height=6.9cm]{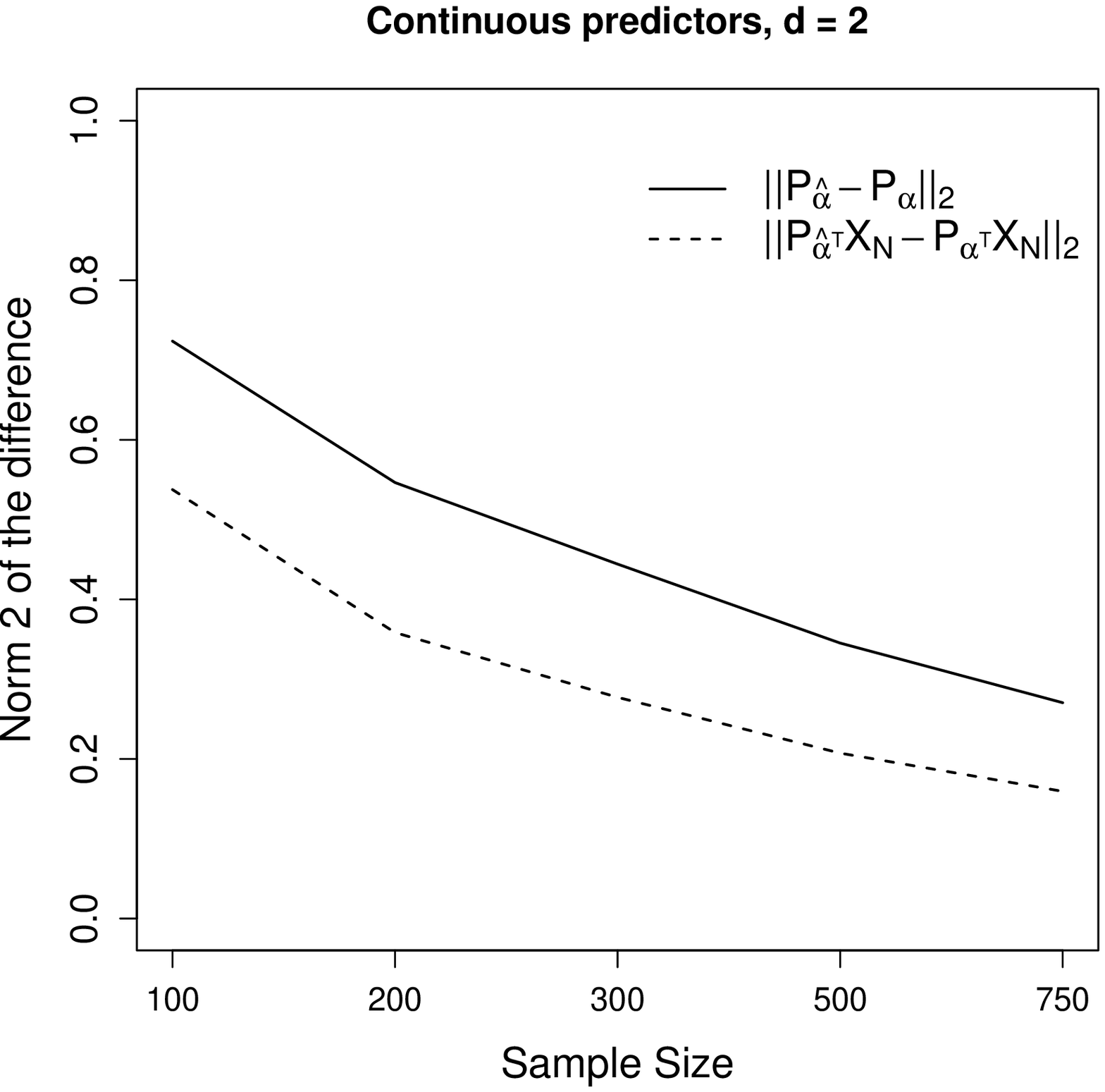}\\
    %%%%%%%%
    \includegraphics[width=7cm, height=6.9cm]{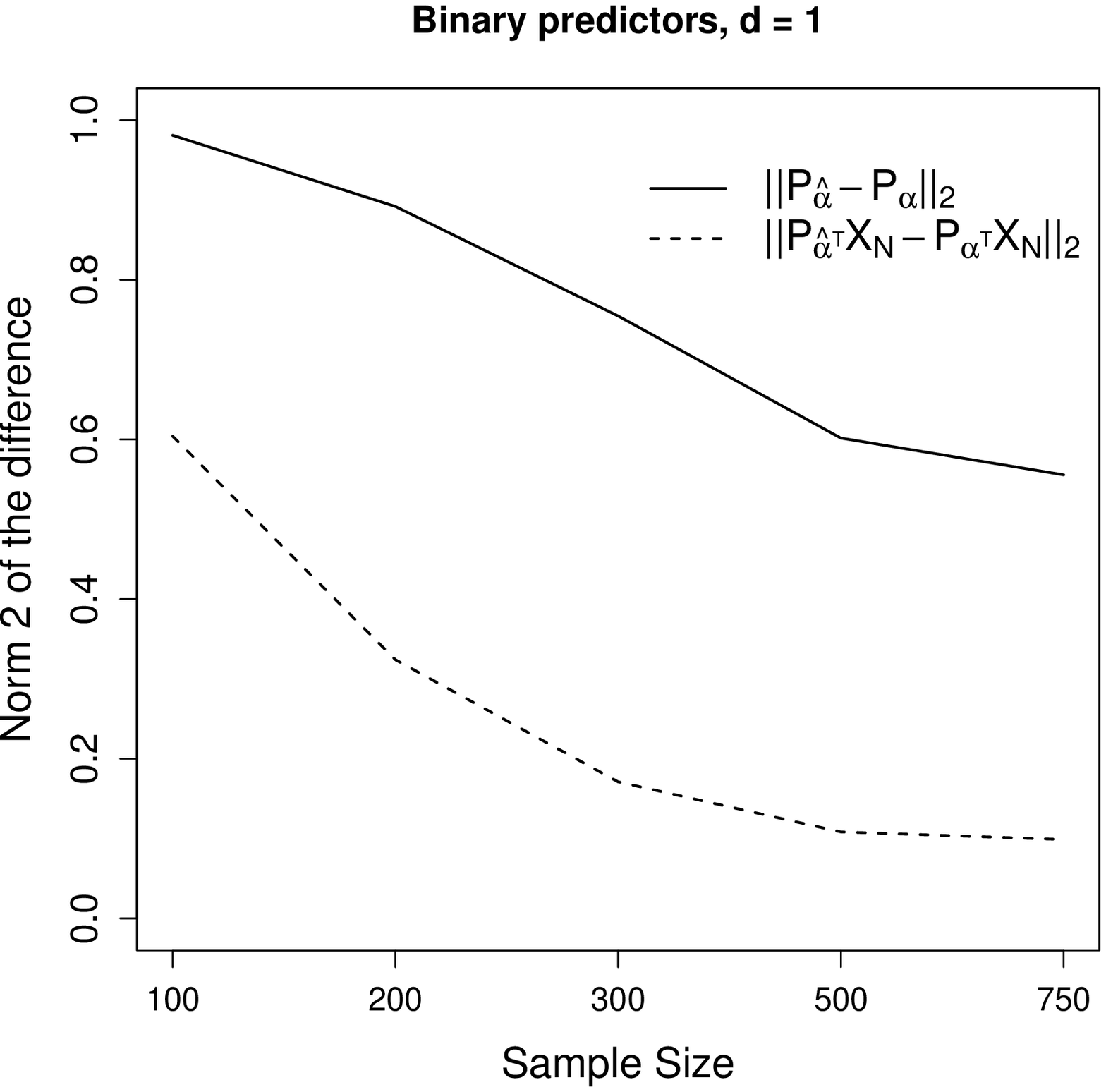}
    \includegraphics[width=7cm, height=6.9cm]{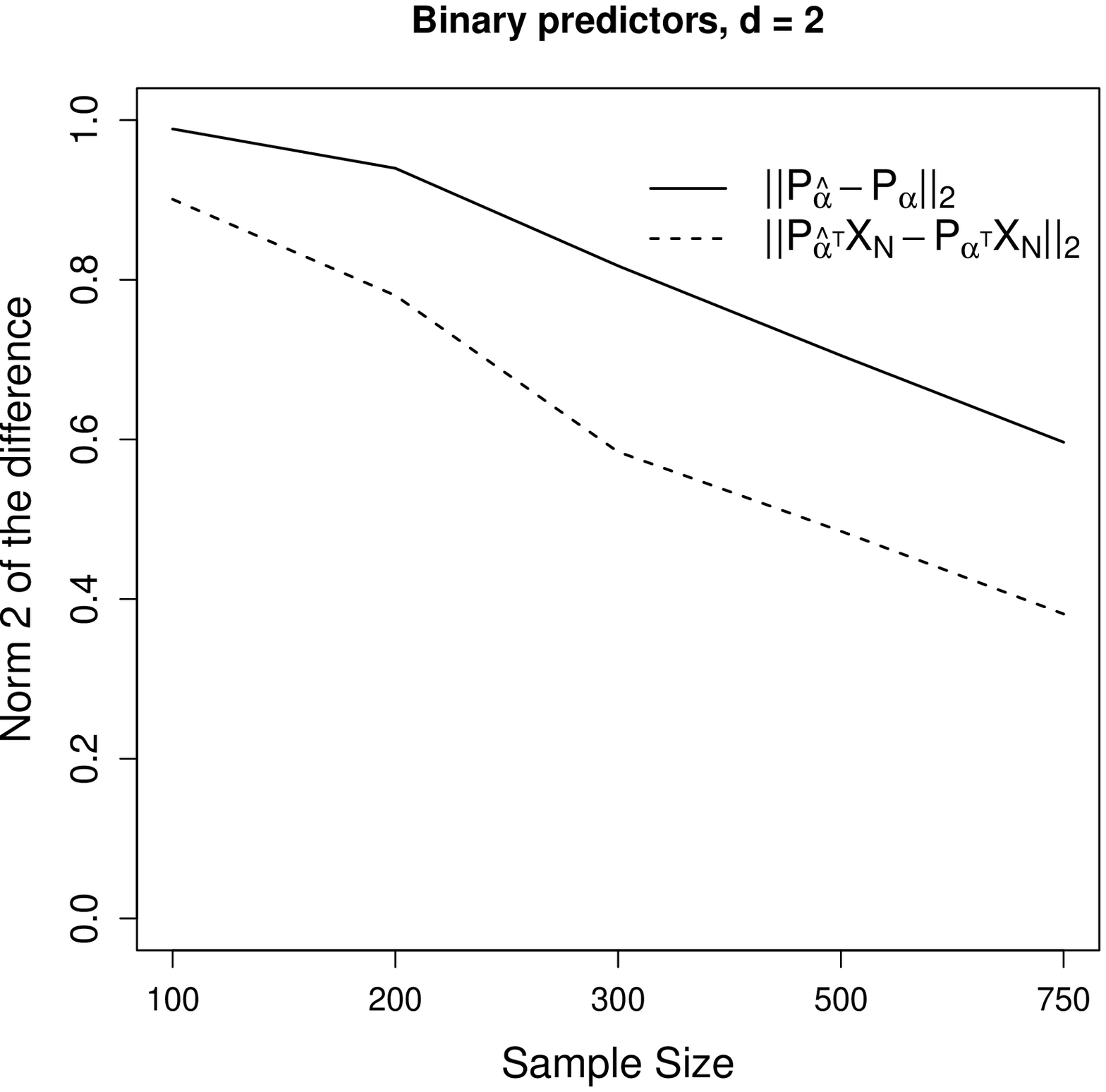}\\
    %%%%%%%%
    \includegraphics[width=7cm, height=6.9cm]{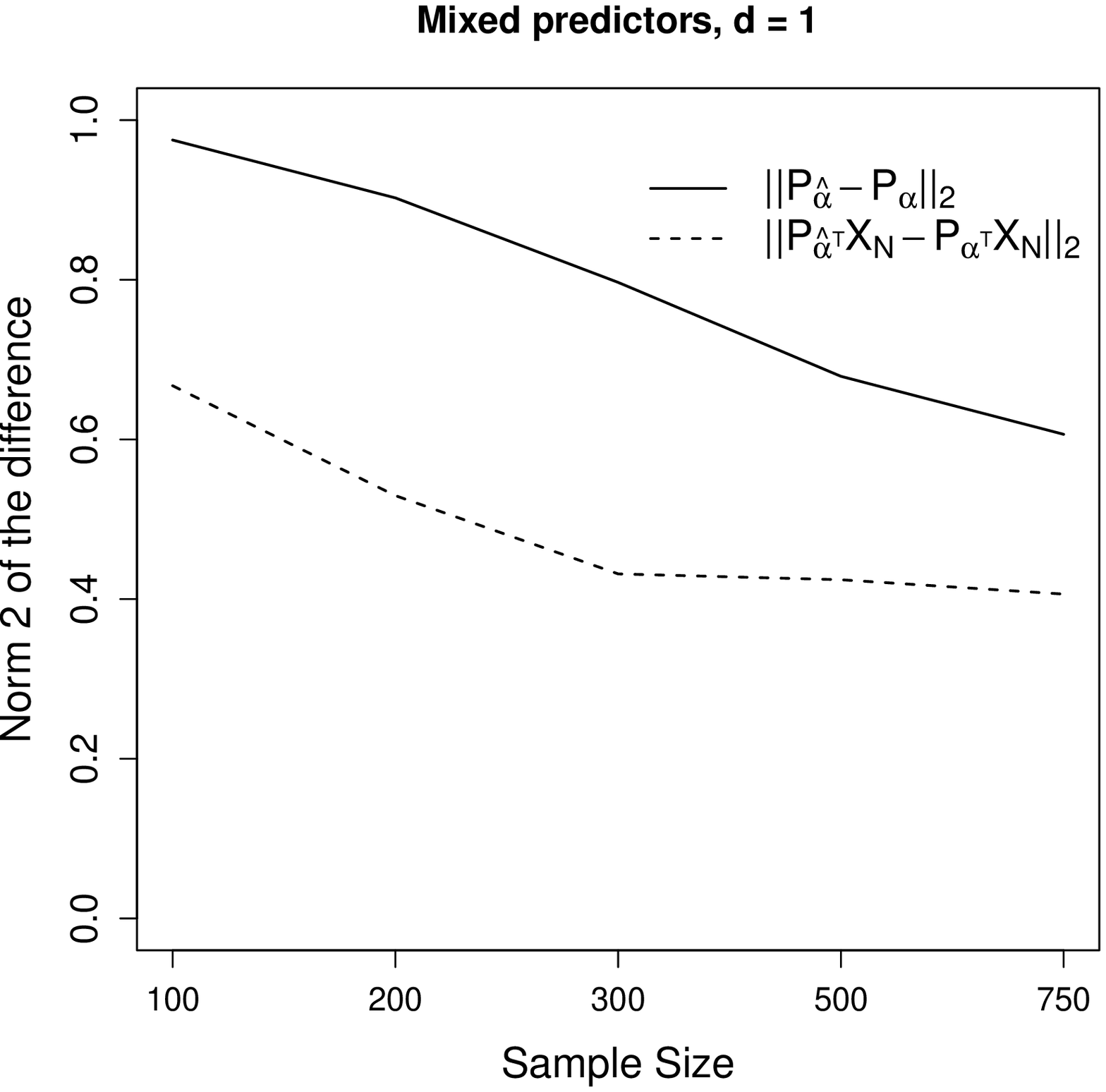}
    \includegraphics[width=7cm, height=6.9cm]{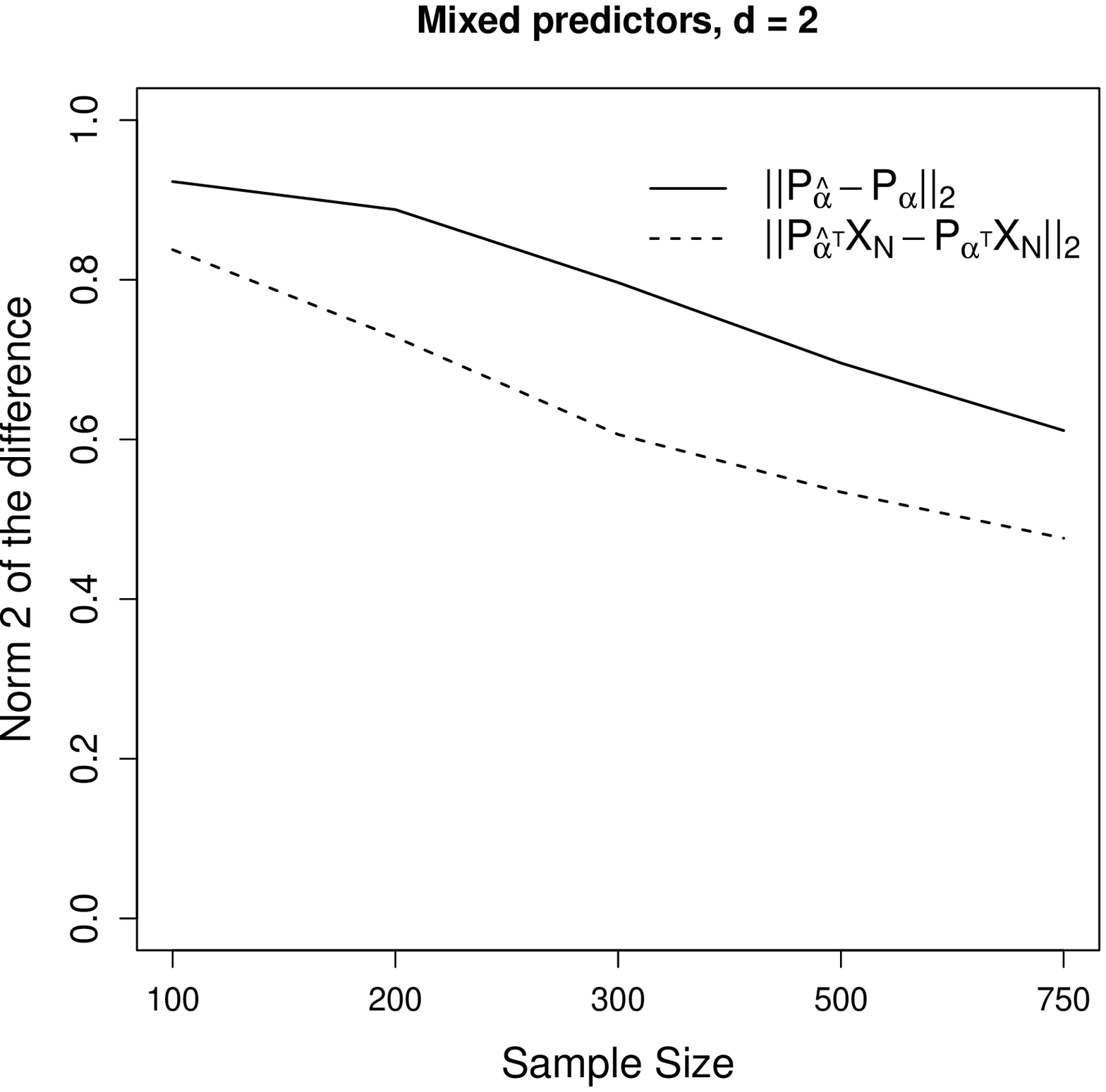}
\caption{Estimation error and out of sample prediction of optimal SDR with continuous, binary and mixed predictors for $d=1$ (Left) and $d=2$ (Right).} \label{fig:opt}
\end{figure}

 In Figure \ref{fig:opt},  we plot the estimation error $\norm{\Pb_{\widehat{\alphabf}} -\Pb_{\alphabf}}_2$  and the prediction error $\norm{\Pb_{\widehat{\alphabf}^T\X_N} -\Pb_{\alphabf^T \X_N}}_2$ for \textit{optimal SDR} on the $y$-axis   versus the training sample size on the $x$-axis  across all our  simulation scenarios. For all types of predictors the prediction  is smaller than the estimation error and both decrease as the sample size increases. Moreover, both increase as the dimension increases from 1 to 2 in the left and right panels, respectively, across types of predictors. When comparing types of predictors, continuous predictors exhibit higher estimation and prediction errors across sample sizes and mixed predictors result in the highest estimation and prediction errors.

In Figure \ref{fig:subopt} we plot  the estimation and prediction error of \textit{sub-optimal SDR}, where the continuous and binary variables are reduced separately. The pattern of behavior is consistent with that of optimal SDR in Figure \ref{fig:opt}, with the continuous variables inducing larger errors of both types  across sample sizes and $d=1,2$. Again, the errors are smaller for dimension 1. 
 
\begin{figure}
\centering
\makebox{\includegraphics[width=8cm, height=7cm]{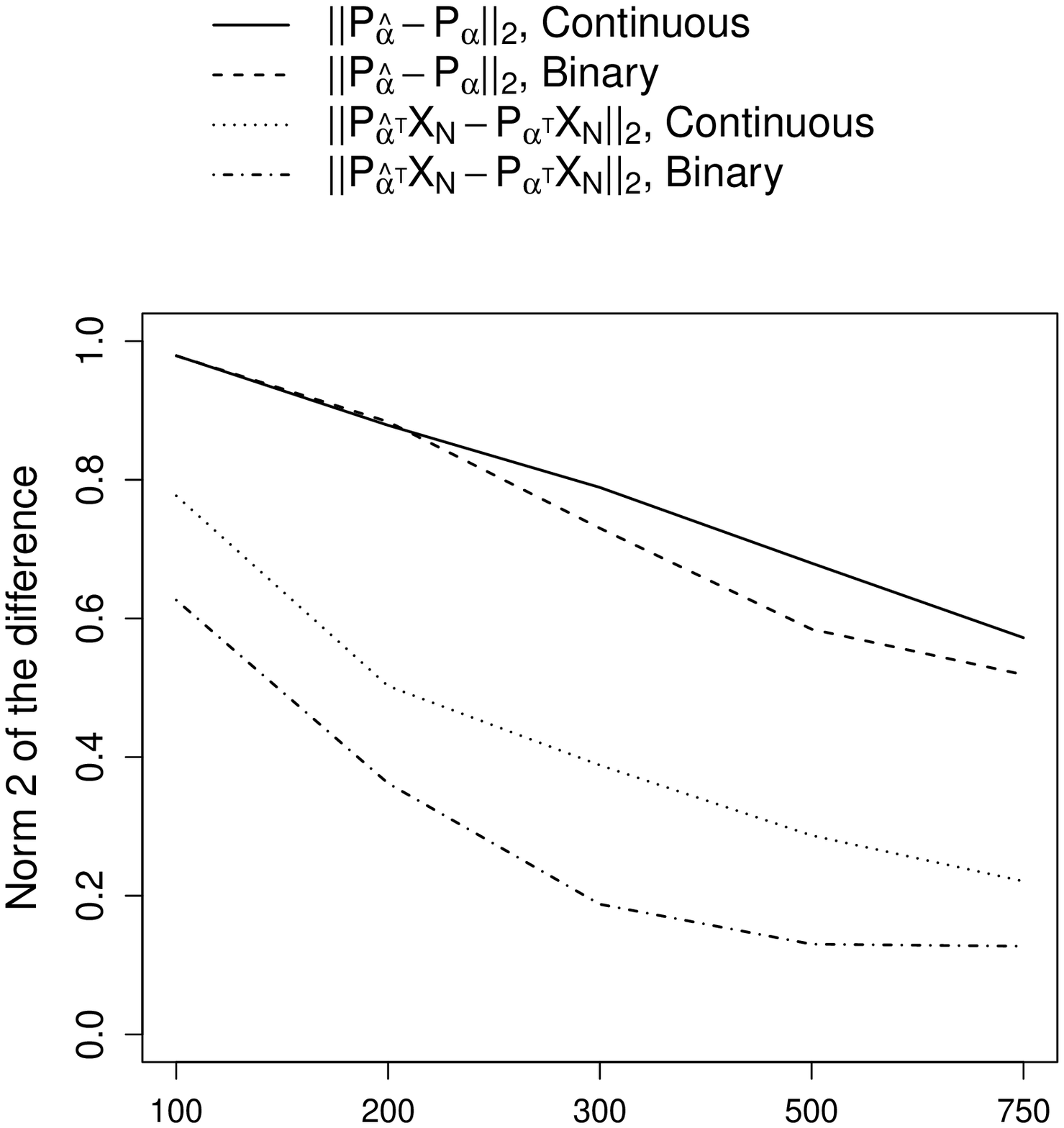}
\includegraphics[width=8cm, height=7cm]{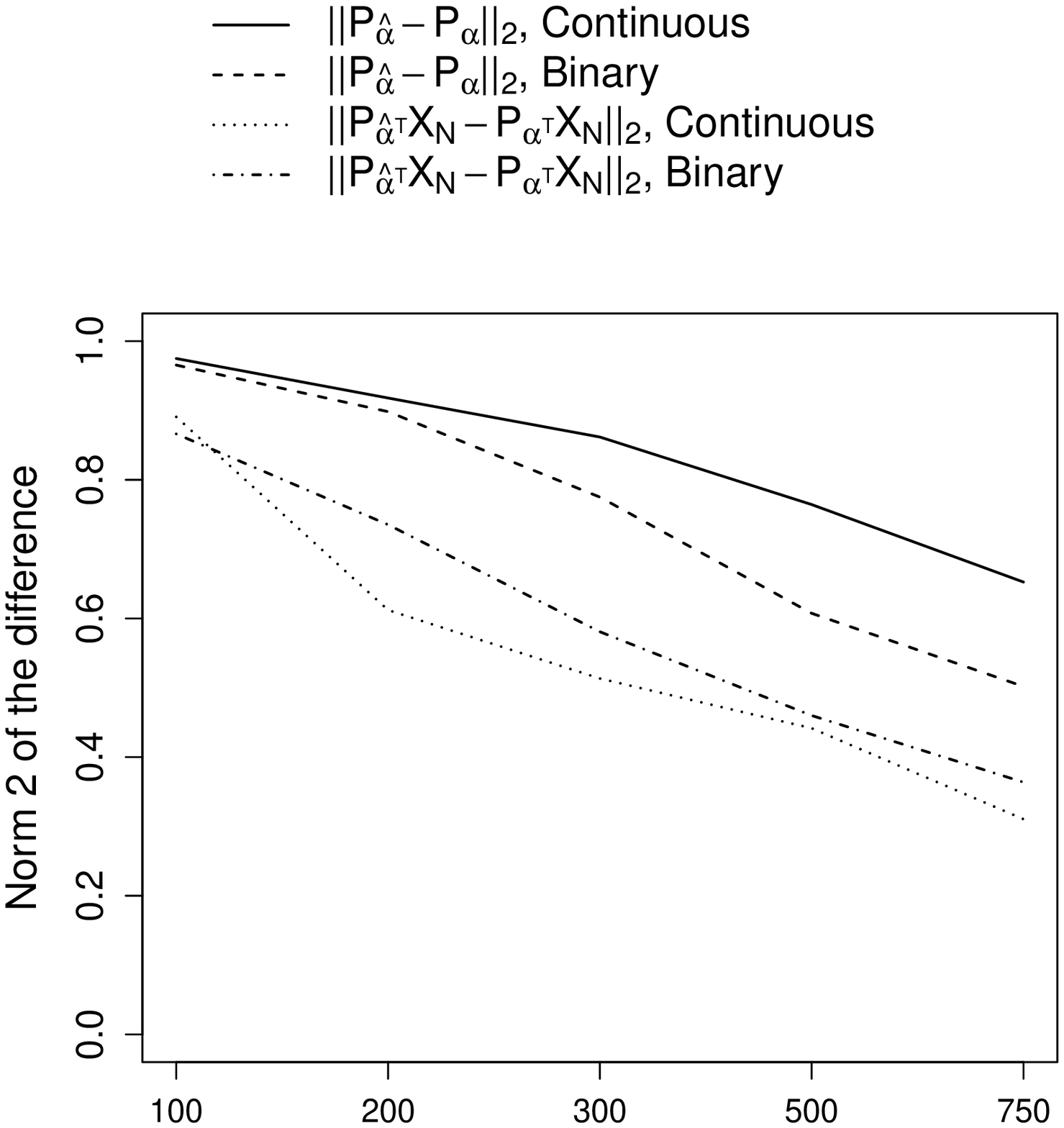}}
\caption{\label{fig:subopt}Estimation error and out of sample prediction with mixed predictors. Suboptimal Reduction for $d=1$ (Left) and $d=2$ (Right).}
\end{figure}

Under the same simulation settings, we also evaluate the performance of our simultaneous variable selection and dimension reduction method that is presented in Section~\ref{varsel}. In Table \ref{simuVS} we report the proportion of variables correctly identified as non-relevant (true positives, TP) and the proportion of variables erroneously assessed as non-relevant (false negatives, FN).  Between $d=1$ and $d=2$, TP is higher across sample sizes,  whereas FN is lower. Both rates improve substantially as the sample size increases. When all predictors are continuous both rates are lower across sample sizes. This is expected since the inclusion of a binary variable results in second order interaction effects in the reduction. Therefore, to rule out a binary variable both its own-coefficient and all the coefficients of its interaction terms must be zero. Overall, our regularized SDR approach achieves high true positive and small false negative rates for reasonable sample sizes.

\begin{table}
\caption{\label{simuVS}Accuracy of the regularized estimator in variable selection.}
\centering
%\ra{0.80}
\fbox{
\begin{tabular}{cccccccc}
%\hline
     & & & \multicolumn{5}{c}{Sample Size}\\
     \hline
     Predictors & $d$ & Rates & 100 & 200 & 300 & 500 & 750\\
    \hline
Continuous   & $1$ & TP &  0.653 & 0.751 & 0.796  & 0.851   & 0.889 \\
               &  & FN  & 0.314   & 0.17 & 0.095 & 0.044  & 0.012\\
         & $2$ & TP & 0.521   & 0.591 & 0.629  & 0.748  &  0.843\\
               
               &  & FN  & 0.165  & 0.048 & 0.014 & 0.004    & 0.002\\
  \hline
        Binary   & $1$ & TP &  0.188& 0.310 & 0.400 & 0.55  & 0.623 \\
               &  & FN   & 0.167 & 0.117 & 0.045 & 0.015  & 0.018\\
      
               & $2$ & TP  & 0.255  &  0.300 & 0.368   & 0.458   & 0.528 \\
 
               &  & FN   &  0.048 & 0.020 & 0.012 & 0.000  & 0.000\\\hline
                Mixed   & $1$ & TP & 0.632 & 0.592 & 0.589 & 0.671  &  0.674\\
               &  & FN   & 0.493 & 0.333 & 0.196  & 0.200  & 0.170 \\
               & $2$ & TP  & 0.596 & 0.656 & 0.639 & 0.583  &  0.610  \\
               &  & FN   & 0.451 & 0.413 & 0.325  & 0.163  & 0.124 \\
               %\hline
    \end{tabular}
    }
\end{table}
                                             
In Table \ref{elegir} we report the
proportion of times out of $100$ replications that the  dimension $d$ was correctly estimated
based on the sequential tests of dimension in Section \ref{testdimension} for all our simulation settings.
The sample size has a noticeable effect in the accuracy of the estimation of dimension, as expected since both tests are asymptotic. 
The weighted $\chi^2$ test accuracy suffers  more from increasing the dimension and all binary predictors as compared to that of the chi-squared test, across sample sizes. For mixed predictors, as well, the chi-squared test exhibits higher accuracy for both optimal and sub-optimal SDR across sample sizes.

\begin{table}
\caption{\label{elegir}Proportion of correct dimension estimation under the simulation settings in Section \ref{simulation}.}
\centering
\fbox{
%\ra{0.80}
%\resizebox{\columnwidth}{!} {
{\scriptsize
\begin{tabular}{@{}lcllrrrrrr@{}}
%\toprule
%\hline
Predictors &  &  & \multicolumn{6}{c}{Sample Size} \\
%\cmidrule{4-9}
   & Dimension & Test & Method &100 &200 &300 &500 &750 \\
 %   \midrule
 \hline
Continuous   
  & $d=1$ &   Weighted $\chi^2$  & & 0.60   &  0.82  &  0.83  &  0.89   & 0.95    
   \\ 
    &  & $\chi^2$  & & 0.00 &0.99& 0.99 &1.00& 0.98 
   \\ \cmidrule{2-9}
   %\midrule
 & $d=2$ &   Weighted $\chi^2$  &   &  0.65& 0.77 &0.86& 0.94& 0.94  \\
    & &  $\chi^2$  &  & 0.2& 1.00& 0.99 &0.97 &0.95 
   \\
 % \midrule
 \hline
   Binary  
   &  $d=1$& Weighted $\chi^2$&&  0 &  0.80 &0.96 &0.94  &0.94  \\
    &  & $\chi^2$&  & 0  &0.02 & 0.94 &0.99 &0.94 \\ \cmidrule{2-9}
  &  $d=2$ & Weighted $\chi^2$ & &    0          &    0.02    &  0.30     & 0.66     & 0.96 \\
    &  & $\chi^2$  & &          0.02  &  0.20&    0.92&  0.96&    0.94
   \\
 % \midrule
 \hline
    Mixed &  $d=1$&  
      Weighted $\chi^2$ & Optimal    &0	&0.2&	0.45&	0.64&	0.94
   \\
  && Weighted $\chi^2$ & Sub-optimal (cts)    &0.68	&0.84&	0.89&	0.90&	0.95
   \\
  &&  Weighted $\chi^2$ & Sub-optimal (bin)    &0.5 & 0.75 & 0.87 & 0.94 & 0.95
   \\
  && $\chi^2$ & Optimal    &0 & 0.18 & 1 & 0.98 &0.98
   \\
 &&  $\chi^2$ & Sub-optimal (cts)   
 &0&0.95&0.98& 0.98&0.95
   \\
 &&   $\chi^2$ & Sub-optimal (bin)    &0 & 0.18 & 0.94&0.98&0.96
   \\ \cmidrule{2-9}
    &  $d=2$ &  
      Weighted $\chi^2$ & Optimal   &  0          &    0.06    &  0.30     & 0.45     & 0.90 \\
  && Weighted $\chi^2$ & Sub-optimal (cts)   &  0.60& 0.75 &0.84& 0.95& 0.95
   \\
  &&  Weighted $\chi^2$ & Sub-optimal (bin)   & 0          &    0.08    &  0.40     & 0.56     & 0.96 
   \\
 &&  $\chi^2$ & Optimal    &0.08 &0.36& 0.64& 0.92& 0.93 
   \\
 &&  $\chi^2$ & Sub-optimal (cts)    &  0.12& 0.98& 0.99 &0.96 &0.95\\
&&    $\chi^2$ & Sub-optimal (bin)   &             0.22  &  0.30&    0.96&  0.96&    0.95
   \\ %\bottomrule
   %\hline
\end{tabular}
}
}
\end{table}

\section{Data Analyses}\label{realdata}
We compare our method with other approaches such as generalized linear models and principal component regression in two data applications. 
In particular,  we compare our methods with {\sc PCA} and {\sc PCAmix} in  Sections \ref{Krzanowski} and \ref{Index}.
{\sc PCAmix} \citep{Chaventetal2012,Chaventetal2014} is a version of PCA that accommodates mixed variables and implements  \textit{PCA with metrics};
i.e., Generalized Singular Value Decomposition (GSVD) of pre-processed data [see \cite{Chaventetal2014} for details]. {\sc PCAmix} is  ordinary  standard {\sc PCA}, when all variables are continuous, and standard multiple correspondence analysis ({\sc MCA}), when all  variables are categorical (\cite{Greenacre06}, \cite{Zhu11}, \cite{Camiz13}).

\subsection{Krzanowski Data Sets}
\label{Krzanowski}

\cite{Krzanowski75} studied the problem of discriminating between two groups in the presence of  both binary and continuous explanatory variables. 
\citet{Krzanowski75} modeled the mixed predictors using the \textit{location model}  \citep{OlkinTate1961} and proposed an allocation rule to two groups similar to Fisher's discriminant function. The location model transforms the $q$ binary variables $H_1,\ldots, H_q$ to  the corresponding $2^q$-category multinomial vector and requires the continuous variables be conditionally normal in each of the $2^q$ categories with different means and same variance-covariance matrix.     
He showed that the simple linear discriminant function often gives satisfactory results, except when there is interaction between the mixed variables. 

We analyze four of the five data sets in Krzanowski's paper which contains continuous and binary predictors and a binary response. 
\begin{enumerate}
\item \textbf{Data Set 1:} Ten variables recorded on 40 patients who were surgically treated for renal hypertension. Seven of the variables were continuous and three binary. After one year, 20 patients were classified as improved and 20 as unimproved.
\item \textbf{Data Set 2:} Seven variables recorded on 93 patients suffering from jaundice. Four of the variables were continuous and three binary. The two groups were patients requiring medical and surgical treatment.
\item \textbf{Data Set 3:} Twelve variables recorded on 62 patients suffering from jaundice. Eight of the variables were continuous and four binary. The two groups were patients requiring medical and surgical treatment.
\item \textbf{Data Set 4:} Eleven variables recorded on 186 patients who underwent ablative surgery for advanced breast cancer between 1958
and 1965 at Guy's Hospital, London. Six of the variables were continuous and three binary. The two groups were patients for which the treatment was deemed to be successful and failure. 

\end{enumerate}

Some of the continuous variables were transformed to normality across all data sets.
Since the response is binary, $\f_y$ in \eqref{jointlik} is a vector of frequencies with $r=1$, so that  the dimension either SDR method can detect cannot exceed 1. %Hence we set $d=1$. 
We reduced the mixed predictors using our two methods, {\sc SDR Optimal} and {\sc SDR Suboptimal}, and  also {\sc PCA} and {\sc PCAmix} setting $d=1$. 
In order to assess the classification accuracy of each method, the reduced predictors serve as  independent variables in a logistic regression model. For comparison, we also fit an \textit{unreduced} logistic regression model with all the original predictors, which we refer to as {\sc Full}. 

In Table \ref{Krzanowsky} we report the leave-one-out misclassification  rates and the area under the receiver operator characteristics curve, AUC \cite[p. 67]{Pepe2003}, with the smallest and largest values, respectively, in boldface. {\sc Sub-optimal SDR} emerges as the best method to summarize the mixed predictors with respect to misclassification error, followed by SDR Optimal that has  better performance for data set 1. With respect to AUC, {\sc SDR Suboptimal} is  always the best.

In Table \ref{Krzanowsky}, we also provide   the leave-one-out misclassification rates of Fisher's LDA, logistic regression and Krzanowski's allocation rule based on the location model, as reported in \citet[Tab. 3]{Krzanowski75}. {\sc Sub-optimal SDR}  exhibits better  performance than Krzanowski's location model across data sets. {\sc Optimal SDR} performs the best in all data sets except for data set 2 where it is on par with Fisher's linear discriminant analysis. Moreover, the {\sc Optimal} and {\sc Sub-optimal SDR}  misclassification rates are smaller than all other methods in \cite{Krzanowski75}, as well as  mixed nonparametric kernel methods \citep{VlachoMarriott1982}.
Taken all together, our {\sc SDR} methods for mixed predictors consistently produce targeted data reductions that provide better fit and prediction.

\begin{table}
\caption{\label{Krzanowsky}Leave-one-out misclassification rates and AUC values for four data sets in \citet{Krzanowski75}.}
%\begin{center}
\centering
%\vspace{.5cm}
%\resizebox{\columnwidth}{!}{%
\fbox{
{\scriptsize
\begin{tabular}{ccccc|cc|ccc}
%\toprule
{\sc Set} &  & {\sc Optimal} &  {\sc SubOpt.} &  {\sc Full}  & {\sc PCA} & {\sc PCAmix} & {\sc Location}& {\sc Fisher}& {\sc Logistic}\\
\midrule 
1  &  MR &\textbf{0.250}  &  {0.300} &  0.375 &  { 0.325} & 0.425 & 0.350 & { 0.325} & { 0.325} \\
&AUC  & {\bf 0.918}  &  {\bf 0.918} &  0.885 & 0.675 & 0.575 & - & - & - \\
\midrule
2  & MR & {0.280}  &   \textbf{0.204}  & 0.258 &  0.387 & 0.290 & 0.290 & {0.280} & 0.301 \\ 
& AUC & {0.857}  & {\bf 0.858} & 0.837 & 0.513 & 0.469 & - & - & - \\
\midrule 
3  & MR &   {0.161}  &  \textbf{0.145} & 0.226 & 0.484 &0.500 & 0.226 & { 0.177} & 0.222\\
& AUC & 0.949 &  {\bf 0.951} &   0.944 & 0.623 & 0.646 & - & - & - \\
\midrule
4  & MR & {\bf 0.296} &  {\bf 0.290} &   0.392 &0.457 & 0.430 & { 0.328} & 0.382 & 0.371\\ 
& AUC& {\bf 0.784}  & {\bf 0.785} &  0.738 & 0.544 & 0.572 & - & - & - \\
%\bottomrule
\end{tabular} 
}
}
%\end{center}
\end{table}

%%%%%%%%%%%%%%%%%%%%%%%%%%%%%%%%%%%%%%%%%%%%%%%%%%%%%%%%%%%%%

\subsection{Governance index application}\label{Index}

Considerable social science and economics research is devoted to the construction of indexes for descriptive and predictive purposes 
\citep{Vyas06,Kolenikov09,Filmer12,Merola14, Forzanietal2018}. An index is a statistical summary measure of change in a representative group of individual data points. It usually synthesizes the information contained in a set of $p$ variables $\X\in \real^p$ via a linear combination, $\R(\X)=\omegabf^T\X \in \real$, where  $\omegabf$ is the vector of weights of the composite index. 

In this example, we study the impact of governance on economic growth in the twelve South American countries as measured by per capita \textit{Gross Domestic Product} (GPD) using the World Bank Governance Indicators.\footnote{Governance Indicators and per capita GDP data can be downloaded from  \href{https://info.worldbank.org/governance/wgi}{\textit{Worldwide Governance Indicators}} and \href{https://data.worldbank.org/indicator/NY.GDP.MKTP.KD}{\textit{The World Bank Data}}, respectively.}  
The World Bank considers the following six aggregate indicators of governance that combine the views of a large number of enterprise, citizen and expert survey respondents: 
control of corruption ($X_1$); rule of law ($X_2$); regulatory quality ($X_3$);  government effectiveness ($X_4$); political stability ($X_5$);  voice and accountability ($X_6$). They are standardized to have mean zero and  standard deviation one, with values from  -2.5 to 2.5, approximately, where  higher values correspond to better governance. 
All six  are highly positively correlated, and are all positively correlated with the per capita GDP; i.e., economic growth is positively associated with better governance indicators.

Our aim is to build a {\it Composite Governance index} (CG) to predict $Y$, the logarithm of per capita \textit{Gross Domestic Product} (GPD), measured in  2010 US dollars, over the period 1996 to 2018.  Using the set of governance indicator variables, we start by constructing the CG index via standard Principal Component Analysis (PCA) and Principal Fitted Components (PFC) [see Corollary \ref{CoroPFC}] setting  
   $d=1$ and $\f_y=\log(GDP)$ in \eqref{jointlik}. 

In the left panel of  Figure \ref{CGI1}, we plot $\log(GDP)$ versus the CG indexes based on PCA, which is the standard approach  in such index construction \citep{Mazziotta2019}. In the right panel of  Figure \ref{CGI1}, the response is plotted versus the index based on PFC. 
Both plots indicate dependence of the response on the indexes but the nature of relationship is the data pattern is hard to understand. A linear trend appears stronger in the right panel, which is reflected in the better fit of  the linear regression model (black) with $R^2=0.27$ versus 0.17 for PCA. However, the PCA-based index in nonparametric kernel regression (blue)  results in better fit. Using the \texttt{np} \texttt{R} package, the value of the nonparametric version of $R^2$ is 0.32 for the PFC-based CG index, which is much lower than  0.54, the value for the PCA-based index.

\begin{figure}
    \centering
     \makebox{\includegraphics[width=0.45\textwidth]{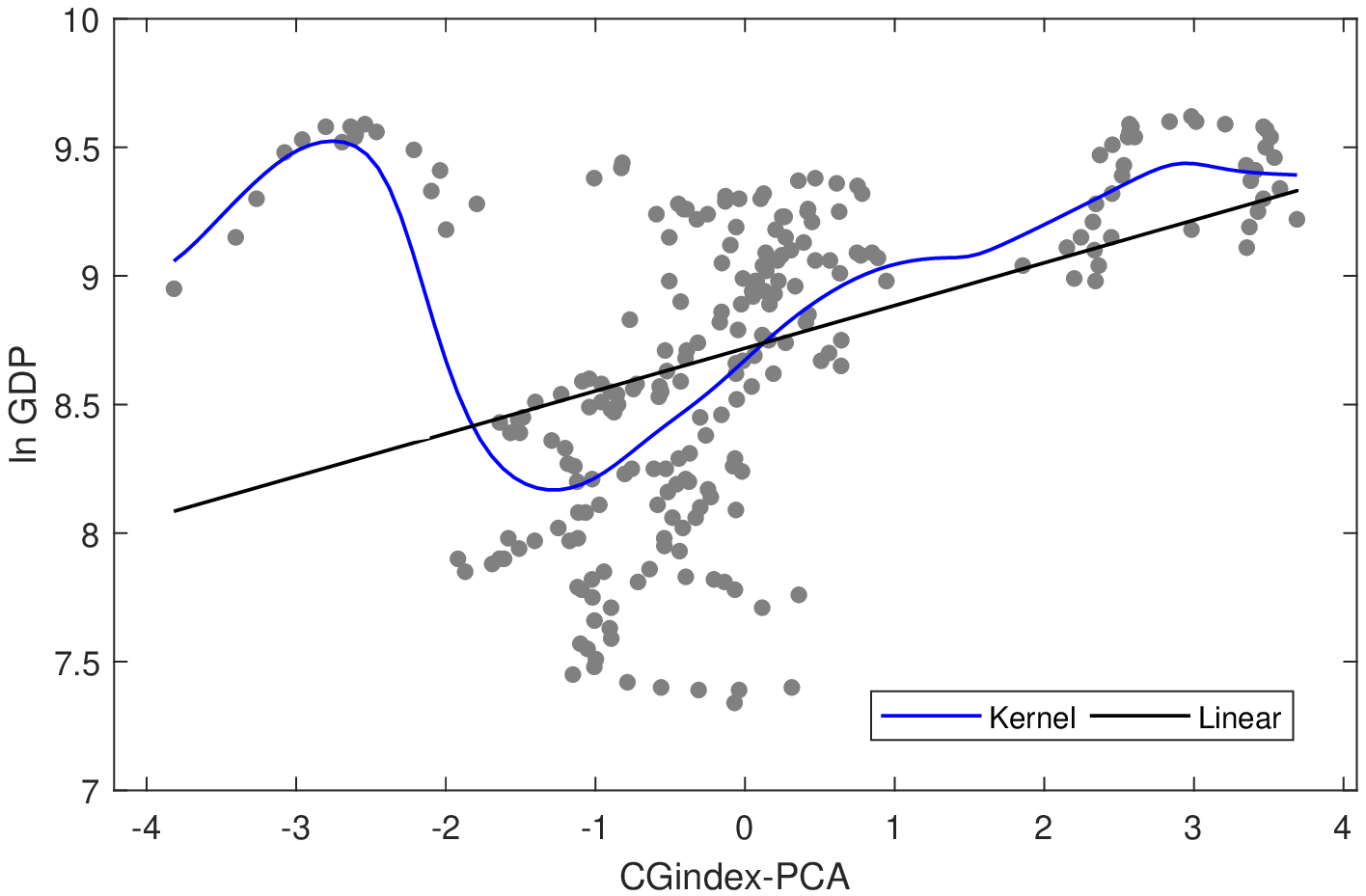}
      \includegraphics[width=0.45\textwidth]{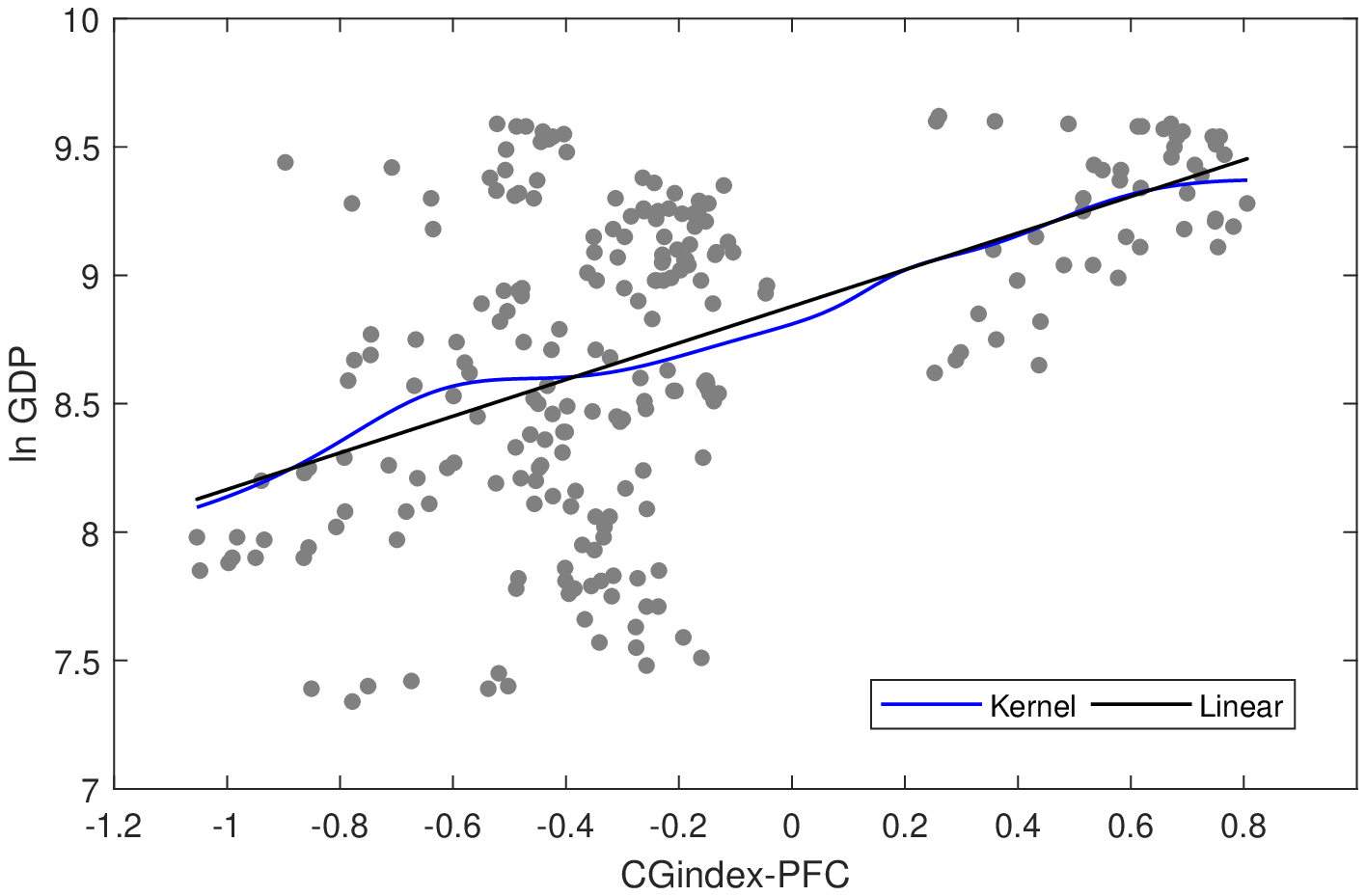}}
    \caption{\label{CGI1}Log of per capita GDP versus Standard \texttt{PCA} and \texttt{PFC} based composite governance indexes.}
\end{figure}

In Figure \ref{CGIcolor1}, we plot $\log(GDP)$ versus the PCA and PFC composite governance indexes by country. The plots indicate that the PFC index gives a much better visualization of the relationship of $\log(GDP)$ within each country, 
suggesting that adjusting the index by country could improve its predictive performance.

\begin{figure}
    \centering
    \makebox{\includegraphics[width=0.45\textwidth]{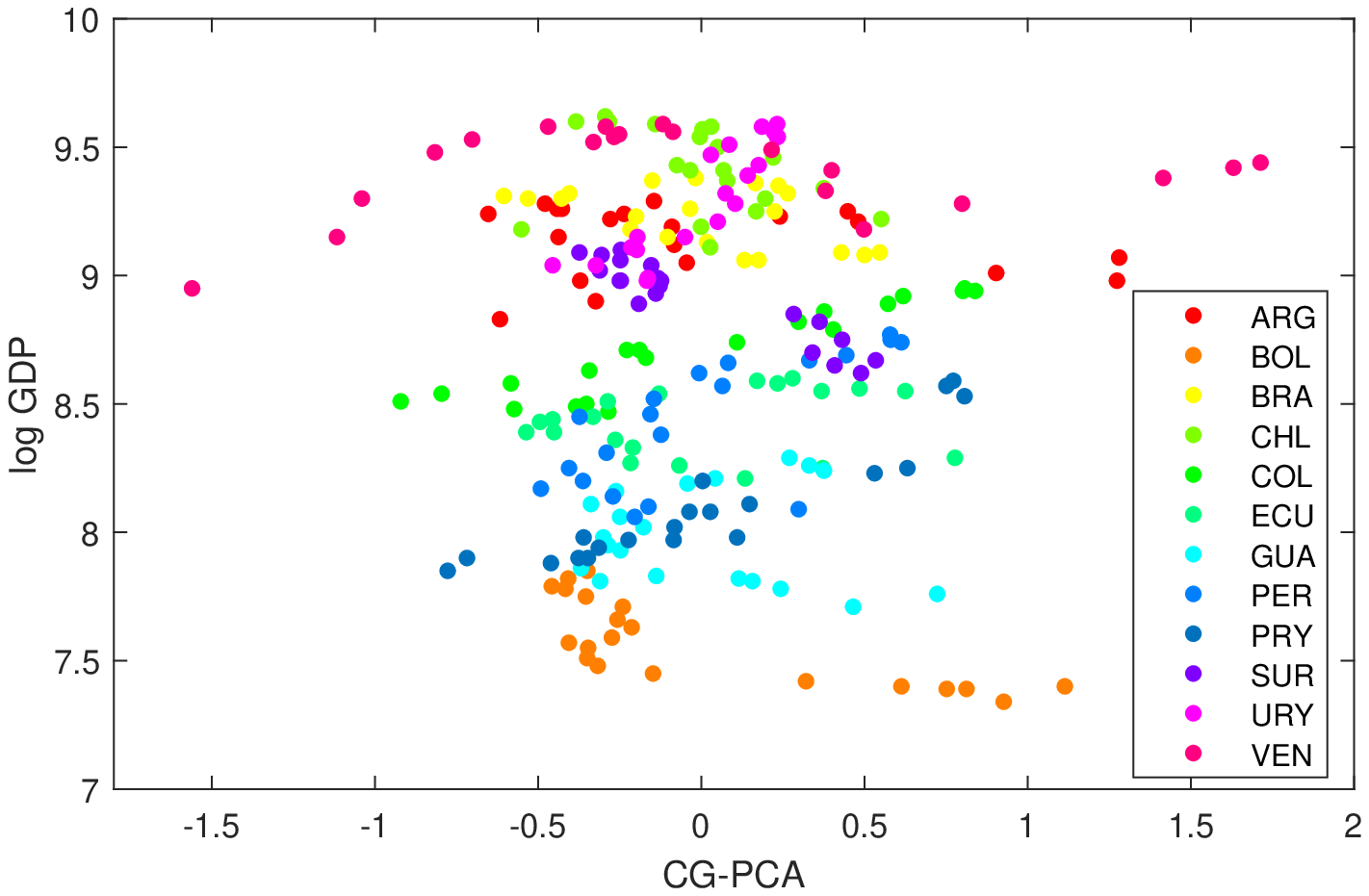}
    \includegraphics[width=0.45\textwidth]{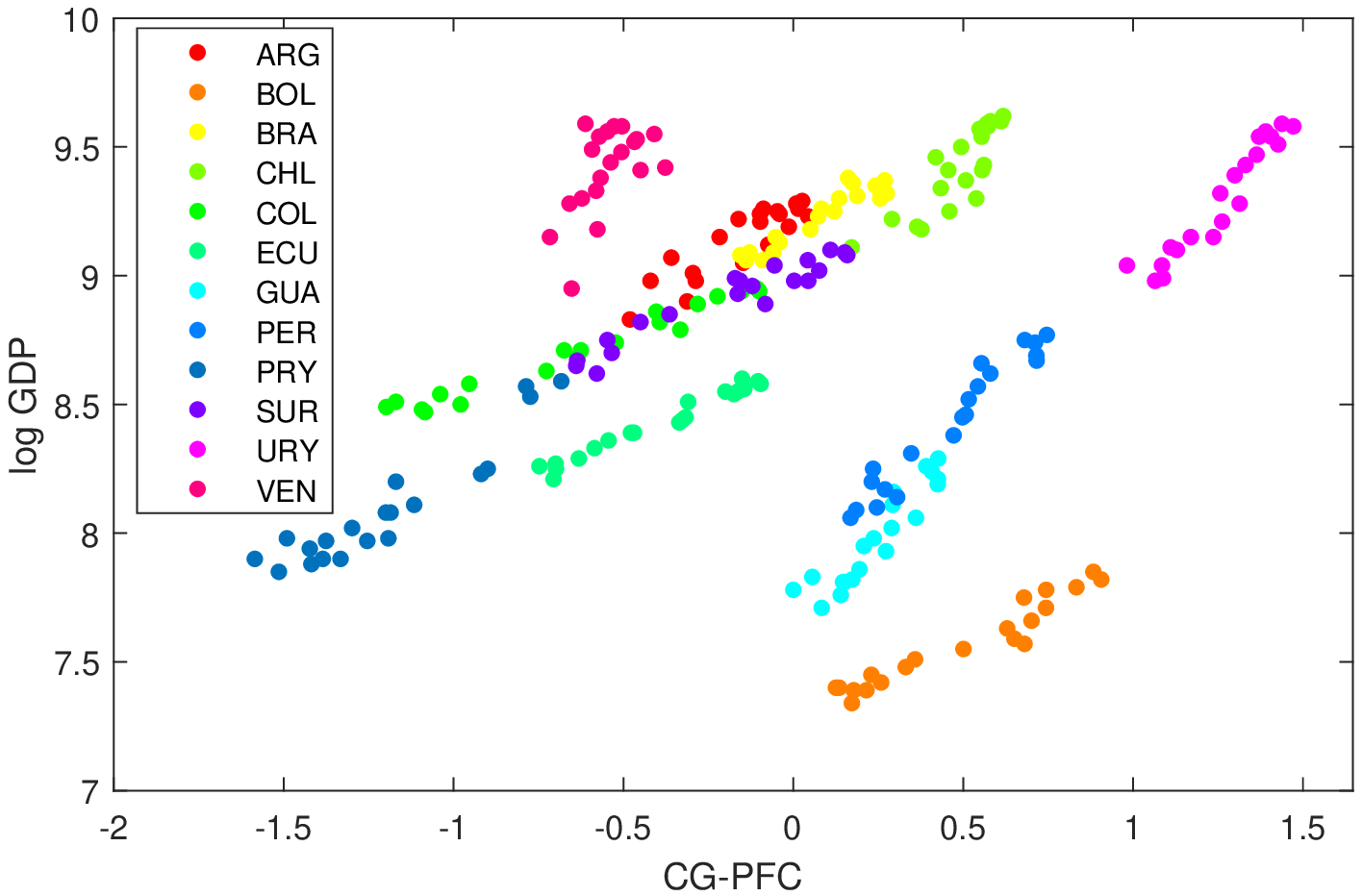}}
    \caption{\label{CGIcolor1} Log of per capita GDP versus standard PCA and standard PFC composite governance indexes by country.}
\end{figure}

We add country effect by introducing eleven binary variables $\Hbf$.
In Figure \ref{CGI3} we plot the log of GDP versus the CG index constructed by PCA for mixed variables ({\sc PCAmix}) %\citep{PCAmixto2,PCAmixto4} 
 in the left panel and by our mixed {\sc optimal SDR} approach  in the right panel. 
Hardly any difference  between the plots in the left panels of Figures \ref{CGI1} and \ref{CGI3} is noticeable. The  {\sc PCAmix} based CG index is very similar to the conventional {\sc PCA} based CG that does not  include country effect, with  $R^2$ equal to 0.17 and 0.61 for the linear and nonparametric models, respectively. Moreover, neither {\sc PCA} based CG index exhibits an easy to understand or model relationship with the response.

In contrast, a very clear and simple   pattern appears in the right panel of  Figure \ref{CGI3}, where the response is plotted versus our {\sc optimal SDR} based index.  The pattern suggests modeling $\log(GDP)$ as a linear function of the GC index.  This is a distinct improvement over {\sc PCA} and  {\sc PCAmix} (left panels of  Figures \ref{CGI1} and \ref{CGI3}) but also the SDR method {\sc PFC}, which does not account for country effect (right panel of  Figure \ref{CGI1}).   As a result, both the linear (black) and the kernel (blue) regression models for the regression of the log per capita GDP on the {\sc optimal SDR} for mixed predictors based CG index have excellent fit with respective $R^2$ values of  0.91 and 0.93.

\begin{figure}
    \centering
    \makebox{\includegraphics[width=0.45\textwidth]{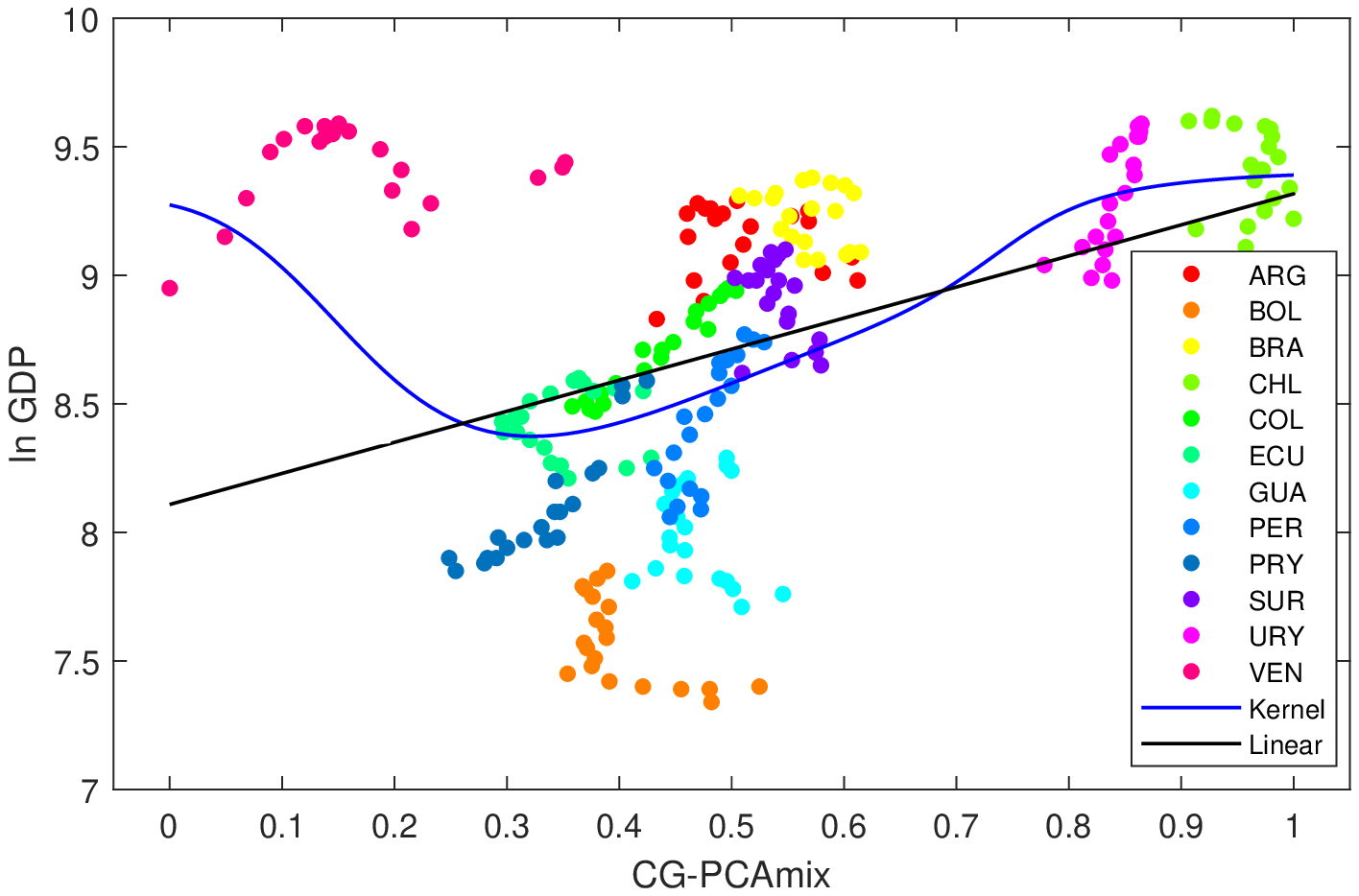}
      \includegraphics[width=0.45\textwidth]{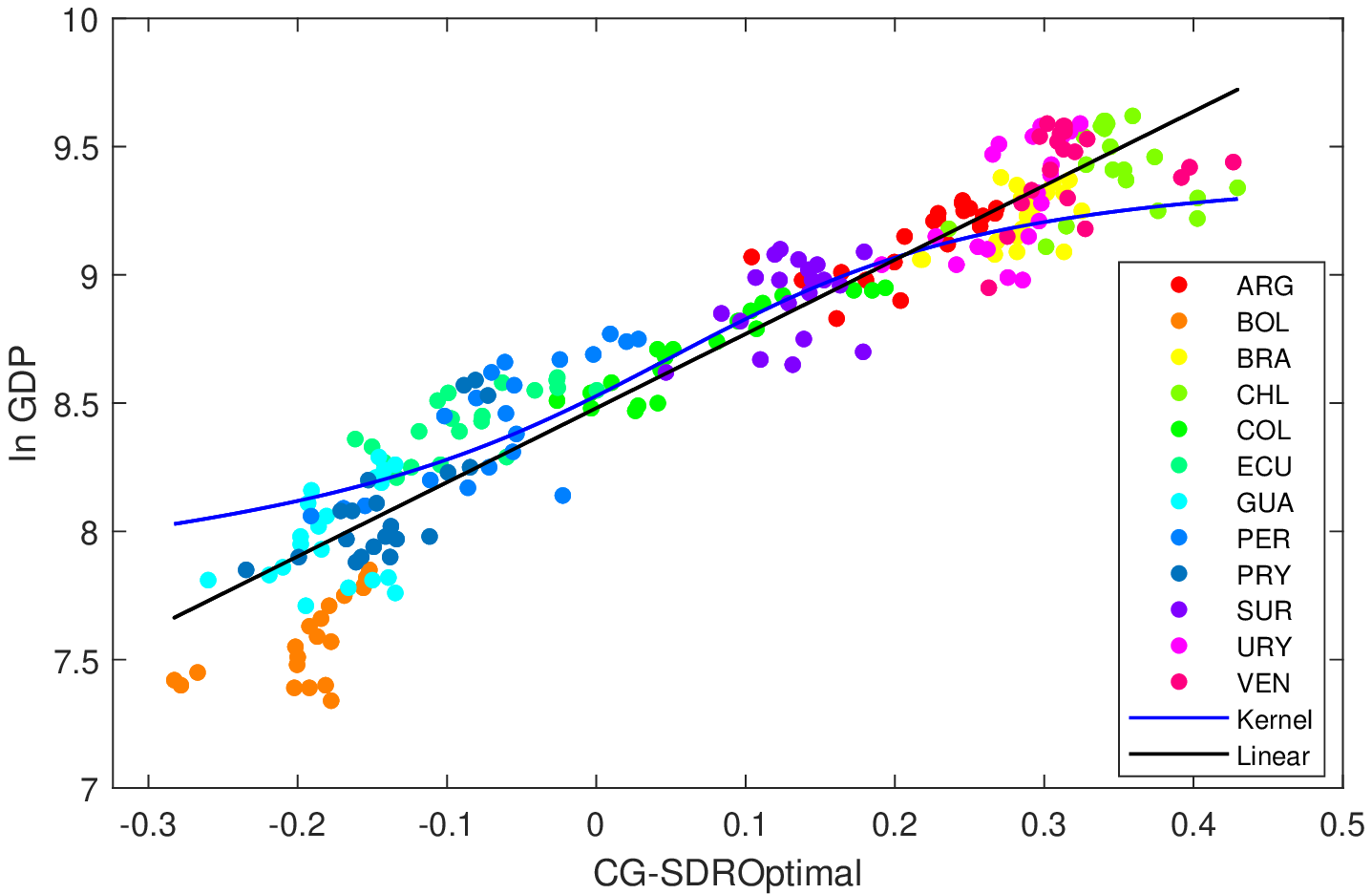}}
    \caption{\label{CGI3}Log of per capita GDP versus Composite Governance index  with country effect.}
    \end{figure}

The average  of the leave-one-out mean square  prediction errors of the linear and kernel regression models in Table \ref{governance}, provides an unbiased measure of predictive performance. The logarithm of the per capita GDP is regressed on the unsupervised CG indexes, constructed by {\sc PCA} using only continuous predictors ($\text{PCA}(\X)$) %and including country effects as factor in the fitted regression ($\text{PCA}(\X)+\Hbf$),
and its extension for mixed variables ($\text{PCAmix}(\X,\Hbf)$), and the supervised CG Indexes, constructed by {\sc PFC} only on continuous predictors  ($\text{PFC}(\X$)) and our mixed predictor SDR methods, $\text{\sc Optimal}(\X,\Hbf)$ and $\text{\sc SubOptimal}(\X,\Hbf)$.

The leave-one-out mean squared prediction errors of the supervised {\sc PFC} based CG index are smaller than both {\sc PCA} and {\sc PCAmix} for the  linear model, even though {\sc PFC} does not account for country  effect. Nevertheless, when the kernel regression model is fitted, the {\sc PCA} based index exhibits better performance than {\sc PFC}. 
The dramatic drop in prediction error results from using {\sc optimal} and {\sc sub-optimal SDR}, as it is between 5 to 9 times smaller than the {\sc PCA, PCAmix} and {\sc PFC} errors for both the linear regression and the kernel regression models.

\begin{table}
\caption{\label{governance}Leave-one-out mean squared prediction errors for the per capita log GDP in South-American countries.} 
 \centering
%\vspace{.2cm}
%\resizebox{\columnwidth}{!}{%
\fbox{
\begin{tabular}{llcc}
%\toprule
  &  & \multicolumn{2}{c} {Predictive Model}\\
Index Type  & Method & Linear & Non-Parametric\\
\midrule
Unsupervised & $\text{PCA}(\X)$ & 0.319 & 0.189\\
 &$\text{PCAmix}(\X,\Hbf)$ & 0.320 & 0.209\\
 \midrule
 Supervised & $\text{PFC}(\X)$ & 0.292 & 0.282\\
 &$\text{\sc SDROptimal}(\X,\Hbf)$ &0.029 & 0.028 \\
 &$\text{\sc SDRSubOptimal}(\X,\Hbf)$ & 0.028 & 0.022\\
%\bottomrule
\end{tabular}
}
\end{table}

The regularized estimation of the {\sc SDROptimal} reduction  selects all five continuous predictors except for {\it rule of law}. {\it Political stability} and  {\it voice and accountability}  have the highest weights in the CG index. \textit{Rule of law} is the most correlated with four of the other variables, with correlation coefficient values over 0.80. We stipulate that our method drops it as its relationship with GDP is mostly absorbed by the other four. 
The binary variables are all selected. That is, our method finds a significant country effect on GDP.

\section{Discussion}\label{discussion}

Our approach falls within \textit{model-based inverse regression} for sufficient dimension reduction (SDR) (\cite{Cook2007,CookForzani2008,BuraForzani2015,BuraDuarteForzani2016}). Model-based SDR requires knowledge of the family of distributions of the inverse predictors in contrast to moment-based SDR, such as {\sc SIR} \cite{Li1991}, {\sc SAVE} \cite{CookWeisberg1991}, or {\sc DR} \cite{LiWang2007}, that impose conditions on the moments of the marginal distribution of the predictors. Because of this, our approach provides exhaustive identification and statistically \textit{efficient} estimation of sufficient reductions for the conditional distribution of an output given mixed variables that contain all information in the mixed predictors for the output $Y$. 

Furthermore, outside the context of dimension reduction for the forward regression problem of $Y$ on mixed predictors $\Z$, the modeling we use to accommodate the factorization in \eqref{cdf.fact} in developing our SDR methods, is a new  multivariate modeling approach for \textit{response} vectors comprised of mixed variables. 
That is, if one were to only consider the multivariate regression of the mixed vector $\Z=(\X^T,\Hbf^T)^T$ on some other variables, say $\F$, the models we use for the continuous and binary elements of $\Z$ in our development provides a new regression tool for mixed responses. Specifically, since the joint distribution of $\Z \mid \F$ belongs to the exponential family \eqref{genExpFam}, our approach yields \textit{sufficient statistics} for the unknown natural parameters $\varthetabf$ in \eqref{naturalpar}, as well as optimal (efficient) maximum likelihood  estimators, in a similar manner to generalized linear modeling for univariate responses.

\acks{EB would like to acknowledge support for this project
from the Austrian Science Fund (FWF P 30690-N35) and the Vienna Science and Technology Fund (WWTF ICT19-018).}

%%%%%%%%%%%%%%%%%%%%%%%%%%%%%%%%%%%%
%%%%%%%%%%%%%%%%%%%%%%%%%%%%%%%%%%%%

\newpage

\section*{Appendix A. Proofs and Derivations for Section \ref{SufRed}}\label{app:A}

\subsection*{Derivation of Eqn. (\ref{genExpFam})}\label{A}

From Eqn. (\ref{jointlik}), the density $f(\X,\Hbf|Y=y)$, up to the  constant $1/\sqrt{2\pi}$, equals 
\begin{align*}
& \exp \bigg\{-\frac{1}{2} \Big((\X - \muxbar) - \A \f_y - \betabf (\Hbf-\muhbar)\Big)^T \Deltabf^{-1} \Big((\X - \muxbar) - \A \f_y - \betabf (\Hbf-\muhbar) \Big)  \\ & \hspace{4cm}
+ \vech^T(\Hbf \Hbf^T)\left(\taubf_0 + \taubf\f_y\right) +\frac{1}{2} \log(|\Deltabf|^{-1}) - \log(G(\Gammabf_y)) \bigg\}. 
\end{align*}
After some algebra and rearrangement of terms we obtain
\begin{equation*}
f(\X,\Hbf|Y=y)= h(\X,\Hbf) \exp\left(\T^T(\X,\Hbf)\etabf_y    - \psi(\etabf_y )\right),
\end{equation*}
with
$h(\X,\Hbf) = (2\pi)^{-1/2}$, 
\begin{align}\label{eq1}
\T^T(\X,\Hbf)\etabf_y &=  \X^T\Deltabf^{-1}\muxbar  - \X^T\Deltabf^{-1}\betabf \muhbar + \X^T\Deltabf^{-1}\A\f_y  \notag\\
&\hspace{1cm}  - \Hbf^T \betabf^T\Deltabf^{-1}\muxbar + \Hbf^T \betabf^T\Deltabf^{-1}\betabf \muhbar - \Hbf^T \betabf^T\Deltabf^{-1}\A \f_y  \notag\\
&\hspace{1cm}  - \frac{1}{2}\X^T\Deltabf^{-1}\X +\X^T\Deltabf^{-1} \betabf \Hbf  \notag\\
& \hspace{1cm}  - \frac{1}{2}\Hbf^T \betabf^T\Deltabf^{-1}\betabf \Hbf  
+ \vech^T(\Hbf \Hbf^T) \taubf_0 + \vech^T(\Hbf \Hbf^T)\taubf\f_y,
\end{align}
and 
\begin{align}\label{psii}
  \psi(\etabf_y) &= \frac{1}{2}\muxbar^T\Deltabf^{-1} \muxbar +\frac{1}{2}\f_y^T\A^T\Deltabf^{-1} \A \f_y + \frac{1}{2} \muhbar^T \betabf^T \Deltabf^{-1} \betabf \muhbar \notag \\ &\hspace{1cm} + \muxbar^T\Deltabf^{-1} \A \f_y - \muxbar^T\Deltabf^{-1}\betabf \muhbar- \muhbar^T \beta^T \Deltabf^{-1} \A \f_y \\ & \hspace{1cm}- \frac{1}{2} \log(|\Deltabf|^{-1})+ \log(G(\Gammabf_y)). \notag
\end{align}
Since $\tr(\A^T \B) =  \vecc(\A)^T \vecc(\B)$, $\vecc(\A\B\Cc) = (\Cc^T \otimes \A) \vecc(\B)$ and $\D_q$ in Section \ref{model} is such that $\vecc(\A) = \D_q \vech(\A)$,  (\ref{eq1}) becomes
\begin{align}\label{eq2}
\T^T(\X,\Hbf)\etabf_y &=   \X^T\left(\Deltabf^{-1}\muxbar-  \Deltabf^{-1}\betabf \muhbar
+  (\f_y^T \otimes \I_p) \vecc(\Deltabf^{-1}\A)\right) \notag\\
&\hspace{.3cm} + \Hbf^T \left( -\betabf^T\Deltabf^{-1}\muxbar +   \betabf^T\Deltabf^{-1}\betabf \muhbar -  (\f_y^T \otimes \I_q)  \vecc(\betabf^T\Deltabf^{-1}\A )\right) \notag\\
&
\hspace{.3cm} - \frac{1}{2} (\D_p \D_p^T \vech(\X\X^T))^T \vech(\Deltabf^{-1}) + \vecc(\X\Hbf^T)^T\vecc(\Deltabf^{-1} \betabf) \notag \\
&\hspace{.3cm} +\vech (\Hbf\Hbf^T)^T \left( - \frac{1}{2} \D_q^T \vecc(  \betabf^T\Deltabf^{-1}\betabf) 
+  \taubf_0 + (\f_y^T \otimes \I_{q(q+1)/2}) \vecc( \taubf )\right). \notag
\end{align}
Finally, using the matrices $\J_q$ and $\Lbf_q$ defined in Section \ref{model}, we obtain Eqns (\ref{suffstat}) and (\ref{naturalpar}) from
\begin{align*}
\T^T(\X,\Hbf)\etabf_y &=   \X^T\left(\Deltabf^{-1}\muxbar-  \Deltabf^{-1}\betabf \muhbar
+  (\f_y^T \otimes \I_p) \vecc(\Deltabf^{-1}\A)\right) \\
&\hspace{.3cm} + \Hbf^T \left( -\betabf^T\Deltabf^{-1}\muxbar +   \betabf^T\Deltabf^{-1}\betabf \muhbar -  (\f_y^T \otimes \I_q)  \vecc(\betabf^T\Deltabf^{-1}\A )\right) \\
&\quad - \frac{1}{2} (\D_p \D_p^T \vech(\X\X^T))^T \vech(\Deltabf^{-1})  + \vecc(\X\Hbf^T)^T\vecc(\Deltabf^{-1} \betabf)   \\
& \quad + \Hbf^T   \left( - \frac{1}{2} \Lbf_q \D_q^T \vecc(  \betabf^T\Deltabf^{-1}\betabf)
+  \Lbf_q \taubf_0 + (\f_y^T \otimes \I_{q}) \vecc(\Lbf_q \taubf )\right) \\
&\quad +(\J_q\vech (\Hbf\Hbf^T))^T \left( - \frac{1}{2} \J_q \D_q^T \vecc(  \betabf^T\Deltabf^{-1}\betabf) 
+  \J_q \taubf_0 + (\f_y^T \otimes \I_{k_q}) \vecc(\J_q \taubf )\right)  \\
&=  \X^T \etabf_{y1}+ \Hbf^T \etabf_{y2}  - \frac{1}{2} (\D_p^T \D_p\vecc(\X\X^T))^T \etabf_3 + \vecc(\X\Hbf^T)^T \etabf_4 \\
&\qquad +(\J_q\vech (\Hbf\Hbf^T))^T \etabf_{y5},  
\end{align*}
where $\T(\X,\Hbf)$ is defined in (\ref{suffstat}) and  
\[\etabf_{y1} = \Deltabf^{-1} \muxbar - \Deltabf^{-1} \betabf \mubf_{\Hbf}+ (\f_y^T \otimes \I_p) \vecc(\Deltabf^{-1}\A) =\F_{y1}  \varthetabf_1, \]
with $ \F_{y1} = (\I_p , \f_y^T \otimes \I_p) $, $\varthetabf_1 = (\varthetabf_{10}^T , \varthetabf_{11}^T)^T$, $\varthetabf_{10} =  \Deltabf^{-1} \muxbar - \Deltabf^{-1} \betabf \mubf_{\Hbf}$, $\varthetabf_{11} = \vecc (\Deltabf^{-1}\A )$,
\begin{align*}
\etabf_{y2} &= -\betabf^T   \Deltabf^{-1}\muxbar + \betabf^T \Deltabf^{-1} \betabf \mubf_{\Hbf} -  (\f_y^T\otimes \I_q) \vecc(\betabf^T \Deltabf^{-1} \A)  \\ 
&\qquad -
\frac{1}{2} \Lbf_q {{\D}}^T_q \vecc (\betabf^T \Deltabf^{-1} \betabf)
+ \Lbf_q     \taubf_0 +(\f_y^T \otimes \I_{q}) \vecc(\Lbf_q \taubf )  \\
&= \F_{y2} \varthetabf_2,
\end{align*}
with $ \F_{y2} = (\I_q , \f_y^T \otimes \I_q) $,
$\varthetabf_2 = (\varthetabf_{20}^T , \varthetabf_{21}^T )^T$, $\varthetabf_{20} = -\betabf^T   \Deltabf^{-1}\muxbar  
+ \betabf^T \Deltabf^{-1} \betabf \mubf_{\Hbf}+ \Lbf_q \taubf_0   -
\frac{1}{2} \Lbf_q {{\D_q^T}} \vecc (\betabf^T \Deltabf^{-1} \betabf) $, $\varthetabf_{21} = \vecc(\Lbf_q \taubf - \betabf^T \Deltabf^{-1} \A)$,
\[
\etabf_{3} = \etabf_{3y} = \vech (\Deltabf^{-1}),
\]
\[
\etabf_{4} = \etabf_{4y} = \vecc (\Deltabf^{-1}\betabf),
\]
and
\begin{eqnarray*}
\etabf_{y5}&= &  - \frac{1}{2} \J_q \D_q^T \vecc(  \betabf^T\Deltabf^{-1}\betabf) 
+  \J_q \taubf_0 + (\f_y^T \otimes \I_{k_q}) \vecc(\J_q \taubf )\\
&= &\F_{y5} \varthetabf_5,
\end{eqnarray*}
with $\F_{y5}=(\I_{k_q} , \f_y^T \otimes \I_{k_q})$, $\varthetabf_5= (\varthetabf_{50}^T , \varthetabf_{51}^T)^T$, $\varthetabf_{50} =
- {{\frac{1}{2} \J_q \D^T_q}} \vecc  (\betabf^T \Deltabf^{-1} \betabf ) + \J_q \taubf_0 $
and $\varthetabf_{51} = \vecc (\J_{q} \taubf)$.

By Eqn. (\ref{ising-vech2}),
\[
G(\Gammabf_y) = \sum_H \exp\big[\vech^T(\Hbf \Hbf^T)\left(\taubf_0 + \taubf\f_y\right)\big].
\]
Plugging in matrices $\J_q$, $\Lbf_q$, $\D_q$ and $\Cc_q$, defined in Section \ref{model}, and letting  $\bar{\etabf}_4= \unvecc ({\etabf}_4)$, we obtain
\begin{eqnarray}\label{GdeGy}
G(\Gammabf_y)&=& \sum_H \exp \left[ ( \J_q \Cc_q \vecc (\Hbf \Hbf^T))^T \left(\etabf_{y5 }+    \frac{1}{2} \J_q{{\D}}^T_q {\vecc}(  \bar{\etabf}_4^T (\unvecc (\D_p \etabf_3))^{-1} \bar{\etabf}_4) )\right.\right) \\
&&\left. + \Hbf^T \left(\etabf_{y2}+ \bar\etabf_4^T (\unvecc (\D_p \etabf_3))^{-1}\etabf_{y1} + \frac{1}{2}\Lbf_q {{\D}}_q^T \vecc( \bar{\etabf}_4^T (\unvecc (\D_p \etabf_3))^{-1}\bar{\etabf}_4))\right)\right]. \notag
\end{eqnarray}
Finally,  using the matrix $\D_p$ defined in Section \ref{model}, Eqn. (\ref{psii}) yields
\begin{eqnarray}\label{psiii}
\psi (\etabf_y) &=& \frac{1}{2} \etabf_{y1}^T  (\unvecc (\D_p \etabf_3))^{-1} \etabf_{y1} + \log G(\Gammabf_y) - \frac{1}{2} \log |  \unvecc (\D_p \etabf_3)|   \notag \\
&=& \psi_1 (\etabf_y) + \psi_2(\etabf_y) + \psi_3 (\etabf_y),
\end{eqnarray}
with $G(\Gammabf_y)$ given in (\ref{GdeGy}).

\subsection*{Proof of Theorem \ref{propo1}}\label{proofpropo1}

Since the density of $\X,\Hbf\mid Y$ belongs to the full rank  exponential family (Eqn. (\ref{genExpFam})),  the minimal sufficient reduction for the regression $Y\mid (\X,\Hbf)$ is given by 
\begin{align*}
\R(\X,\Hbf)&=\alphabf_{\mathbf{a}}^T\left(\T(\X,\Hbf)-\E(\T(\X,\Hbf))\right), 
\end{align*}
where $\alphabf_{\mathbf{a}}$ is a basis for $\spc_{\alphabf_{\mathbf{a}}}=\spn\{\etabf_Y - \E({\etabf_Y}), Y\in \mathcal{Y}\}$,  with $\etabf_Y$
given in \eqref{naturalpar} [see  \citet[Thm 1]{BuraDuarteForzani2016}]. Since $\E (\f_Y) = \bm{0}$, applying Eqns. \eqref{naturalpar} and (\ref{varetas}) obtains
\[ \etabf_y -\E (\etabf_y) = \left( \begin{array}{c}
(\f_y^T \otimes \I_p) \vecc (\Deltabf^{-1}\A)\\
(\f_y^T \otimes \I_q)  \vecc(\Lbf_q\taubf - \betabf^T \Deltabf^{-1} \A)\\
0\\
0\\
(\f_y^T \otimes \I_{k_q}) \vecc (\J_q \taubf)\end{array} \right) = 
\left( \begin{array}{c}
   \Deltabf^{-1}\A \f_y \\
 (\Lbf_q\taubf - \betabf^T \Deltabf^{-1} \A) \f_y\\
0\\
0\\
\J_q \taubf \f_y \end{array} \right).
\]
Then, $\spn (\abf)=\spn \left(\etabf_y - \E({\etabf_Y}), y \in  \mathcal{Y} \right)$ with
\[ 
\mathbf{a}= \left( \begin{array}{c} \Deltabf^{-1} \A\\
\Lbf_q \taubf - \betabf^T \Deltabf^{-1} A\\
0\\
0\\
\J_q \taubf \end{array}\right).
\]

\subsection*{Proof of Corollary \ref{CoroPFC}}
If follows from Corollary \ref{optimalSDR} since, in this case,  $\varthetabf_{2,1}=0$ and $\varthetabf_{5,1}=0$.

\subsection*{Proof of Corollary \ref{binonly}}
If follows from Corollary \ref{optimalSDR} since in this case $\varthetabf_{1,1}=0$ and $\varthetabf_{2,1} = \Lbf_q \taubf $.

\subsection*{Proof of Corollary \ref{suboptSDR}}
It suffices to show that $\spn ({\mathbf{b}}) \subset \spn (\alphabf_{\mathbf{c}})$. We can write $\textbf{b}$ as
\[ {\mathbf b} =\left(\begin{array}{cc} \Deltabf \A & \0\\-\betabf^T \Deltabf^{-1} \A &\Lbf_1 \taubf\\
\0 & \J_q\taubf \end{array} \right)\left(\begin{array}{c} \I_r \\ \I_r\end{array}\right) = \mathbf{\widetilde{b}}\left(\begin{array}{c} \I_r \\ \I_r\end{array}\right),
\]
with $\spn (\mathbf{\widetilde{b}}) = \spn (\alphabf_{\mathbf{c}}).$
As a consequence, $\spn ({\mathbf b}) \subset \spn (\alphabf_{\mathbf c})$, and therefore $\R(\X,\Hbf)$ in Eqn. (\ref{suboptSR}) is a sufficient dimension reduction, not necessary minimal. The rest of the corollary immediately follows.

%%%%%%%%%%%%%%%%%%%%%%%%%%%%%%%%%%%%
%%%%%%%%%%%%%%%%%%%%%%%%%%%%%%%%%%%%

\section*{Appendix B: Proof of  Proposition \ref{AsymptDistn}}

We first derive the asymptotic distribution of  $\widehat{\bb}$ in (\ref{mleb}) %in Section \ref{aaa} 
and prove auxiliary lemmas in Section 
%\ref{auxx} 
in order to prove Proposition \ref{AsymptDistn}.  

\subsection*{Asymptotic distribution of $\widehat{\bb}$}\label{aaa}
\begin{proposition}\label{PropNormalAsint}
$\sqrt{n}  \vecc(\widehat{\bb} -\bb) \xrightarrow{\mathcal{D}}  \mathcal{N} (\mathbf{0}, \V_{rcl})$ with
$$\V_{rcl}= \W \M \V \M^T \W^T,
$$
as in equation (\ref{vrcl}), where $\M$, $\W$ and $\V$ are defined in Eqns. (\ref{laM}), (\ref{W}) and (\ref{Vinv}), respectively.
\end{proposition}

%\noindent {\bf Proof of Proposition \ref{PropNormalAsint}.}

\begin{proof}
We rewrite
\begin{align*}
\bb &= \begin{pmatrix} {\Deltabf}^{-1} {\A}\\
\Lbf_q {\taubf} - {\betabf}^T {\Deltabf}^{-1} {\A} \\
\J_q {\taubf} \end{pmatrix} =
  \begin{pmatrix} \unvecc ({\varthetabf}_{1,1}) \\
\unvecc ({\varthetabf}_{2,1})\\
\unvecc ({\varthetabf}_{5,1}) \end{pmatrix},   
\end{align*}
as follows. Let $\widetilde{\bb}=\left(
\varthetabf_{1,1}^T,\varthetabf_{2,1}^T,\varthetabf_{5,1}^T\right)^T$. Then,  $\widetilde{\bb}=\M \varthetabf$, with $\M$ given in (\ref{laM}),
so that
\begin{equation}\label{holas}
\vecc (\bb) = \W \mathbf{\widetilde{b}} = \W \M\varthetabf, 
\end{equation}
with $\W$ defined on (\ref{W}).
Then,  %Eqn. (\ref{holas}) applied to $\widehat{\bb}$ yields 
\begin{equation}\label{bhat}
\vecc (\widehat{\bb}) = \W \M \widehat{\varthetabf}.
\end{equation}
The asymptotic normality of $\widehat{\mathbf{b}}$ follows from the asymptotic normality of $\widehat{\varthetabf}$, which  is derived  in  Lemma \ref{asint}, with
\begin{align*}
\aVar(\sqrt{n} \widehat{\bb}) &= \W \M  \aVar\left(\sqrt{n} \widehat{\varthetabf}\right) \M^T\W^T \\ &= \W \M  \V \M^T\W^T.
\end{align*}
\end{proof}

\begin{lemma}\label{asint}
If $\aVar\left(\sqrt{n} \widehat{\varthetabf}\right) = \V$, then  
$$
\V^{-1} = E \left [ {\mathbf F}_y^T \mathbf{J}{\mathbf F}_y \right], 
$$
as in (\ref{Vinv}), where $\F_y$ is defined in (\ref{naturalpar}) and %$\mathbf{J}$ is the matrix of partial derivatives  
$$
\mathbf{J} = \frac{\partial^2 \psi (\etabf_y)} {\partial \etabf_y \partial \etabf_y^T}
$$
in (\ref{deriva}).
\end{lemma}

%{\bf Proof of Lemma \ref{asint}.}
\begin{proof}
Since $ \widehat{\varthetabf}$ is the maximum likehood estimator, \[\V = \aVar\left(\sqrt{n} \widehat{\varthetabf}\right) = -\left(E\left[\frac{\partial^2 \log f(\X, \Hbf \mid Y=y)}{\partial \thetabf \partial \thetabf^T}\right]\right)^{-1}\]
We plug in $\etabf_y = \F_y \varthetabf$  (from Eqn. (\ref{naturalpar})) in Eqn. (\ref{genExpFam}) to obtain
\begin{align*}
\log f(\X,\Hbf \mid Y=y)  &= \log h(\X,\Hbf) +  {\mathbf T}^T(\X, \Hbf) \etabf_y - \psi (\etabf_y)\\  &= \log h(\X,\Hbf) + 
{\mathbf T}^T(\X, \Hbf) {\mathbf F}_y \varthetabf - \psi ({\mathbf F}_y \varthetabf).
\end{align*}
Then,
\begin{align*}
\frac{\partial \log f(\X, \Hbf \mid Y=y)}{\partial \vecc^T (\thetabf)} & = 
{\mathbf T}^T(\X, \Hbf)  {\mathbf F}_y - \frac{\partial \psi (\etabf_y)}{\partial \etabf_y^T} {\mathbf F}_y,
\\
\frac{\partial^2 \log f(\X, \Hbf \mid Y=y)}{\partial \vecc(\thetabf ) \vecc^T(\partial \thetabf)} &= -
{\mathbf F}_y^T \frac{\partial^2 \psi (\etabf_y)}{\partial \etabf_y \partial \etabf_y^T}{\mathbf F}_y= -
{\mathbf F}_y^T {\mathbf J}{\mathbf F}_y.
\end{align*}
Therefore $\V^{-1} = \E \left [ \F_y^T \J\F_y \right]$ from which Proposition \ref{asint} follows.

In order to compute $\mathbf{J}$, the first and second derivatives of $\psi (\etabf_y)$ with respect to $\etabf_y$ are required. The computation is carried out in Section \ref{supplement} (Supplementary Material). 
\end{proof}

\subsection*{Auxiliary lemmas for Proposition \ref{AsymptDistn}} \label{auxx}

\begin{lemma}\label{lemmas}
Let $\widehat{\Hbf}= \widehat{\Uu}_1\widehat{{\mathbf{K}}}_1 \widehat{\R}_1^T \R_1{\mathbf{K}}^{-1}$. Then,
\begin{equation*}
\sqrt{n} \vecc (\widehat{\Hbf}-\Uu_1 ) \rightarrow \mathcal{N}( 0, ({\mathbf{K}}^{-1}\R_1^T \otimes \I_m){\mathbf{V}}_{rlc}(\R_1{\mathbf{K}}^{-1} \otimes \I_m)).
\end{equation*}
where $\mathbf{V}_{rlc}$ is defined in Eqn. (\ref{vrcl}), $\widehat{\Uu}_1, \widehat{\mathbf{K}}_1$ and $\widehat{\R}_1$
in Eqn. (\ref{bb}), and ${\Uu}_1, {\mathbf{K}}$ and $\R_1$ in Eqn. (\ref{nec}). 
\end{lemma}

\begin{proof}
By Eqn. (\ref{nec}), $\bb = {\Uu}_1\mathbf{K} {\R}_1^T$ and by Eqn. (\ref{bb}), $\widehat{\bb}  = \widehat{\Uu}_1 \widehat{{\mathbf{K}}}_1 \widehat{\R}_1^T + \widehat{\Uu}_0 \widehat{{\mathbf{K}}}_0
  \widehat{\R}_0^T$. Then, 
  \begin{align*}
 \widehat{\Hbf} -\Uu_1 &= \widehat{\Uu}_1\widehat{{\mathbf{K}}}_1 \widehat{\R}_1^T \R_1{\mathbf{K}}^{-1} -\Uu_1  \\
 &=  \widehat{\bb}\R_1{\mathbf{K}}^{-1}  - \widehat{\Uu}_0 \widehat{{\mathbf{K}}}_0 \widehat{\R}_0^T \R_1{\mathbf{K}}^{-1} -\Uu_1 \\
 &= (\widehat{\bb} - \bb)\R_1{\mathbf{K}}^{-1} - \widehat{\Uu}_0 \widehat{{\mathbf{K}}}_0 \widehat{\R}_0^T \R_1{\mathbf{K}}^{-1}.
 \end{align*}
Thus,
\begin{align}\label{esta}
\sqrt{n}\vecc(\widehat{\Hbf} -\Uu_1) &= \sqrt{n}\vecc((\widehat{\bb} - \bb)\R_1{\mathbf{K}}^{-1} - \widehat{\Uu}_0 \widehat{{\mathbf{K}}}_0 \widehat{\R}_0^T \R_1{\mathbf{K}}^{-1}) \notag \\ &=  \sqrt{n}(\mathbf{K}^{-1}\R_1^T \otimes \I_m) \vecc (\widehat{\bb}-\bb)  - \vecc(\widehat{\Uu}_0 \widehat{{\mathbf{K}}}_0 \widehat{\R}_0^T \R_1{\mathbf{K}}^{-1}).
\end{align}
From Proposition \ref{PropNormalAsint} we have 
\[
\sqrt{n}  \vecc(\widehat{\bb} -\bb) \xrightarrow{\mathcal{D}}  \mathcal{N} (\mathbf{0}, \V_{rcl}),
\]
so that
\begin{equation}\label{yesta}
\sqrt{n}(\mathbf{K}^{-1}\R_1^T \otimes \I_m) \vecc (\widehat{\bb}-\bb) \xrightarrow{\mathcal{D}}  \mathcal{N} (\mathbf{0}, \Sigmabf_\Uu),
\end{equation}
with $\Sigmabf_\Uu=({\mathbf{K}}^{-1}\R_1^T \otimes \I_k) \V_{rcl}({\mathbf{K}}^{-1}\R_1^T \otimes \I_k)^T = ({\mathbf{K}}^{-1}\R_1^T \otimes \I_k)\V_{rcl}(\R_1{\mathbf{K}}^{-1} \otimes \I_k)$.

Since $\sqrt{n} \left(\widehat{\Uu}_0 \widehat{{\mathbf{K}}}_0  \widehat{\R}_0^T \right)=O_p(1)$  and $\Pb_{\R_1}=\left(\Pb_{\widehat{\R}_1}+O_p(n^{-1/2})\right)$, we get
\begin{align*}
  \sqrt{n} (\widehat{\Uu}_0 \widehat{{\mathbf{K}}}_0  \widehat{\R}_0^T\R_1 {\mathbf{K}}^{-1} )&= \sqrt{n} (\widehat{\Uu}_0 \widehat{{\mathbf{K}}}_0  \widehat{\R}_0^T ) \Pb_{\R_1} \R_1 {\mathbf{K}}^{-1}\\ 
 &= \sqrt{n} \left(\widehat{\Uu}_0 \widehat{{\mathbf{K}}}_0  \widehat{\R}_0^T \right) \left(\Pb_{\widehat{\R}_1}+O_p(n^{-1/2}) \right)\R_1 {\mathbf{K}}^{-1} \\ 
   &= \sqrt{n} \left(\widehat{\Uu}_0 \widehat{{\mathbf{K}}}_0  \widehat{\R}_0^T \right) O_p(n^{-1/2})\R_1 {\mathbf{K}}^{-1} \\ &= O_p (n^{-1/2}),
\end{align*}
where we use  $\widehat{\R}_0^T \widehat{\R}_1=\mathbf{0}$. As a consequence, $\sqrt{n}
\vecc(\widehat{\Uu}_0 \widehat{\K}_0  \widehat{\R}_0^T\R_1 {\mathbf{K}}^{-1} ) \rightarrow
\mathbf{0}$ in probability, which, together with (\ref{yesta}) in (\ref{esta}), obtain the result. 
\end{proof}

\begin{lemma}\label{auxi}
Let $\Gammabf$ be a matrix of dimension $p\times d$ of full rank $d$ with $d \leq p$,  and let $\Pb_{\Gammabfs}$ be the orthogonal projection onto the column space  of $\Gammabf$ and $\Q_{\Gammabfs}=\I-\Pb_{\Gammabfs}$. Also, let $\K_{pm} \in {\mathbb R}^{pm \times pm}$ be the unique  matrix such that, for any symmetric $p\times m$ matrix $\A$,  $\vecc(\A^T) = \K_{pm} \vecc(\A)$.
Then, 
\begin{equation}\label{uno}
\frac{\partial{\Pb_{\Gammabfs}}}{\partial \vecc^{T}(\Gammabf)} =
(\I_{p^2}+ \K_{pp})(\Gammabf(\Gammabf^{T} \Gammabf)^{-1} \otimes
\Q_{\Gammabfs}).
\end{equation}
%\textcolor{red}{define $\K_{pp}$} 
\end{lemma}

\begin{proof}
We will use the following two identities.
\begin{itemize}
\item[(i)] Let $\X$ be a matrix and $\F(\X): m\times p$ and $\Gg(\X): p\times q$ differentiable matrix valued functions of $\X$. Then,
\begin{equation}\label{derivadaproducto}
 \frac{\partial \vecc[\F(\X)\Gg(\X)]}{\partial \vecc^T(\X)} = (\Gg^T \otimes \I_m)\frac{\partial \vecc[\F(\X)]}{\partial \vecc^T(\X)} + (\I_q \otimes \F) \frac{\partial \vecc[\Gg(\X)]}{\partial \vecc^T(\X)}.
\end{equation}
\item[(ii)] Let $\F(\X)=\X^{T}$ and $\Gg(\X)=\X$ with $\X: p\times q$. By (\ref{derivadaproducto}),  \begin{eqnarray}\label{deriv1}
\frac{\partial \vecc(\X^T\X)}{\partial^T\vecc(\X)}&=& (\I_{q^2} + \K_{qq}) (\I_{q}\otimes \X^T). \\ \label{deriv2} \frac{\partial
\vecc(\X^T\X)^{-1}}{\partial^T\vecc(\X)}&=& - ((\X^T\X)^{-1}\otimes
(\X^T\X)^{-1})) \frac{\partial \vecc(\X^T\X)}{\partial^T\vecc(\X)}.\notag
\end{eqnarray}
\end{itemize}

Applying (\ref{derivadaproducto}) yields   
\begin{align*}
\frac{\partial \vecc\Pb_{\Gammabfs}}{\partial \vecc^T (\Gammabf)} & =  \frac{\partial \vecc(\Gammabf (\Gammabf^T\Gammabf)^{-1}\Gammabf^T)}{\partial \vecc^T(\Gammabf)}\\ 
 &  = (\Gammabf(\Gammabf^T\Gammabf)^{-1} \otimes \I_p ) \frac{\partial \vecc(\Gammabf)}{\partial \vecc^T (\Gammabf)} + (\I_p\otimes \Gammabf) \frac{\partial \vecc((\Gammabf^T\Gammabf)^{-1}\Gammabf^T)}{\partial \vecc^T (\Gammabf)}\\ 
 &= (\Gammabf(\Gammabf^T\Gammabf)^{-1} \otimes \I_p ) + (\I_p\otimes \Gammabf) \frac{\partial \vecc((\Gammabf^T\Gammabf)^{-1}\Gammabf^T)}{\partial \vecc^T (\Gammabf)}.
\end{align*}
Let 
\[
\Hbf=\frac{\partial \vecc((\Gammabf^T\Gammabf)^{-1}\Gammabf^T)}{\partial \vecc^T
(\Gammabf)}.
\]
Using (\ref{derivadaproducto}),
(\ref{deriv1}) and (\ref{deriv2}), we get
\begin{eqnarray*}
\Hbf&=& (\Gammabf\otimes \I_d) \frac{\partial
\vecc(\Gammabf^T\Gammabf)^{-1}}{\partial \vecc^T (\Gammabf)}+(\I_p
\otimes (\Gammabf^T\Gammabf)^{-1})\K_{pd}\\ &=&
-(\Gammabf\otimes \I_d) ((\Gammabf^T\Gammabf)^{-1}\otimes
(\Gammabf^T\Gammabf)^{-1}))(\I_{d^2} + \K_{dd}) (\I_{d}\otimes
\Gamma^T)+(\I_p \otimes (\Gammabf^T\Gammabf)^{-1})\K_{pd}.
\end{eqnarray*}
Then,
\begin{eqnarray*} 
\frac{\partial \vecc\Pb_{\Gammabfs}}{\partial \vecc^T (\Gammabf)} &=&
(\Gammabf(\Gammabf^T\Gammabf)^{-1} \otimes \I_p ) + (\I_p\otimes
\Gammabf) \times  \\
&\quad& \left[-(\Gammabf\otimes \I_d) ((\Gammabf^T\Gammabf)^{-1}\otimes
(\Gammabf^T\Gammabf)^{-1})(\I_{d^2} + \K_{dd}) (\I_{d}\otimes
\Gammabf^T)+(\I_p \otimes
(\Gammabf^T\Gammabf)^{-1})\K_{pd}\right] \\
&=&(\Gammabf(\Gammabf^T\Gammabf)^{-1} \otimes \I_p)+ (\I_p
\otimes\Gammabf(\Gammabf\Gammabf)^{-1}){\mathbf{K}}_{pd}-
(\Gammabf(\Gammabf^T\Gammabf)^{-1}\otimes
\Gammabf(\Gammabf^T\Gammabf)^{-1}\Gammabf^T)\\  & &-
(\Gammabf(\Gammabf^T\Gammabf)^{-1}\otimes
\Gammabf(\Gammabf^T\Gammabf)^{-1}){\mathbf{K}}_{dd}(\I_d \otimes \Gammabf^T)\\
 &=& (\I_{p^2}+ {\mathbf{K}}_{pp})
(\Gammabf(\Gammabf^T\Gammabf)^{-1}\otimes
\I_p)-(\Gammabf(\Gammabf^T\Gammabf)^{-1}\otimes
\Gammabf(\Gammabf^T\Gammabf)^{-1}\Gammabf^T)\\ &
&-(\Gammabf(\Gammabf^T\Gammabf)^{-1} \otimes
\Gammabf(\Gammabf^T\Gammabf)^{-1})(\Gammabf^T\otimes \I_d){\mathbf{K}}_{pd}\\
 &=& (\I_{p^2}+ {\mathbf{K}}_{pp})
(\Gammabf(\Gammabf^T\Gammabf)^{-1}\otimes \I_p)- (\I_{p^2}+
{\mathbf{K}}_{pp})(\Gammabf(\Gammabf^T\Gammabf)^{-1}\otimes \Pb_{\Gammabfs})\\
 &=& (\I_{p^2}+ {\mathbf{K}}_{pp}) (\Gammabf(\Gammabf^T\Gammabf)^{-1}
\otimes \I_p-\Pb_{\Gammabfs})\\ 
&=& (\I_{p^2}+ {\mathbf{K}}_{pp}) (\Gammabf(\Gammabf^T\Gammabf)^{-1}
\otimes \Q_{\Gammabfs}).
\end{eqnarray*}
\end{proof}

\begin{lemma}\label{lema: Proyeccion}
Suppose the two matrices  $\widehat{\Gammabf}$ and $\Gammabf$ are of order $p\times d$ with $d \leq p$ with  $\Gammabf$ of full rank  $d$. Assume that $\widehat{\Gammabf}$ is asymptotically normal with
\begin{equation*}
  \sqrt{n} \vecc (\widehat{\Gammabf}- \Gammabf)  \xrightarrow{\mathcal{D}} \mathcal{N}( 0,   \V).
\end{equation*}
Then, $\sqrt{n} \vecc (\Pb_{\widehat{\Gammabfs}} - \Pb_{\Gammabfs})$ is asymptotically normal with mean $\0$ and variance
\begin{equation*}
  (\I_{p^2} + \K_{pp})(\Gammabf (\Gammabf^T\Gammabf)^{-1} \otimes \Q_{\Gammabfs})\V ((\Gammabf^T\Gammabf)^{-1}\Gammabf^T\otimes \Q_{\Gammabfs})(\I_{p^2} + {\mathbf{K}}_{pp}).
\end{equation*}
\end{lemma}
\begin{proof}
Let $\Pb_{\Gammabfs}=\Gammabf (\Gammabf^T \Gammabf)^{-1} \Gammabf^T$ be the orthogonal projection onto the column space of $\Gammabf$ and let $g$ be a function defined in the subspace of the $p\times d$ matrices  of full rank $d$ such that $g(\Gammabf)=\Gammabf(\Gammabf^T\Gammabf)^{-1}\Gammabf^T =\Pb_{\Gammabfs}$.  
From Lemma \ref{auxi} we have that
\[
\nabla g (\Gammabf)= \frac{\partial{\Pb_{\Gammabfs}}}{\partial \vecc^{T}(\Gammabf)} =
(\I_{p^2}+ \K_{pp})(\Gammabf(\Gammabf^{T} \Gammabf)^{-1} \otimes
\Q_{\Gammabfs}).
\]
By the Delta method,
\begin{equation*}
\sqrt{n}  \left( g(\widehat{\Gammabf}) - g(\Gammabf) \right) \rightarrow
\mathcal{N}\left( 0,\nabla g (\Gammabf) \V\nabla^T g (\Gammabf)\right),
\end{equation*}
which completes the proof.
\end{proof}

%%%%%%%%%%%%%%%%%%%%%%%%
%%%%%%%%%%%%%%%%%%%%%%%%

\subsection*{Proof of  Proposition \ref{AsymptDistn}}

%\begin{proof}
From (\ref{bound}), $\widehat{\alphabf}_{\mathbf b}=\widehat{\Uu}_1$ and therefore $\alphabf_{\mathbf b}=\Uu_1 $ and $\spn (\widehat{\Uu}_1) = \spn (\widehat{\Hbf})$ with 
$\widehat{\Hbf}= \widehat{\Uu}_1\widehat{{\mathbf{K}}}_1 \widehat{\R}_1^T \R_1{\mathbf{K}}^{-1}$ defined in Lemma \ref{lemmas}, which also provides the asymptotic distribution of $\widehat{\Hbf}$. Applying Lemma \ref{lema: Proyeccion} with $\widehat{\Gammabf} = \widehat{\Hbf}$ and $\Gammabf=\Hbf=\Uu_1$ we obtain the asymptotic distribution with asymptotic variance
\begin{eqnarray*}
(\I_{p^2} + \K_{pp})  (\Uu_1 {\mathbf{K}}^{-1}\R_1^T \otimes \Q_{\Uu_1}){\mathbf{V}}_{rlc}(\R_1{\mathbf{K}}^{-1}\Uu_1^T \otimes \Q_{\Uu_1}))  (\I_{p^2} + {\mathbf{K}}_{pp}) 
\end{eqnarray*}
since $\Uu_1^T \Uu_1=\I_d$. 
By (\ref{bbb}), $
\bb = \Uu_1  \mathbf{K} \R^{T}_1
$, therefore $\bb^-=\R_1 \mathbf{K}^{-1} \Uu_1^T$ and the result follows.


\vskip 0.2in


\begin{thebibliography}{51}
\providecommand{\natexlab}[1]{#1}
\providecommand{\url}[1]{\texttt{#1}}
\expandafter\ifx\csname urlstyle\endcsname\relax
  \providecommand{\doi}[1]{doi: #1}\else
  \providecommand{\doi}{doi: \begingroup \urlstyle{rm}\Url}\fi

\bibitem[Aitchison and Aitken(1976)]{AitchisonAitken1976}
J.~Aitchison and C.~G.~G. Aitken.
\newblock Multivariate binary discrimination by the kernel method.
\newblock \emph{Biometrika}, 63\penalty0 (3):\penalty0 413--420, 1976.
\newblock ISSN 00063444.
\newblock URL \url{http://www.jstor.org/stable/2335719}.

\bibitem[Anderson(1972)]{Anderson1972}
J.~A. Anderson.
\newblock Separate sample logistic discrimination.
\newblock \emph{Biometrika}, 59\penalty0 (1):\penalty0 19--35, 1972.
\newblock ISSN 00063444.
\newblock URL \url{http://www.jstor.org/stable/2334611}.

\bibitem[Anderson(1975)]{Anderson1975}
J.~A. Anderson.
\newblock Quadratic logistic discrimination.
\newblock \emph{Biometrika}, 62\penalty0 (1):\penalty0 149--154, 1975.
\newblock ISSN 00063444.
\newblock URL \url{http://www.jstor.org/stable/2334497}.

\bibitem[Bach et~al.(2012)Bach, Jenatton, Mairal, and Obozinski]{bach2012}
Francis Bach, Rodolphe Jenatton, Julien Mairal, and Guillaume Obozinski.
\newblock Structured sparsity through convex optimization.
\newblock \emph{Statist. Sci.}, 27\penalty0 (4):\penalty0 450--468, 11 2012.
\newblock \doi{10.1214/12-STS394}.
\newblock URL \url{https://doi.org/10.1214/12-STS394}.

\bibitem[Bura and Yang(2011)]{BuraYang2011}
E.~Bura and J.~Yang.
\newblock Dimension estimation in sufficient dimension reduction: A unifying
  approach.
\newblock \emph{Journal of Multivariate Analysis}, 102\penalty0 (1):\penalty0
  130 -- 142, 2011.
\newblock ISSN 0047-259X.
\newblock \doi{https://doi.org/10.1016/j.jmva.2010.08.007}.
\newblock URL
  \url{http://www.sciencedirect.com/science/article/pii/S0047259X10001661}.

\bibitem[Bura et~al.(2016)Bura, Duarte, and Forzani]{BuraDuarteForzani2016}
E.~Bura, S.~Duarte, and L.~Forzani.
\newblock Sufficient reductions in regressions with exponential family inverse
  predictors.
\newblock \emph{Journal of the American Statistical Association}, 111\penalty0
  (515):\penalty0 1313--1329, 2016.

\bibitem[Bura and Forzani(2015)]{BuraForzani2015}
Efstathia Bura and Liliana Forzani.
\newblock Sufficient reductions in regressions with elliptically contoured
  inverse predictors.
\newblock \emph{Journal of the American Statistical Association}, 110\penalty0
  (509):\penalty0 420--434, 2015.
\newblock \doi{10.1080/01621459.2014.914440}.
\newblock URL \url{https://doi.org/10.1080/01621459.2014.914440}.

\bibitem[Buuren(2018)]{Buuren2018}
Stef~van Buuren.
\newblock \emph{Flexible imputation of missing data}.
\newblock CRC Press, 2nd edition, 2018.

\bibitem[Camiz and Gomes(2013)]{Camiz13}
S.~Camiz and G.C. Gomes.
\newblock Joint correspondence analysis versus multiple correspondence
  analysis: a solution to an undetected problem.
\newblock In \emph{Classification and data mining}, Stud. Classification Data
  Anal. Knowledge Organ., pages 11--18. Springer, Heidelberg, 2013.

\bibitem[Chavent et~al.(2012)Chavent, Kuentz-Simonet, Liquet, and
  Saracco]{Chaventetal2012}
M.~Chavent, V.~Kuentz-Simonet, B.~Liquet, and J.~Saracco.
\newblock Orthogonal rotation in pcamix.
\newblock \emph{Advances in Data Analysis and Classification}, 6:\penalty0
  131--146, 2012.

\bibitem[Chavent et~al.(2014)Chavent, Kuentz-Simonet, Labenne, and
  Saracco]{Chaventetal2014}
Marie Chavent, Vanessa Kuentz-Simonet, Amaury Labenne, and Jérôme Saracco.
\newblock Multivariate analysis of mixed data: The r package pcamixdata, 2014.

\bibitem[Chen et~al.(2014)Chen, Witten, and Shojaie]{Chenetal2015}
Shizhe Chen, Daniela~M. Witten, and Ali Shojaie.
\newblock {Selection and estimation for mixed graphical models}.
\newblock \emph{Biometrika}, 102\penalty0 (1):\penalty0 47--64, 12 2014.
\newblock ISSN 0006-3444.
\newblock \doi{10.1093/biomet/asu051}.
\newblock URL \url{https://doi.org/10.1093/biomet/asu051}.

\bibitem[Cheng et~al.(2014)Cheng, Levina, Wang, and Zhu]{Chengetal14}
Jie Cheng, Elizaveta Levina, Pei Wang, and Ji~Zhu.
\newblock A sparse {I}sing model with covariates.
\newblock \emph{Biometrics}, 70\penalty0 (4):\penalty0 943--953, 2014.
\newblock ISSN 0006-341X.
\newblock \doi{10.1111/biom.12202}.
\newblock URL \url{https://doi.org/10.1111/biom.12202}.

\bibitem[Cheng et~al.(2017)Cheng, Li, Levina, and Zhu]{Chengetal2017}
Jie Cheng, Tianxi Li, Elizaveta Levina, and Ji~Zhu.
\newblock High-dimensional mixed graphical models.
\newblock \emph{Journal of Computational and Graphical Statistics}, 26\penalty0
  (2):\penalty0 367--378, 2017.
\newblock \doi{10.1080/10618600.2016.1237362}.
\newblock URL \url{https://doi.org/10.1080/10618600.2016.1237362}.

\bibitem[Cook(2007)]{Cook2007}
R.D. Cook.
\newblock Fisher lecture: Dimension reduction in regression (with discussion).
\newblock \emph{Statistical Science}, 22:\penalty0 1--26, 2007.

\bibitem[Cook and Forzani(2008)]{CookForzani2008}
R.D. Cook and L.~Forzani.
\newblock Principal fitted components for dimension reduction in regression.
\newblock \emph{Statistical Science}, 23:\penalty0 485--501, 2008.

\bibitem[Cook and Weisberg(1991)]{CookWeisberg1991}
R.D. Cook and S.~Weisberg.
\newblock Discussion of sliced inverse regression for dimension reduction.
\newblock \emph{Journal of the American Statistical Association}, 86:\penalty0
  328--332, 1991.

\bibitem[Dai(2012)]{Dai2012}
Bin Dai.
\newblock Multivariate bernoulli distribution models.
\newblock Technical report, Dept. Statistics, Univ. Wisconsin, Madison, WI
  53706, July 2012.

\bibitem[Dai et~al.(2013)Dai, Ding, and Wahba]{Daietal2013}
Bin Dai, Shilin Ding, and Grace Wahba.
\newblock Multivariate bernoulli distribution.
\newblock \emph{Bernoulli}, 19\penalty0 (4):\penalty0 1465--1483, 09 2013.
\newblock \doi{10.3150/12-BEJSP10}.
\newblock URL \url{https://doi.org/10.3150/12-BEJSP10}.

\bibitem[Day and Kerridge(1967)]{DayKerridge1967}
N.~E. Day and D.~F. Kerridge.
\newblock A general maximum likelihood discriminant.
\newblock \emph{Biometrics}, 23\penalty0 (2):\penalty0 313--323, 1967.
\newblock ISSN 0006341X, 15410420.
\newblock URL \url{http://www.jstor.org/stable/2528164}.

\bibitem[Filmer and Scott(2012)]{Filmer12}
D.~Filmer and K.~Scott.
\newblock {Assessing Asset Indices}.
\newblock \emph{Demography}, 49:\penalty0 359--392, 2012.

\bibitem[Fitzmaurice and Laird(1997)]{FitzmauriceLaird1997}
Garrett~M. Fitzmaurice and Nan~M. Laird.
\newblock Regression models for mixed discrete and continuous responses with
  potentially missing values.
\newblock \emph{Biometrics}, 53\penalty0 (1):\penalty0 110--122, 1997.
\newblock ISSN 0006341X, 15410420.
\newblock URL \url{http://www.jstor.org/stable/2533101}.

\bibitem[Forzani et~al.(2018)Forzani, Garc\'ia-Arancibia, Llop, and
  Tomassi]{Forzanietal2018}
L.~Forzani, R.~Garc\'ia-Arancibia, P.~Llop, and D.~Tomassi.
\newblock Supervised dimension reduction for ordinal predictors.
\newblock \emph{Computational Statistics and Data Analysis}, 125, 2018.

\bibitem[Greenacre and Blasius(2006)]{Greenacre06}
M.~Greenacre and J.~Blasius, editors.
\newblock \emph{Multiple correspondence analysis and related methods}.
\newblock Statistics in the Social and Behavioral Sciences Series. Chapman \&
  Hall/CRC, Boca Raton, FL, 2006.
\newblock ISBN 978-1-58488-628-0; 1-58488-628-5.
\newblock \doi{10.1201/9781420011319}.

\bibitem[Ising(1925)]{Ising1925}
Ernst Ising.
\newblock Beitrag zur theorie des ferromagnetismus.
\newblock \emph{Zeitschrift f{\"u}r Physik}, 31\penalty0 (1):\penalty0
  253--258, Feb 1925.
\newblock ISSN 0044-3328.
\newblock \doi{10.1007/BF02980577}.
\newblock URL \url{https://doi.org/10.1007/BF02980577}.

\bibitem[Javaras and van Dyk(2003)]{JavarasVanDyk2003}
Kristin~N. Javaras and David~A. van Dyk.
\newblock Multiple imputation for incomplete data with semicontinuous
  variables.
\newblock \emph{Journal of the American Statistical Association}, 98\penalty0
  (463):\penalty0 703--715, 2003.
\newblock ISSN 01621459.
\newblock URL \url{http://www.jstor.org/stable/30045298}.

\bibitem[Kolenikov and Angeles(2009)]{Kolenikov09}
S.~Kolenikov and G.~Angeles.
\newblock Socioeconomic status measurement with discrete proxy variables: Is
  principal component analysis a reliable answer?
\newblock \emph{The Review of Income and Wealth}, 55\penalty0 (1):\penalty0
  128--165, 2009.

\bibitem[Krzanowski(1993)]{Krzanowski1993}
W.~J. Krzanowski.
\newblock The location model for mixtures of categorical and continuous
  variables.
\newblock \emph{Journal of Classification}, 10\penalty0 (1):\penalty0 25--49,
  Jan 1993.
\newblock ISSN 1432-1343.
\newblock \doi{10.1007/BF02638452}.
\newblock URL \url{https://doi.org/10.1007/BF02638452}.

\bibitem[Krzanowski(1975)]{Krzanowski75}
W.J. Krzanowski.
\newblock Discrimination and classification using both binary and continuous
  variables.
\newblock \emph{Journal of the American Statistical Association}, 70\penalty0
  (352):\penalty0 782--790, 1975.

\bibitem[Lauritzen(1996)]{Lauritzen1996}
S.~L. Lauritzen.
\newblock \emph{Graphical Models}.
\newblock Oxford University Press, Oxford, 1996.

\bibitem[Lauritzen and Wermuth(1989)]{LauritzenWermuth1989}
S.~L. Lauritzen and N.~Wermuth.
\newblock Graphical models for associations between variables, some of which
  are qualitative and some quantitative.
\newblock \emph{Ann. Statist.}, 17\penalty0 (1):\penalty0 31--57, 03 1989.
\newblock \doi{10.1214/aos/1176347003}.
\newblock URL \url{https://doi.org/10.1214/aos/1176347003}.

\bibitem[Lee and Hastie(2015)]{LeeHastie2015}
Jason~D. Lee and Trevor~J. Hastie.
\newblock Learning the structure of mixed graphical models.
\newblock \emph{Journal of Computational and Graphical Statistics}, 24\penalty0
  (1):\penalty0 230--253, 2015.
\newblock \doi{10.1080/10618600.2014.900500}.
\newblock URL \url{https://doi.org/10.1080/10618600.2014.900500}.
\newblock PMID: 26085782.

\bibitem[Li and Wang(2007)]{LiWang2007}
B.~Li and S.~Wang.
\newblock {On directional regression for dimension reduction}.
\newblock \emph{Journal of the American Statistical Association}, 102\penalty0
  (479):\penalty0 997--1008, 2007.

\bibitem[Li(1991)]{Li1991}
K.~C. Li.
\newblock Sliced inverse regression for dimension reduction (with discussion).
\newblock \emph{Journal of the American Statistical Association}, 86:\penalty0
  316--342, 1991.

\bibitem[Liu and Ye(2010)]{Liuetal2010}
Jun Liu and Jieping Ye.
\newblock Fast overlapping group lasso.
\newblock \emph{arXiv:1009.0306v1}, 2010.

\bibitem[Mazziotta and Pareto(2019)]{Mazziotta2019}
Matteo Mazziotta and Adriano Pareto.
\newblock Use and misuse of pca for measuring well-being.
\newblock \emph{Social Indicators Research}, 142\penalty0 (2):\penalty0
  451--476, Apr 2019.
\newblock ISSN 1573-0921.
\newblock \doi{10.1007/s11205-018-1933-0}.

\bibitem[Merola and Baulch(2014)]{Merola14}
G.~Merola and B.~Baulch.
\newblock Using sparse categorical principal components to estimate asset
  indices new methods with an application to rural south east asia.
\newblock 2014.

\bibitem[Morris(2006)]{Morris2006}
Carl~N. Morris.
\newblock \emph{Natural Exponential Families}.
\newblock American Cancer Society, 2006.
\newblock ISBN 9780471667193.
\newblock \doi{10.1002/0471667196.ess1759.pub2}.
\newblock URL
  \url{https://onlinelibrary.wiley.com/doi/abs/10.1002/0471667196.ess1759.pub2}.

\bibitem[Olkin and Tate(1961)]{OlkinTate1961}
I.~Olkin and R.~F. Tate.
\newblock Multivariate correlation models with mixed discrete and continuous
  variables.
\newblock \emph{Ann. Math. Statist.}, 32\penalty0 (2):\penalty0 448--465, 06
  1961.
\newblock \doi{10.1214/aoms/1177705052}.
\newblock URL \url{https://doi.org/10.1214/aoms/1177705052}.

\bibitem[Pepe(2003)]{Pepe2003}
M.S. Pepe.
\newblock \emph{The Statistical Evaluation of Medical Tests for Classification
  and Prediction}.
\newblock Oxford University Press, New York, 2003.

\bibitem[Vlachonikolis and Marriott(1982)]{VlachoMarriott1982}
I.~G. Vlachonikolis and F.~H.~C. Marriott.
\newblock Discrimination with mixed binary and continuous data.
\newblock \emph{Journal of the Royal Statistical Society. Series C (Applied
  Statistics)}, 31\penalty0 (1):\penalty0 23--31, 1982.
\newblock ISSN 00359254, 14679876.
\newblock URL \url{http://www.jstor.org/stable/2347071}.

\bibitem[Vyas and Kumaranayake(2006)]{Vyas06}
S.~Vyas and L.~Kumaranayake.
\newblock Constructing socio-economic status indices: How to use principal
  components analysis.
\newblock \emph{Health Policy and Planning}, 21\penalty0 (6):\penalty0
  459--468, 2006.

\bibitem[Wainwright and Jordan(2008)]{WainwrightJordan2008}
Martin~J. Wainwright and Michael~I. Jordan.
\newblock Graphical models, exponential families, and variational inference.
\newblock \emph{Foundations and Trends® in Machine Learning}, 1\penalty0
  (1–2):\penalty0 1--305, 2008.
\newblock ISSN 1935-8237.
\newblock \doi{10.1561/2200000001}.
\newblock URL \url{http://dx.doi.org/10.1561/2200000001}.

\bibitem[Whittaker(2009)]{Whittaker2009}
Joe Whittaker.
\newblock \emph{Graphical Models in Applied Multivariate Statistics}.
\newblock Wiley Publishing, 2009.
\newblock ISBN 0470743662, 9780470743669.

\bibitem[Yang et~al.(2014{\natexlab{a}})Yang, Baker, Ravikumar, Allen, and
  Liu]{Yangetal2014}
Eunho Yang, Yulia Baker, Pradeep Ravikumar, Genevera Allen, and Zhandong Liu.
\newblock {Mixed Graphical Models via Exponential Families}.
\newblock In Samuel Kaski and Jukka Corander, editors, \emph{Proceedings of the
  Seventeenth International Conference on Artificial Intelligence and
  Statistics}, volume~33 of \emph{Proceedings of Machine Learning Research},
  pages 1042--1050, Reykjavik, Iceland, 22--25 Apr 2014{\natexlab{a}}. PMLR.
\newblock URL \url{http://proceedings.mlr.press/v33/yang14a.html}.

\bibitem[Yang et~al.(2014{\natexlab{b}})Yang, Ravikumar, Allen, Baker, Wan, and
  Liu]{Yangetal2014b}
Eunho Yang, Pradeep Ravikumar, Genevera~I. Allen, Yulia Baker, Ying-Wooi Wan,
  and Zhandong Liu.
\newblock A general framework for mixed graphical models, 2014{\natexlab{b}}.

\bibitem[Yang et~al.(2015)Yang, Ravikumar, Allen, Zh, and ong
  Liu]{Yangetal2015}
Eunho Yang, Pradeep Ravikumar, Genevera~I. Allen, Zh, and ong Liu.
\newblock Graphical models via univariate exponential family distributions.
\newblock \emph{Journal of Machine Learning Research}, 16\penalty0
  (115):\penalty0 3813--3847, 2015.
\newblock URL \url{http://jmlr.org/papers/v16/yang15a.html}.

\bibitem[Ye and Lim(2016)]{YeLim2016}
Ke~Ye and Lek-Heng Lim.
\newblock Schubert varieties and distances between subspaces of different
  dimensions.
\newblock \emph{SIAM Journal on Matrix Analysis and Applications}, 37\penalty0
  (3):\penalty0 1176--1197, 2016.
\newblock \doi{10.1137/15M1054201}.
\newblock URL \url{https://doi.org/10.1137/15M1054201}.

\bibitem[Yuan and Lin(2006)]{yuan2006}
Ming Yuan and Yi~Lin.
\newblock Model selection and estimation in regression with grouped variables.
\newblock \emph{Journal of the Royal Statistical Society. Series B, statistical
  methodology}, 2006.
\newblock ISSN 1369-7412.

\bibitem[Yuan and Lin(2007)]{YuanLin2007}
Ming Yuan and Yi~Lin.
\newblock Model selection and estimation in the gaussian graphical model.
\newblock \emph{Biometrika}, 94\penalty0 (1):\penalty0 19--35, 2007.
\newblock ISSN 00063444.
\newblock URL \url{http://www.jstor.org/stable/20441351}.

\bibitem[Zhu et~al.(2011)Zhu, Lin, Shyu, and Chen]{Zhu11}
Q~Zhu, L.~Lin, M.-L. Shyu, and S.-C. Chen.
\newblock Effective supervised discretization for classification based on
  correlation maximization.
\newblock \emph{IEEE International Conference on Information Reuse \&
  Integration}, pages 390--395, 2011.

\end{thebibliography}
\end{document}